\documentclass[reqno,a4paper,11pt]{amsart}

\usepackage[utf8]{inputenc}
\usepackage{graphicx}
\usepackage{subfigure}
\usepackage{amsmath, amssymb, amsthm, mathtools,times}
\usepackage{amsaddr}
\usepackage{enumitem}
\usepackage{cite}
\usepackage{tikz}
\usepackage{mathrsfs}
\usepackage{cleveref}
\usepackage[english]{babel}
\usepackage[T1]{fontenc}
\usepackage{fourier}
\usepackage{fullpage}

\DeclareMathOperator{\cov}{Cov}
\newcommand{\N}{\mathbb{N}} %natural numbers
\newcommand{\R}{\mathbb{R}} %real numbers
\newcommand{\whp}{whp}%with high probability}
\newcommand{\prob}[1]{\mathbb{P}\left[#1\right]} %Probability
\newcommand{\condprob}[2]{\mathbb{P}\left[#1 \;\middle|\; #2\right]}
\newcommand{\variance}[1]{\mathbb{V}\left[#1\right]}
\newcommand{\covariance}[2]{\cov\left[#1, #2\right]}
\newcommand{\expec}[1]{\mathbb{E}\left[#1\right]} %Expectation
\newcommand{\rounddown}[1]{\left\lfloor#1\right\rfloor} %floor function

\def\Erdos{Erd\H{o}s}
\def\Renyi{R\'enyi}
\def\ER{\Erdos-\Renyi}

\newtheorem{thm}{Theorem}[section]
\newtheorem{prop}[thm]{Proposition}
\newtheorem{coro}[thm]{Corollary}
\newtheorem{lem}[thm]{Lemma}
\newtheorem{conjecture}[thm]{Conjecture}
\theoremstyle{remark}
\newtheorem{remark}[thm]{Remark}
\theoremstyle{definition}
\newtheorem{definition}[thm]{Definition}
\newcommand{\proofof}[1]{\subsection{Proof of \Cref{#1}}}
\newcommand{\proofofL}[1]{\subsection{Proof of \Cref{#1}}\label{proof:#1}}

\crefname{thm}{theorem}{theorems}
\crefname{prop}{proposition}{propositions}
\crefname{coro}{corollary}{corollaries}
\crefname{lem}{lemma}{lemmas}
\crefname{definition}{definition}{definitions}
\crefname{conjecture}{conjecture}{conjectures}

\newtheoremstyle{claim}%name
{}%Space above
{}%Space below
{\itshape}%Body font
{}%Indent amount
{\bf}%Theorem head font
{.}%Punctuation after theorem head
{.5em}%Space after theorem head
{}%Theorem head spec(can be left empty, meaning ‘normal’)
\theoremstyle{claim}
\newtheorem{claim}{Claim}

\crefname{claim}{claim}{claims}

\def\nc{N}
\def\nd{k}
\def\val{x}
\def\valSecond{y}
\def\ind{i}
\def\indSecond{j}
\def\fixGraph{H}
\def\fixMultigraph{H}
\def\randomGraph{A}
\def\ur{\in_R}
\def\minimum{{X_*}}
\def\maximum{{X^*}}
\def\LargestComponent{L_1}
\def\planarClass{\mathcal{P}}
\def\planarRandomGraph{P}
\def\randomCore{C}
\def\randomMultiCore{\tilde{C}}
\def\constant{M}
\def\generalGraphClass{\mathcal{A}}
\def\generalRandomGraph{A}
\def\colour{F}
\def\pl{kernel-stable}
\def\Pl{Kernel-stable}
\def\loopInsertion{loop insertion}
\def\loopInsertions{loop insertions}
\def\LoopInsertion{Loop insertion}
\def\cl{\mathcal{A}}
\def\property{\mathcal{Q}}
\def\2s{2-simple}
\def\ce{\alpha}
\def\rest{R}
\def\bridgeInsertion{bridge insertion}
\def\bridgeInsertions{bridge insertions}
\def\BridgeInsertion{Bridge insertion}
\def\bridgeClass{\mathcal{M}}
\def\bridgeMarkedClass{\mathcal{M}'}
\def\bridgeGraph{M}
\def\bridgeStable{bridge-stable}
\def\BridgeStable{Bridge-stable}
\def\bridgeNumber{bridge number}
\def\BridgeNumber{Bridge number}
\def\bridgeNumbers{bridge numbers}
\def\block{B}
\def\bridge{e}
\def\largeBlock{dominant}
\def\bb{b}
\def\wb{w}
\def\condRandomGraph{F}
\def\exponent{\mu}
\def\func{\Phi}
\def\seq{\mathbf{s}}

\newcommand{\girth}[1]{g\left({#1}\right)}
\newcommand{\longestCycle}[1]{c\left({#1}\right)}
\newcommand{\numberLoops}[1]{\lambda\left(#1\right)}
\def\NumberLoops{\lambda}

\newcommand{\weight}[1]{w\left(#1\right)}
\newcommand{\numberVertices}[1]{v\left({#1}\right)}
\newcommand{\numberEdges}[1]{e\left({#1}\right)}
\newcommand{\core}[1]{C\left(#1\right)}
\newcommand{\kernel}[1]{K\left(#1\right)}
\newcommand{\complexPart}[1]{Q\left(#1\right)}
\newcommand{\largestComponent}[1]{L_1\left(#1\right)}
\newcommand{\vertexSet}[1]{V\left(#1\right)}
\newcommand{\edgeSet}[1]{E\left(#1\right)}
\newcommand{\Rest}[1]{R\left(#1\right)}
\newcommand{\bn}[1]{\beta\left(#1\right)}
\newcommand{\Bn}[2]{\beta_{#1}\left(#2\right)}
\newcommand{\Nb}[2]{\beta\left(#1, #2\right)}
\newcommand{\blockOrder}[2]{b_{#1}\left(#2\right)}
\newcommand{\blockLargest}[2]{B_{#1}\left(#2\right)}
\newcommand{\subdivisionNumber}[1]{S\left(#1\right)}
\newcommand{\condGraph}[2]{#1 \mid #2}

\definecolor{lightGray}{RGB}{220,220,220}
\definecolor{Green}{RGB}{34,139,34}

\title{Longest and shortest cycles in random planar graphs}
\author{Mihyun Kang, Michael Missethan}
\address{Institute of Discrete Mathematics, Graz University of Technology, Steyrergasse 30, 8010 Graz, Austria}
\email{\{kang,missethan\}@math.tugraz.at}
\thanks{Supported by Austrian Science Fund (FWF): I3747 and W1230}

\begin{document}

\begin{abstract}
Let $P(n,m)$ be a graph chosen uniformly at random from the class of all planar graphs on vertex set $\{1, \ldots, n\}$ with $m=m(n)$ edges. We study the cycle and block structure of $P(n,m)$ when $m\sim n/2$. More precisely, we determine the asymptotic order of the length of the longest and shortest cycle in $P(n,m)$ in the critical range when $m=n/2+o(n)$. In addition, we describe the block structure of $P(n,m)$ in the weakly supercritical regime when $n^{2/3}\ll m-n/2\ll n$.
\end{abstract}

\maketitle

\section{Introduction and results}\label{sec:intro}

\subsection{Motivation}\label{sec:background}
In their seminal papers \cite{erdoes1, erdoes2}, \Erdos\ and \Renyi\ introduced the uniform random graph $G(n,m)$, also known as the \ER\ random graph, which is a graph chosen uniformly at random from the class $\mathcal{G}(n,m)$ of all vertex-labelled graphs on vertex set $[n]:=\{1, \ldots, n\}$ with $m=m(n)$ edges, denoted by $G(n,m) \ur \mathcal G(n,m)$. Since then, $G(n,m)$ and its variants, in particular their component structure, were extensively studied (see e.g. \cite{rg3, rg1, rg2,general1, erdoes2, general2}). For example, \Erdos\ and \Renyi\ \cite{erdoes2} showed that there is a drastic change of the component structure of $G(n,m)$ when $m\sim n/2$. More precisely, letting $m=dn/2$ for a positive constant $d$ they showed that the following hold in $G(n,m)$ with high probability (meaning with probability tending to 1 as $n$ tends to infinity, {\em \whp\ } for short): if $d<1$, then every component has at most a logarithmic number of vertices; in contrast, if $d>1$, there is a unique component containing linearly many vertices. These results raised the question whether also the cycle structure of $G(n,m)$ undergoes such a significant change when $d\sim 1$. Ajtai, {Koml\'os}, and {Szemer\'edi} \cite{longCycle} proved that whp there is a cycle of linear length when $d>1$, and {Bollob\'as} \cite[Corollary 5.8]{rg3} showed that whp every cycle is bounded when $d<1$. (Throughout the paper, we use the standard Landau notation as well as notations in \Cref{def:asy} for asymptotic orders.)

\begin{thm}[\cite{longCycle, rg3}]\label{known_0}
Let $m=dn/2$ for a constant $d>0$. Then the following hold in $G(n,m)\ur \mathcal G(n,m)$.
\begin{enumerate}
\item
If $d<1$, then all cycles are of length $O_p(1)$.
\item
If $d>1$, then \whp\ there is a cycle of length $\Theta(n)$.
\end{enumerate}
\end{thm}
Kolchin \cite{kolchin} and later {\L}uczak \cite{general2} took a closer look at the critical range when $m=n/2+o(n)$ and provided a relation between longest and shortest cycles and the component structure of $G(n,m)$. Their results are strengthened by {\L}uczak \cite{luczak} and later by {\L}uczak, Pittel, and Wierman \cite{luczak_pittel_wierman}. Given a graph $H$, we denote by $\longestCycle{\fixGraph}$ the length of the longest cycle of $\fixGraph$ (also known as the {\em circumference}) and by $\girth{\fixGraph}$ the length of the shortest cycle of $\fixGraph$ (also known as the {\em girth}).
\begin{thm}[\cite{kolchin,general2, luczak, luczak_pittel_wierman}]\label{thm:known}
Let $m=n/2+s$ for $s=s(n)=o\left(n\right)$. Let $G=G(n,m)\ur \mathcal G(n,m)$, $\LargestComponent=\largestComponent{G}$ be the largest component of $G$, and $\rest=\rest(G)=G\setminus\LargestComponent$.
\begin{enumerate}
\item\label{thm:known1}
If $s^3 n^{-2} \to -\infty$, then \whp\ $\LargestComponent$ is a tree. Furthermore, we have $\longestCycle{\rest}=\Theta_p\left(n|s|^{-1}\right)$.
\item\label{thm:known2}
If $s=O\left(n^{2/3}\right)$, then the probability that $\LargestComponent$ is a tree is bounded away from both 0 and 1. Provided there is a cycle in $\LargestComponent$, we have $\longestCycle{\LargestComponent}=\Theta_p\left(n^{1/3}\right)$ and $\girth{\LargestComponent}=\Theta_p\left(n^{1/3}\right)$. Moreover, we have $\longestCycle{\rest}=\Theta_p\left(n^{1/3}\right)$.
\item\label{thm:known3} If $s^3 n^{-2} \to \infty$, then \whp\ the longest cycle in $G$ is contained in $\LargestComponent$ and $\longestCycle{\LargestComponent}=\Theta\left(s^2 n^{-1}\right)$. In addition, we have $\girth{\LargestComponent}=\Theta_p\left(n s^{-1} \right)$ and $\longestCycle{\rest}=\Theta_p\left(n s^{-1}\right)$. 
\end{enumerate}	
\end{thm}
Perhaps the most interesting case in \Cref{thm:known} is when $s^3 n^{-2} \to \infty$ (in the so-called {\em weakly supercritical} regime): \whp\ the longest cycle is contained in the largest component $\LargestComponent$. 
Moreover, the length of the shortest cycle in $\LargestComponent$ is of the same asymptotic order as the length of the longest cycle {\em outside} $\LargestComponent$. In other words, there exists a \lq threshold\rq\ function $f(n):=n s^{-1}$ in the sense that \whp\ for all cycles $K$ in $G$ we have
\begin{align}\label{equ:threshold}
\begin{cases}
K \text{ is contained in } L_1 & \text{if } |K|=\omega\left(f(n)\right); \\[0.1cm]
K \text{ is not contained in } L_1 & \text{if } |K|=o\left(f(n)\right), 
\end{cases}
\end{align}
where $|K|$ denotes the length of $K$.

An important structure related to cycles are {\em blocks}, because every cycle is contained in a block: a block of a graph $H$ is a maximal 2-connected subgraph of $H$. We emphasise that we do not consider a bridge, i.e. an edge whose deletion increases the number of components, to form a block. {\L}uczak \cite{luczak} investigated the block structure of $G(n,m)$ when $m=n/2+s$ for $s^3 n^{-2}\to \infty$, showing that whp there is a unique largest block, while all other blocks are \lq small\rq. In addition, the largest component contains only a few number of blocks. Given a graph $\fixGraph$, let $\blockOrder{i}{\fixGraph}$ denote the number of vertices in the $i$-th largest block $\blockLargest{i}{\fixGraph}$ of $\fixGraph$ for each $i\in \N$.

\begin{thm}[\cite{luczak}]\label{thm:block_luczak}
	Let $m=n/2+s$ for $s=s(n)=o\left(n\right)$ and $s^3 n^{-2} \to \infty$. 
	Let $G=G(n,m)\ur \mathcal G(n,m)$ and $\LargestComponent=\largestComponent{G}$ be the largest component of $G$. Then the following hold.
	\begin{enumerate}
		\item \label{thm:block_luczaka}
		$\blockOrder{1}{G}=\Theta\left(s^2 n^{-1}\right)$ \whp.
		\item \label{thm:block_luczakb}
		$\blockOrder{i}{G}=O_p\left(n s^{-1}\right)$ for each $i\in \N$ with $i\geq 2$.
		\item \label{thm:block_luczakc}
		The number of blocks in $\LargestComponent$ is $O_p(1)$.
	\end{enumerate}
\end{thm}

We note that Theorems \ref{thm:known}\ref{thm:known3} and \ref{thm:block_luczak} imply that \whp\ the longest cycle lies in the largest block $\blockLargest{1}{G}$ and its length grows asymptotically like $\blockOrder{1}{G}$.

In the last few decades various models of random graphs have been introduced by imposing additional constraints to $G(n,m)$, e.g. topological constraints or degree restrictions. 
Particularly interesting models are random {\em planar} graphs and, more generally, random graphs on surfaces which have attained considerable
attention~\cite{chap,chap2,gim,planar,surface,dks,mcdr}, since the pioneering work of McDiarmid, Steger, and Welsh~\cite{msw} on random planar graphs and that of McDiarmid~\cite{mcd} on random graphs on surfaces. Many exciting results have been obtained, revealing richer and more complex behaviour.
In particular, 
results are often found to feature thresholds,
meaning that the probabilities of various properties change dramatically 
according to which \lq region\rq\ the edge density falls into. 

A natural question is whether random planar graphs satisfy similar properties as in \Cref{thm:known,thm:block_luczak}. Kang and {\L}uczak \cite{planar} showed that the component structure of a random planar graph $\planarRandomGraph(n,m)$ changes drastically when $m\sim n/2$, analogously to $G(n,m)$. In contrast, not much is known about the cycle and block structure of $\planarRandomGraph(n,m)$. In this paper we investigate this open problem, determining the length of the shortest and longest cycle in $\planarRandomGraph(n,m)$ and the order of blocks in $\planarRandomGraph(n,m)$, in the light of \Cref{thm:known,thm:block_luczak}.

\subsection{Main results}\label{sec:main}
Throughout this section, we let $\mathcal \planarRandomGraph(n,m)$ denote the class of all vertex-labelled planar graphs on vertex set $[n]$ with $m=m(n)$ edges and $\planarRandomGraph(n,m)$ be a graph chosen uniformly at random from $\mathcal \planarRandomGraph(n,m)$, denoted by $\planarRandomGraph(n,m) \ur \mathcal \planarRandomGraph(n,m)$.

Our first main result concerns the distribution of cycles in the random planar graph $\planarRandomGraph(n,m)$. 
\begin{thm}\label{thm:main1}
Let $\planarRandomGraph=\planarRandomGraph(n,m) \ur \mathcal \planarRandomGraph(n,m)$, $\LargestComponent=\largestComponent{\planarRandomGraph}$ be the largest component of $\planarRandomGraph$, and $\rest=\rest(\planarRandomGraph)=\planarRandomGraph \setminus \LargestComponent$. Assume $m=n/2+s$ for $s=s(n)=o\left(n\right)$. Then the following hold.
	\begin{enumerate}
		\item \label{thm:main1_a}
If $s^3 n^{-2} \to -\infty$, then \whp\ $\LargestComponent$ is a tree. Furthermore, we have $\longestCycle{\rest}=\Theta_p\left(n|s|^{-1}\right)$.
		\item \label{thm:main1_b}
		If $s=O\left(n^{2/3}\right)$, then the probability that $\LargestComponent$ is a tree is bounded away from both 0 and 1. Provided there is a cycle in $\LargestComponent$, we have $\longestCycle{\LargestComponent}=\Theta_p\left(n^{1/3}\right)$ and $\girth{\LargestComponent}=\Theta_p\left(n^{1/3}\right)$. Moreover, we have $\longestCycle{\rest}=\Theta_p\left(n^{1/3}\right)$.
		\item \label{thm:main1_c}
		If $s^3 n^{-2} \to \infty$, then \whp\ the longest cycle in $\planarRandomGraph$ is contained in $\LargestComponent$ and $\longestCycle{\LargestComponent}=O\left(sn^{-1/3}\right)$.
		In addition, we have $\longestCycle{\LargestComponent}=\Omega_p\left(n^{1/3}\log \left(sn^{-2/3}\right)\right)$, $\girth{\LargestComponent}=\Theta_p\left(n s^{-1}\right)$, and $\longestCycle{\rest}=\Theta_p\left(n^{1/3}\right)$.
	\end{enumerate}	
\end{thm}
As we will see in \Cref{thm:general}, \Cref{thm:main1} holds for a more general universal class of graphs, the so-called {\em \pl} classes of graphs (see \Cref{def:planar_like} for a formal definition), which include the class of series-parallel graphs, the class of planar graphs, and the class of graphs on a surface, to mention a few.

Our second main result deals with the block structure of $\planarRandomGraph(n,m)$. Due to \Cref{thm:main1}\ref{thm:main1_a} and \ref{thm:main1_b} we focus on the weakly supercritical regime and will show that whp $\planarRandomGraph(n,m)$ contains a unique largest block $B_1$ which is significantly larger than all other blocks, similarly as in $G(n,m)$. However, the largest component in $\planarRandomGraph(n,m)$ contains \lq many\rq\ blocks, while the largest component in $G(n,m)$ contains only a bounded number of blocks (cf. \Cref{thm:block_luczak}\ref{thm:block_luczakc}).
\begin{thm}\label{thm:block_structure}
Let $\planarRandomGraph=\planarRandomGraph(n,m) \ur \mathcal \planarRandomGraph(n,m)$ and $\LargestComponent=\largestComponent{\planarRandomGraph}$ be the largest component of $\planarRandomGraph$. 
Assume $m=n/2+s$ for $s=s(n)=o\left(n\right)$ and $s^3 n^{-2} \to \infty$. Then the following hold.
	\begin{enumerate}
		\item\label{thm:block_structure_a}
		$\blockOrder{1}{\planarRandomGraph}=\Theta_p\left(sn^{-1/3}\right)$.
		\item\label{thm:block_structure_b}
		$\blockOrder{i}{\planarRandomGraph}=\Theta_p\left(s^{2/3}n^{-1/9}\right)$, for each $i\in \N$ with $i\geq 2$.
		\item\label{thm:block_structure_c}
		The number of blocks in $\LargestComponent$ is \whp\ $\Theta\left(sn^{-2/3}\right)$.
	\end{enumerate}
\end{thm}

It is well known that when $s^3 n^{-2}\to -\infty$ or $s=O\left(n^{2/3}\right)$, the probability that $G(n,m)$ is planar is bounded away from 0 (see e.g. \Cref{thm:critcalRegime}\ref{thm:critcalRegime_a} and \cite{comp_fac,luczak_pittel_wierman,prob_planarity}). Hence, each graph property that holds \whp\ in $G(n,m)$ is also true \whp\ in $\planarRandomGraph(n,m)$. In particular, this implies that the cycle structure of $\planarRandomGraph(n,m)$ \lq behaves\rq\ similarly like that of $G(n,m)$ (see Theorems \ref{thm:known}\ref{thm:known1}, \ref{thm:known2} and \ref{thm:main1}\ref{thm:main1_a}, \ref{thm:main1_b}). However, when $s^3n^{-2}\to \infty$, \whp\ $G(n,m)$ is not planar (see e.g. \cite{luczak_pittel_wierman,prob_planarity}) and therefore, $G(n,m)$ and $\planarRandomGraph(n,m)$ can exhibit different asymptotic behaviours. Theorems \ref{thm:main1}\ref{thm:main1_c} and \ref{thm:block_structure} indicate that in view of the cycle and block structure this is indeed the case. For example, a \lq threshold\rq\ function in the sense of (1) does not exist in $\planarRandomGraph(n,m)$, because $\girth{\LargestComponent}= \Theta_p\left(n s^{-1}\right)\ \ll\ \longestCycle{\rest}=\Theta_p\left(n^{1/3}\right)$. However, \whp\ the longest cycle in $\planarRandomGraph(n,m)$ is still contained in the largest component.
	
Kang and {\L}uczak \cite{planar} proved that in the case of $s^3 n^{-2} \to \infty$ the core, i.e. the maximal subgraph of minimum degree at least two (also known as 2-core), is much smaller in $\planarRandomGraph(n,m)$ compared to $G(n,m)$. More precisely, \whp\ the core of $\planarRandomGraph(n,m)$ is of order $\Theta\left(sn^{-1/3}\right)$, while that of $G(n,m)$ is of order $\Theta\left(s^2n^{-1}\right)$. This has a natural impact on the order of the longest cycle and largest block in $\planarRandomGraph(n,m)$ (cf. Theorems \ref{thm:known}\ref{thm:known3}, \ref{thm:main1}\ref{thm:main1_c} and Theorems \ref{thm:block_luczak}\ref{thm:block_luczaka} and \ref{thm:block_structure}\ref{thm:block_structure_a}):
\begin{align*}
\longestCycle{\planarRandomGraph(n,m)}=O\left(sn^{-1/3}\right)\quad \text{\whp} \quad &\ll \quad \longestCycle{G(n,m)}=\Theta\left(s^2 n^{-1}\right) \quad \text{\whp}
\\
\blockOrder{1}{\planarRandomGraph(n,m)}=\Theta_p\left(sn^{-1/3}\right) \quad &\ll \quad \blockOrder{1}{G(n,m)}=\Theta\left(s^2 n^{-1}\right) \quad \text{\whp}.
\end{align*}	
Furthermore, it is known that the \lq edge density\rq\ in the part without the largest component is typically much larger in $P(n,m)$ than in $G(n,m)$ (see e.g. \cite[Theorem 1.7]{surface}). This affects the order of the longest cycle outside the largest component (cf. Theorems \ref{thm:known}\ref{thm:known3}, \ref{thm:main1}\ref{thm:main1_c}):
\begin{align*}
\longestCycle{\rest (\planarRandomGraph(n,m))}=\Theta_p\left(n^{1/3}\right)\quad &\gg \quad \longestCycle{\rest (G(n,m))}=\Theta_p\left(n s^{-1}\right).
\end{align*}

\subsection{Related work}
The so-called {\em $n$-vertex} model of a random planar graph is a graph chosen uniformly at random from the class of all vertex-labelled planar graphs on vertex set $[n]$, denoted by $P(n)\ur \planarClass(n)$. Gim\'enez and Noy \cite{gim} showed that \whp\ $P(n)$ has $\left(1+o(1)\right)\kappa n$ edges for a constant $\kappa\approx 2.21$, i.e. $P(n)$ \lq behaves\rq\ like $P(n,m)$ where $m\approx \kappa n$. Many exciting results on the block structure of $P(n)$ were obtained in recent literature, revealing a different behaviour from that of $P(n,m)$ as observed in \Cref{thm:block_structure}. For example, Panagiotou and Steger \cite{ps} proved that \whp\ $\blockOrder{1}{P(n)}=\Theta(n)$. Later, Gim\'{e}nez, Noy, and Ru\'{e} \cite[Proposition 5.3]{GNR2013} established that in fact an Airy-type central limit theorem holds for $\blockOrder{1}{P(n)}$. Stufler \cite[Theorem 6.20, Corollary 6.42]{Stu2020} determined the limiting distribution of $\blockOrder{i}{P(n)}$ for any fixed integer $i\geq 2$, showing, among others, that $\blockOrder{i}{P(n)}=O_p\left(n^{2/3}\right)$. Furthermore, Stufler \cite[Remark 9.13]{Stu2019} provided a detailed structural description of the graph $P(n)\setminus \blockLargest{1}{P(n)}$, i.e. $P(n)$ without its largest block.

Another model related to $P(n,m)$ is the random {\em connected} planar graph $C(n,m)$, which is a graph chosen uniformly at random from the class of all connected planar graphs on vertex set $[n]$ having $m=m(n)$ edges. Panagiotou \cite[Theorem 1, Corollary 1]{Pan09} proved that if $m=\rounddown{cn}$ for a constant $c\in (1,3)$, then $C(n,m)$ has a block of linear order. Moreover, there is a discussion in \cite[End of Section 5]{GNR2013} sketching how much stronger results on the order of the largest block in $C(n,m)$ can in principle be obtained.

\subsection{Key techniques}
One of the main proof techniques is the so-called {\em core-kernel approach}. We decompose a graph into the simple part (in which each component contains at most one cycle) and the complex part (in which each component contains at least two cycles). Then we decompose the complex part into its core and then into its kernel, a key structure obtained from the core by replacing each path whose internal vertices all have degree exactly two by an edge. Conversely, each graph can be uniquely constructed from the kernel by first subdividing the edges of the kernel, thereby obtaining the core, then replacing vertices of the core with rooted trees and adding the simple part (see \Cref{sub:dec}).

In order to investigate the cycle and block structure of a random planar graph $\planarRandomGraph=\planarRandomGraph(n,m)$, we begin with the analysis of the structure of its core $\core{\planarRandomGraph}$, which is itself a random graph. Instead of directly analysing the random core $\core{\planarRandomGraph}$, we introduce an auxiliary random core model $\tilde C$, in which we split the \lq randomness\rq\ into smaller parts. More precisely, we choose randomly a kernel and then randomly a subdivision number which is a total number of vertices that will be used for a subdivision of the kernel. Given these two random bits (i.e. a random kernel and a random subdivision number) we then randomly construct a core by randomly inserting vertices on the edges of the kernel. A crucial technique to analyse the random core $\tilde C$ is the famous {\em P\'olya urn model}: 
We derive results on the maximum and minimum number of drawn balls of some colour in order to determine the length of the longest and shortest cycles in the core, respectively.

\subsection{Outline of the paper} 
The rest of the paper is organised as follows. After providing the necessary notations, definitions, and concepts in \Cref{sec:prelim}, we present our proof strategy in \Cref{sec:proof_strategy}. In \Cref{sec:planar_like} we define \pl\ classes of graphs. In \Cref{sec:polya} we provide results on the P\'olya urn model, which we use in \Cref{sec:CoreLabelled} to derive the cycle structure of a core randomly built from a fixed kernel and a fixed subdivision number. \Cref{sec:random_kernel} is devoted to a random kernel and \Cref{sec:block_structure} to the block structure of a random planar graph. In \Cref{sec:proofs_main,sec:proofs_aux} we provide the proofs of our main and auxiliary results, respectively. Finally in \Cref{sec:discussion}, we discuss various questions that remain unanswered.

\section{Preliminaries}\label{sec:prelim}

\subsection{Notations and parameters for graphs}
Unless stated otherwise, all considered (simple or multi) graphs are {\em undirected}.

\begin{definition}
Given a (simple or multi) graph $\fixMultigraph$ we denote by
	\begin{itemize}
		\item 
		$\vertexSet{\fixMultigraph}$ the vertex set of $\fixMultigraph$ and \item[] 
		$\numberVertices{\fixMultigraph}$ the order of $H$, i.e. the number of vertices in $\fixMultigraph$;
		\item 
		$\edgeSet{\fixMultigraph}$ the edge set of $\fixMultigraph$ and
		\item[] 
		$\numberEdges{\fixMultigraph}$ the size of $H$, i.e. the number of edges in $\fixMultigraph$;	
		\item 
		$\largestComponent{\fixMultigraph}$ the largest component of $\fixMultigraph$;
		\item
		$\Rest{\fixMultigraph}:=\fixMultigraph \setminus \largestComponent{\fixMultigraph}$ the graph outside the largest component;
		\item 
		$\girth{\fixMultigraph}$ the girth of $\fixMultigraph$, i.e. the length of the shortest cycle in $\fixMultigraph$;
		\item
		$\longestCycle{\fixMultigraph}$ the circumference of $\fixMultigraph$, i.e. the length of the longest cycle in $\fixMultigraph$;
		\item 
		$\numberLoops{\fixMultigraph}$ the number of loops in $\fixMultigraph$;
		\item $\blockLargest{i}{\fixMultigraph}$ the $i$-th largest block of $\fixMultigraph$ and
			\item[] 
		$\blockOrder{i}{\fixMultigraph}$ the number of vertices in $\blockLargest{i}{\fixMultigraph}$ for $i \in \N$.
	\end{itemize}
\end{definition}

\begin{definition}
Let $\fixMultigraph$ be a (simple or multi) graph and let $v, w \in \vertexSet{\fixMultigraph}$ be distinct. We denote by
\begin{itemize}
\item
$vw$ or $\{v,w\}$ an edge between $v$ and $w$;
\item
$vv$ a loop at $v$;
\item
$\fixMultigraph+vw$ the graph obtained from $\fixMultigraph$ by adding an additional edge $vw$;
\item
$\fixMultigraph-vw$ the graph obtained from $\fixMultigraph$ by deleting the edge $vw$.
\end{itemize}
\end{definition}

\begin{definition}\label{def:graph_class}
Given a class $\generalGraphClass$ of vertex-labelled graphs (under consideration of certain constraints, e.g. planarity or degree restrictions), we denote by 
$\generalGraphClass(n)$ the subclass of $\generalGraphClass$ consisting of graphs on vertex set $[n]$ 
and by $\generalGraphClass(n,m)$ the subclass of $\generalGraphClass$ consisting of graphs on vertex set $[n]$ with $m$ edges, respectively. We denote 
by $\generalRandomGraph(n)\ur \generalGraphClass(n)$ a graph chosen uniformly at random from $\generalGraphClass(n)$
and by $\generalRandomGraph(n,m)\ur \generalGraphClass(n,m)$ a graph chosen uniformly at random from $\generalGraphClass(n,m)$, respectively.
\end{definition}

\begin{definition}
Let $\fixGraph$ be a graph and $\property$ a set of graphs. We call $\property$ a {\em graph property}. And if $\fixGraph \in \property$, then we say that $\fixGraph$ satisfies $\property$; or that $\property$ holds in $\fixGraph$; or that $\property$ is true in $\fixGraph$.
\end{definition}

Next, we introduce some notion for random graphs which have the \lq same\rq\ asymptotic behaviour in the sense that they are indistinguishable in view of properties that hold \whp.
\begin{definition}\label{def:contiguous}
For each $n\in \N$, let $G_n$ and $H_n$ be random graphs. We say that $G_n$ and $H_n$ are {\em contiguous} if for every graph property $\property$
\begin{align*}
\lim_{n \to \infty}\prob{G_n \in \property} = 1 \quad \Longleftrightarrow \quad \lim_{n \to \infty}\prob{H_n \in\property} = 1.
\end{align*}
\end{definition}

\subsection{Weighted Multigraphs}
Throughout the paper, we always assume implicitly that multigraphs are weighted by the so-called {\em compensation factor}, which was first introduced by Janson, Knuth, {\L}uczak, and Pittel \cite{comp_fac}.

\begin{definition}\label{def:weight}
Given a multigraph $\fixMultigraph$ and $i \in \N$, we denote by $m_i(\fixMultigraph)$ the number of unordered pairs $\{v, w\}$ of distinct vertices $v, w\in V(H)$ such that there are precisely $i$ edges between $v$ and $w$. Similarly, let $\NumberLoops_i(\fixMultigraph)$ be the number of vertices $v\in V(H)$ such that there are exactly $i$ loops at $v$ and let $\numberLoops{\fixMultigraph}:=\sum_{i \in \N}i\NumberLoops_i(\fixMultigraph)$ be the total number of loops in $H$. Then the {\em compensation factor} (or weight, for short) of $\fixMultigraph$ is defined as
	\begin{align}\label{eq:weight}
	\weight{\fixMultigraph}:=2^{-\numberLoops{\fixMultigraph}}\prod_{i \in \N}\left(i!\right)^{-\NumberLoops_i(\fixMultigraph)-m_i(\fixMultigraph)}.
	\end{align}
	For a finite class $\mathcal{A}$ of multigraphs we define
	\[
	|\mathcal{A}|:=\sum_{H \in \mathcal{A}}w(H).
	\]
\end{definition}

\subsection{Complex part, core, and kernel}\label{sub:dec}
We call a component of a graph $\fixGraph$ {\em complex} if it has at least two cycles and define the {\em complex part} $\complexPart{\fixGraph}$ as the union of all complex components of $\fixGraph$. We decompose the complex part $\complexPart{\fixGraph}$ further into the {\em core} $\core{\fixGraph}$, which is the maximal subgraph of $\complexPart{\fixGraph}$ of minimum degree at least two. Finally, we extract the {\em kernel} $\kernel{\fixGraph}$ from the core $\core{\fixGraph}$ by considering paths $(v_0, v_1, \ldots, v_i)$ such that $v_0$ and $v_i$ have degree at least three and all internal vertices $v_1, \ldots, v_{i-1}$ have degree two. We allow the case $v_0=v_i$, in which $(v_0, v_1, \ldots, v_i)$ is a cycle. To obtain the kernel $\kernel{\fixGraph}$, we replace any such path in the core $\core{\fixGraph}$ by an edge $v_0v_i$. By doing that, loops and multiple edges can be created and therefore in general the kernel $\kernel{\fixGraph}$ is a multigraph. Finally, we reverse the above decomposition and note that we can construct the core $\core{\fixGraph}$ by subdividing the edges of the kernel $\kernel{\fixGraph}$ with additional vertices, thereby ensuring that no loops and multiple edges of the kernel survive in the core. The number of additional vertices that are used to subdivide the kernel $\kernel{\fixGraph}$ to obtain the core $\core{\fixGraph}$ is called the {\em subdivision number}, denoted by $\subdivisionNumber{\fixGraph}$, i.e. $\subdivisionNumber{\fixGraph}:=\numberVertices{\core{\fixGraph}}-\numberVertices{\kernel{\fixGraph}}$.

\subsection{Asymptotic notation}
We will study asymptotic properties of random graphs on vertex set $[n]$ as $n$ tends to $\infty$, and all asymptotics are taken with respect to $n$.
In addition to the standard Landau notation, we will use the notations by Janson \cite{asy_not} to express asymptotic orders of random variables.

\begin{definition}\label{def:asy}
	Let $\left(X_n\right)_{n \in N}$ be a sequence of random variables and $f: \N \to \R_{\geq 0}$.
	Then, we write
	\begin{itemize}
		\item
		\whp\ $X_n=O\left(f\right)$ if there exists a $c>0$ such that \whp\ $|X_n|\leq c f(n)$;
		\item
		\whp\ $X_n=\Omega\left(f\right)$ if there exists a $c>0$ such that \whp\ $|X_n|\geq c f(n)$;
		\item
		\whp\ $X_n=\Theta\left(f\right)$ if \whp\ $X_n=O\left(f\right)$ and $X_n=\Omega\left(f\right)$;
		\item 
		$X_n=O_p\left(f\right)$ if for every $\delta>0$ there exist $c>0$ and $N \in \N$ such that $\prob{|X_n|\leq c f(n)}\geq 1-\delta$ for all $n\geq N$;
		\item 
		$X_n=\Omega_p\left(f\right)$ if for every $\delta>0$ there exist $c>0$ and $N \in \N$ such that $\prob{|X_n|\geq c f(n)}\geq 1-\delta$ for all $n\geq N$ ;
		\item 
		$X_n=\Theta_p\left(f\right)$ if $X_n=O_p\left(f\right)$ and $X_n=\Omega_p\left(f\right)$.
	\end{itemize}
Moreover, for a sequence $\left(A_n\right)_{n\in\N}$ of events we say that the probability that $A_n$ occurs is bounded away from 0 and 1 if $\liminf_{n \to \infty}\prob{A_n}>0$ and $\limsup_{n \to \infty}\prob{A_n}<1$, respectively. 
\end{definition} 

We note that $X_n=O_p\left(f\right)$ if and only if $\prob{\left|X_n\right|\geq h(n)f(n)}=o(1)$ for any function $h=\omega(1)$. Similarly, we have $X_n=\Omega_p\left(f\right)$ if and only if $\prob{\left|X_n\right|\leq f(n)/h(n)}=o(1)$ for any function $h=\omega(1)$. Furthermore, the statement $X_n=O_p\left(f\right)$ is slightly weaker than having \whp\ $X_n=O\left(f\right)$. The latter says that there exists a constant $c>0$ such that for every $\delta>0$ there is a $N\in \N$ satisfying $\prob{|X_n|\leq c f(n)}\geq 1-\delta$ for all $n\geq N$. In other words, we have a uniform constant $c$ if \whp\ $X_n=O\left(f\right)$, while $c$ may depend on $\delta$ in the case of $X_n=O_p\left(f\right)$.

\section{Proof strategy}\label{sec:proof_strategy}
In order to analyse the cycle and block structure of a random planar graph $\planarRandomGraph=\planarRandomGraph(n,m)$, we decompose $\planarRandomGraph$ into smaller parts: We first decompose $\planarRandomGraph$ into the complex part $\complexPart{\planarRandomGraph}$, the core $\core{\planarRandomGraph}$, and the kernel $\kernel{\planarRandomGraph}$ as in \Cref{sub:dec}. Moreover, we obtain the core $\core{\planarRandomGraph}$ by subdividing the edges of the kernel $\kernel{\planarRandomGraph}$ by $\subdivisionNumber{\planarRandomGraph}$ many additional vertices. We note that all blocks and cycles that do not lie in unicyclic components are contained in the core $\core{\planarRandomGraph}$. Therefore, it is crucial to understand the cycle and block structure of the core $\core{\planarRandomGraph}$. To this end, we introduce an auxiliary random core $\randomCore$ which is created stepwise as follows. We choose randomly a typical kernel $K$ and a typical subdivision number $\nd$. Then we construct $\randomCore$ from $K$ by randomly subdividing the edges of $K$ with $\nd$ additional vertices. In order to analyse the cycle and block structure of $\core{\planarRandomGraph}$, we ask the following questions. (a) Which properties do a typical kernel $K$ and a typical subdivision number $k$ have, in particular with respect to cycles and blocks? 
(b) How do these properties translate to $\randomCore$ when choosing a random subdivision? 
(c) How can we relate the random graph $\randomCore$ to the original core $\core{\planarRandomGraph}$ where we first choose the random graph $\planarRandomGraph$ and then extract (deterministically) the core $\core{\planarRandomGraph}$. In the rest of this section we will give an overview how to deal with these questions. We will start by considering question (b) in \Cref{sub:construction_kernel_core}. The main idea is to find a relationship between a random core $C$ and the P\'olya urn model. Next, we will deal with question (a). We note that already lots of results about a typical kernel $K$ and a typical subdivision number $\nd$ are known, e.g. asymptotic order of $\numberVertices{K}, \numberEdges{K}$, and $\nd$ (see \Cref{thm:internal_structure}). In contrast, there are no results on the cycle and block structure of $K$ known. We will deal with these open problems as follows. Firstly, we will show that a typical kernel \lq behaves\rq\ asymptotically like a random cubic (i.e. 3-regular) planar multigraph (see \Cref{lem:kernel_cubic}). Then we will analyse the cycle (in \Cref{sub:cycles_kernel}) and block structure (in \Cref{sub:block_kernel}) of a random cubic planar multigraph by double counting arguments. Finally, we will introduce the concept of {\em conditional random graphs} in \Cref{sub:conditional} to give an answer to question (c).

\subsection{Random core and P\'olya urn model}\label{sub:construction_kernel_core}
Given a (typical) kernel $K$ and a (typical) subdivision number $k\in \N$ we let $C=C(K,k)$ be a random core chosen uniformly at random from the set of all cores with kernel $K$ and subdivision number $k$. In order to study the cycle and block structure of a random core $C$, we consider an auxiliary random multigraph $\randomMultiCore$ that \lq behaves\rq\ similarly like $C$ and is easier to study. We randomly place $\nd$ additional labelled vertices $\{v_1, \ldots, v_\nd\}$ one after another on the edges $e_1, \ldots, e_{\nc}$ of $K$ to obtain $\randomMultiCore$. More precisely, in the $i$-th step of the random process we choose uniformly at random an edge of the current multigraph and place the vertex $v_i$ on that edge. We note that loops and multiple edges are allowed in $\randomMultiCore$, although they do not appear in $C$. However, we will show that $C$ and $\randomMultiCore$ behave asymptotically quite similarly (see \Cref{lem:random_core} and \Cref{cor:contiguous}). Thus, it suffices to consider the cycle and block structure of $\randomMultiCore$ instead of $C$. 

We note that each cycle of $\randomMultiCore$ is a cycle in the kernel $K$ together with some additional vertices placed on the edges. Thus, we are interested in the distribution of $(X_1, \ldots, X_\nc)$, where $X_i$ is the number of vertices placed on edge $e_i$. We observe that we can model the random vector $(X_1, \ldots, X_\nc)$ with the following P\'olya urn model (see e.g. \cite{johnson_kotz, mahmoud} for details of P\'olya urns). We start with $\nc$ balls of $\nc$ distinct colours $\colour_1, \ldots, \colour_\nc$ in a urn, where colour $\colour_i$ represents edge $e_i$. Then we draw one ball uniformly at random from the urn, say $F_j$, and subdivide the corresponding edge $e_j$ by vertex $v_1$. By doing so the edge $e_j$ is split into two new edges. Hence, we need in the next step two balls of the colour corresponding to $e_j$. Therefore, we return the drawn ball to the urn along with an additional ball of the same colour. We repeat that procedure $\nd$ times and observe that the number of drawn balls (after $\nd$ steps) of colour $\colour_i$ is distributed like $X_i$ for each $i\in [N]$ (see also \Cref{fig:polya}). 

The P\'olya urn model provides bounds on $\min_{1\leq \ind \leq f} X_\ind$ and $\max_{1\leq \ind \leq f} X_\ind$ for various $f$ satisfying $1\leq f \leq \nc$ (see \Cref{main1}). Assuming $e_1, \ldots, e_{\NumberLoops}$ are the loops of $K$, we will derive bounds on the length of the shortest and longest cycle in $\randomMultiCore$ (see \Cref{lem:random_core2}), denoted by $\girth{\randomMultiCore}$ and $\longestCycle{\randomMultiCore}$, by applying the following inequalities:
\begin{align}
\girth{\randomMultiCore}&\ \leq\ 1+\min_{1\leq \ind \leq \NumberLoops}X_\ind; \label{bound1}
\\
\girth{\randomMultiCore}&\ \geq\ \min_{1\leq \ind \leq \nc}X_\ind; \label{bound2}
\\
\longestCycle{\randomMultiCore}&\ \geq\ \max_{1\leq \ind \leq \NumberLoops}X_\ind. \label{bound3}
\end{align}

Next, we consider how the block structure of $K$ translates to the block structure of $\randomMultiCore$. We observe that each block in $\randomMultiCore$ is a block or loop in $K$ together with additional vertices placed on the edges. So in general we can use similar ideas as for the cycle structure above. However, we need to slightly modify our P\'olya urn model (see \Cref{sub:polya2}). For the cycle structure, we use many \lq small\rq\ cycles (in fact, loops) simultaneously, but for the block structure we will fix one \lq large\rq\ block $B$ of $K$ and consider how many vertices are placed on the edges of $B$. Therefore, we just need balls of two different colours, one representing edges in $B$ and the other representing edges outside $B$. Using standard results on this P\'olya urn model, we will show that each \lq large\rq\ block $B$ of $K$ translates to a \lq large\rq\ block $\tilde{B}$ of $\randomMultiCore$ and that $\numberVertices{\tilde{B}}$ is concentrated around its expectation (see \Cref{lem:bridge_number6,lem:bridge_number7}). 

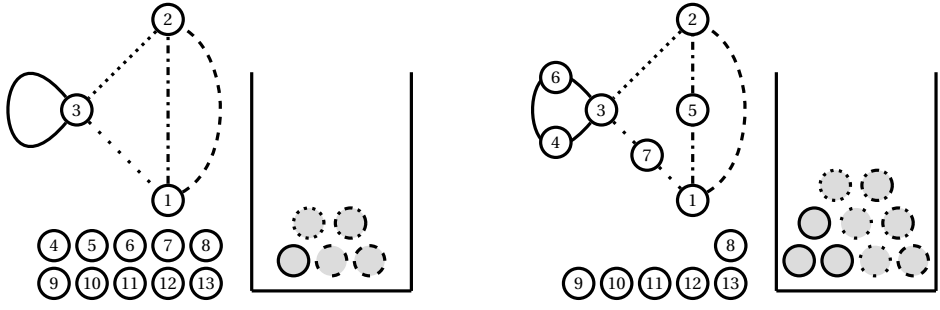
\begin{figure}[t]
	\centering
\begin{tikzpicture}[scale=1, very thick, every node/.style={circle, minimum size=0.4cm, inner sep=0pt}]
{
\node (A) at (0,0) [circle,draw] {\tiny 1};

\node (E) at (0,2.4) [circle,draw] {\tiny 2};
\node (F) at (-1.2,1.2) [circle,draw] {\tiny 3};

\draw[-, dotted] (E) to (F);
\draw [out=30,in=-30,looseness=0.9, dashed] (A) to (E);
\draw[-, dashdotted] (A) to (E);
\draw [out=130,in=230,looseness=12, solid](F) to (F);	
\draw[-, loosely dotted] (F) to (A);

\node (G) at (-1.5,-0.6) [circle,draw, fill=white] {\tiny 4};

\node (H) at (-1,-0.6) [circle, draw, fill=white] {\tiny 5};

\node (I) at (-0.5,-0.6) [circle, draw, fill=white] {\tiny 6};

\node (J) at (0,-0.6) [circle, draw, fill=white] {\tiny 7};

\node (K) at (0.5,-0.6) [circle, draw, fill=white] {\tiny 8};

\node (L) at (-1.5,-1.1) [circle, draw, fill=white] {\tiny 9};
\node (M) at (-1,-1.1) [circle, draw, fill=white] {\tiny 10};

\node (N) at (-0.5,-1.1) [circle, draw, fill=white] {\tiny 11};
\node (O) at (0,-1.1) [circle, draw, fill=white] {\tiny 12};
\node (P) at (0.5,-1.1) [circle, draw, fill=white] {\tiny 13};	
\draw[-] (1.1,-1.2) to (3.2,-1.2);
\draw[-] (1.1,-1.2) to (1.1,1.7);
\draw[-] (3.2,1.7) to (3.2,-1.2);
\node (a) at (1.65,-0.8) [circle, draw, solid, fill=lightGray] {};	
\node (b) at (2.15,-0.8) [circle, draw, loosely dashed,fill=lightGray] {};	
\node (c) at (2.65,-0.8) [circle, draw, dashed,fill=lightGray] {};
\node (d) at (1.85,-0.3) [circle, draw, dotted,fill=lightGray] {};	
\node (e) at (2.4,-0.3) [circle, draw, dashdotted, fill=lightGray] {};	

\node (A1) at (6.9,0) [circle,draw] {\tiny 1};
\node (E1) at (6.9,2.4) [circle,draw] {\tiny 2};
\node (F1) at (5.7,1.2) [circle,draw] {\tiny 3};

\draw[-, dotted] (E1) to (F1);
\draw [out=30,in=-30,looseness=0.9, dashed] (A1) to (E1);
\draw [out=130,in=230,looseness=12, solid](F1) to (F1);

\node (G1) at (5.1,0.77) [circle,draw, fill=white] {\tiny 4};

\node (H1) at (6.9,1.2) [circle, draw, fill=white] {\tiny 5};

\node (I1) at (5.1,1.63) [circle, draw, fill=white] {\tiny 6};

\node (J1) at (6.3,0.6) [circle, draw, fill=white] {\tiny 7};

\node (K1) at (7.4,-0.6) [circle, draw, fill=white] {\tiny 8};

\node (L1) at (5.4,-1.1) [circle, draw, fill=white] {\tiny 9};
\node (M1) at (5.9,-1.1) [circle, draw, fill=white] {\tiny 10};
\node (N1) at (6.4,-1.1) [circle, draw, fill=white] {\tiny 11};
\node (O1) at (6.9,-1.1) [circle, draw, fill=white] {\tiny 12};
\node (P1) at (7.4,-1.1) [circle, draw, fill=white] {\tiny 13};	
\draw[-] (8,-1.2) to (10.1,-1.2);
\draw[-] (8,-1.2) to (8,1.7);
\draw[-] (10.1,1.7) to (10.1,-1.2);
\node (a1) at (8.3,-0.8) [circle, draw, solid, fill=lightGray] {};	
\node (b1) at (8.8,-0.8) [circle, draw, solid,fill=lightGray] {};	
\node (c1) at (9.3,-0.8) [circle, draw, loosely dotted,fill=lightGray] {};
\node (d1) at (8.5,-0.3) [circle, draw, solid,fill=lightGray] {};	
\node (e1) at (9.05,-0.3) [circle, draw, loosely dotted, fill=lightGray] {};	
\node (f1) at (9.8,-0.8) [circle, draw, dashed, fill=lightGray] {};	
\node (g1) at (9.6,-0.3) [circle, draw, dashdotted,fill=lightGray] {};
\node (h1) at (8.775,0.2) [circle, draw, dotted, fill=lightGray] {};	
\node (i1) at (9.325,0.2) [circle, draw, dashdotted , fill=lightGray] {};	
\draw[-, loosely dotted] (J1) to (A1);
\draw[-, loosely dotted] (J1) to (F1);

\draw[-, dashdotted] (A1) to (H1);
\draw[-, dashdotted] (H1) to (E1);

}			
\end{tikzpicture}
	\caption{Use of a P\'olya urn to construct a random core by sequentially subdividing the edges of a kernel with $\nc=5$ edges by $\nd=10$ additional vertices. The left-hand side represents the situation at the beginning, the right-hand side after four drawings.}
\label{fig:polya}
\end{figure}

\subsection{Loops in the kernel}\label{sub:cycles_kernel}
When studying the shortest and longest cycles in the kernel, the number of loops in the kernel plays a crucial role. We will prove in \Cref{sec:random_kernel} that a typical kernel $K$ on $\nc$ edges has $\Theta\left(\nc\right)$ many loops. Firstly, we will show that the kernel $\kernel{\planarRandomGraph}$ \lq behaves\rq\ asymptotically like a random cubic planar multigraph (see \Cref{lem:kernel_cubic}). Secondly, we estimate the typical number of loops in a random cubic planar multigraph by the second moment method (see \Cref{thm:linear_loop}). To this end, we introduce the so-called {\em \loopInsertion} (see \Cref{def:loop_insertion}), which is a natural operation that changes an arbitrary cubic graph with $2n-2$ vertices to a cubic graph with $2n$ vertices and an additional loop. By using this \loopInsertion\ we can estimate the probability that there is a loop at some fixed vertex in a random cubic planar multigraph, from which we deduce the typical number of loops in a random cubic planar multigraph.

\subsection{Blocks of the kernel}\label{sub:block_kernel}
As in \Cref{sub:cycles_kernel} we will use the fact that the kernel $\kernel{\planarRandomGraph}$ behaves asymptotically like a random cubic planar multigraph. In order to analyse the block structure of a random connected cubic planar multigraph $\bridgeGraph$, we assign the \bridgeNumber\ $\bn{\bridge}$ to a bridge $\bridge$, defined as the order of the smaller component which we obtain from $\bridgeGraph$ by deleting $\bridge$ (see \Cref{def:bridge_number}). We will show that the \bridgeNumber\ $\bn{\bridge}$ is typically quite \lq small\rq: In fact, we will determine the distribution of a \bridgeNumber\ (see \Cref{lem:bridge_number1}), using the so-called {\em \bridgeInsertion} operation (see \Cref{def:bridge_insertion}). Then we combine it with a double counting argument to show that there is one block $B_1$ of linear order (see \Cref{lem:bridge_number6}). For the second largest block $B_2$ in $\bridgeGraph$ we consider the maximum $A$ of all \bridgeNumbers. As there is a bridge $\bridge$ such that $B_1$ and $B_2$ lie in different components of $\bridgeGraph-\bridge$, we will get $\numberVertices{B_2}\leq A$. On the other hand, if $\bridge$ is a bridge with $\bn{e}=A$, then the smaller component of $\bridgeGraph-\bridge$ is distributed similarly as a random connected cubic planar multigraph on $A$ vertices. Hence, this smaller component should contain a block of linear order (in $A$). Thus, the second largest block (and by induction also the $i$-th largest block for every $i\geq 2$) is of the same order as the maximum bridge number $A$. 

\subsection{Conditional random graphs}\label{sub:conditional}
In \Cref{sub:construction_kernel_core} we considered a random core $C=C(K,k)$ obtained from a (candidate) kernel $K$ by randomly subdividing the edges of $K$ by $\nd$ additional vertices. In other words, given $K$ and $k$, we considered the \lq conditional\rq\ random core $C$ conditioned on the event that its kernel is equal to $K$ and its subdivision number is equal to $k$. However, we are actually interested in the \lq unconditional\rq\ random core $\core{\planarRandomGraph}$ of a random planar graph $\planarRandomGraph = \planarRandomGraph(n,m)\ur \planarClass(n,m)$ for some function $m=m(n)$. In this section we describe a method how to obtain results on an \lq unconditional\rq\ random graph by studying the corresponding \lq conditional\rq\ random graphs (see \Cref{lem:splitting}). To do so, we need the following definition.

\begin{definition}\label{def:feasible}
Given a class $\cl$ of graphs, a set $\mathcal{S}$, and a function $\func:\cl\to \mathcal{S}$, we call a sequence $\seq=(s_n)_{n\in \N}$ {\em feasible} for $\left(\cl, \func\right)$ if for each $n \in \N$ there exists a graph $H \in \cl(n)$ such that $\func(H)=s_n$. Moreover, for each $n\in \N$ we denote by $\left(\condGraph{\randomGraph}{\seq}\right)(n)$ a graph chosen uniformly at random from the set $\left\{H \in \cl(n): \func(H)=s_n\right\}$. We will often omit the dependence on $n$ and write just $\condGraph{\randomGraph}{\seq}$ (i.e. \lq $\randomGraph$ conditioned on $\seq$\rq) instead of $\left(\condGraph{\randomGraph}{\seq}\right)(n)$.
\end{definition}
\begin{lem}\label{lem:splitting}
Let $\cl$ be a class of graphs, $\mathcal{S}$ a set, $\func:\cl\to \mathcal{S}$ a function, and $\property$ a graph property. Let $\randomGraph=\randomGraph(n)\ur \cl(n)$. If for every sequence $\seq=(s_n)_{n\in \N}$ that is feasible for $\left(\cl, \func\right)$ we have \whp\ $\condGraph{\randomGraph}{\seq} \in \property$, then we have \whp\ $\randomGraph \in \property$.
\end{lem}

The proof of \Cref{lem:splitting} is provided in \Cref{sec:proofs_aux}. In the following we illustrate how one can use \Cref{lem:splitting} to deduce that a graph property $\property$ holds \whp\ in the core $\core{\planarRandomGraph}$ of $\planarRandomGraph=\planarRandomGraph(n,m)$ by studying the random core $C(K,k)$, which is obtained by randomly subdividing the edges of a kernel $K$ with $k$ additional vertices. We start with a \lq strong\rq\ property $\mathcal{T}$ that is \whp\ satisfied by the kernel $\kernel{\planarRandomGraph}$ and the subdivision number $\subdivisionNumber{\planarRandomGraph}$, e.g. $\numberVertices{\kernel{\planarRandomGraph}}$ and $\subdivisionNumber{\planarRandomGraph}$ lie in certain intervals (see \Cref{thm:internal_structure}). Then we let $\cl(n)\subseteq\planarClass(n,m)$ be the subclass of all graphs in $\planarClass(n,m)$ fulfilling $\mathcal{T}$ and $\cl:=\bigcup_{n\in\N}\cl(n)$. We define the function $\func: \cl \to \mathcal{K}\times \N$ by $\func(H):=\left(\kernel{\fixGraph}, \subdivisionNumber{\fixGraph}\right)$ and let $\seq=(K_n, \nd_n)_{n\in \mathbb N}$ be a sequence that is feasible for $\left(\cl, \func\right)$. The core $\core{\condGraph{\randomGraph}{\seq}}$ of the conditional random graph $\condGraph{\randomGraph}{\seq}$ is distributed like $C\left(K_n,k_n\right)$ (see \Cref{rem:core_equal} for details). Now the main step is to show that \whp\ $C\left(K_n,k_n\right)\in\property$. This is usually much easier than proving \whp\ $\core{\planarRandomGraph}\in\property$ directly, as the kernel and subdivision number are not random anymore in $C\left(K_n,k_n\right)$ and furthermore, fulfil property $\mathcal{T}$. Knowing that \whp\ $C\left(K_n,k_n\right)\in\property$, it follows by \Cref{lem:splitting} that \whp\ also the core $\core{\randomGraph}$ of the random graph $\randomGraph=\randomGraph(n)\ur\cl(n)$ satisfies $\property$. Finally, this implies that \whp\ $\core{\planarRandomGraph}\in \property$ as desired, because \whp\ $\planarRandomGraph\in\mathcal{A}$ by definition of $\cl$. Applications of \Cref{lem:splitting} can be found e.g. in the proofs of \Cref{thm:general,thm:block_structure_general}.
 
\section{\Pl\ classes of graphs}\label{sec:planar_like}
In this section we will show that \Cref{thm:main1} holds for a more general class of graphs, called a \pl\ class, in which graphs satisfy certain properties that are extracted from the class of planar graphs and are essential for the aforementioned core-kernel approach. Before defining the \pl\ class, we first recall well-known classes of graphs (see e.g. \cite{kp, msw, msw1, bbg} for details). A class $\generalGraphClass$ of graphs is called
\begin{itemize}
 \item {\em weakly addable} (also known as {\em bridge-addable}) if it is closed under adding an edge between two components;
 \item {\em addable} if $\generalGraphClass$ is weakly addable and fulfils in addition the property that a graph $\fixGraph$ is in $\generalGraphClass$ if and only if all components of $\fixGraph$ are in $\generalGraphClass$;
 \item {\em closed under taking minors} if each minor of a graph $H \in \generalGraphClass$ is again in $\generalGraphClass$.
\end{itemize}
Furthermore, we assume that all considered graph classes are closed under isomorphism, i.e. closed under relabelling of vertices.

\begin{definition}\label{def:planar_like}
A class $\planarClass$ of graphs is called {\em \pl} if it satisfies the following properties.
	\begin{itemize}[labelsep=0.5mm]
		\item[(P1)] {\bf [global]}. The class $\planarClass$ is weakly addable and closed under taking minors. \\ 
		\item[(P2)] {\bf [kernel]}. Let $\mathcal{K}$ be the class of all kernels of graphs in $\planarClass$ and $\mathcal{K}_C$ be the subclass of $\mathcal{K}$ containing all {\em connected} kernels. Then $\mathcal{K}$ and $\mathcal{K}_C$ satisfy the following conditions. 
		\begin{itemize}
			\item[(K1)] {\bf [stability]}. 
			A graph is in $\planarClass$ if and only if its kernel is in $\mathcal{K}$.
			\item[(K2)] {\bf [asymptotic behaviour]}. Let $\mathcal{K}(2n,3n)$ be the subclass of $\mathcal{K}$ consisting of all kernels on vertex set $[2n]$ having $3n$ edges. Then there exist constants $\gamma>0$, $c\geq c_1>0$, and $\ce \in \R$ such that 
			\begin{align*}
			|\mathcal{K}(2n, 3n)|&=(1+o(1))cn^{-\ce}\gamma^{n}(2n)! \\
			\text{and} \quad 
			|\mathcal{K}_C(2n, 3n)|&=(1+o(1))c_1n^{-\ce}\gamma^{n}(2n)!.
			\end{align*}
		\item[(K3)] {\bf [giant component]}. Let $K(2n,3n) \ur \mathcal{K}(2n, 3n)$. Then $\numberVertices{\largestComponent{K(2n,3n)}}=2n-O_p\left(1\right)$. In addition, for each $i \in \N$, the asymptotic probability that $\numberVertices{\largestComponent{K(2n,3n)}}=2n-2i$ is bounded away from both 0 and 1. 
 		\end{itemize}
	\end{itemize}
The constant $\ce$ in (K2) is called the {\em critical exponent}. For short, we say $\planarClass$ is {\em \pl\ with critical exponent $\ce$} if it satisfies (P1) and (P2).
\end{definition}

We can show that condition (K3) in \Cref{def:planar_like} can be deduced from conditions (K1) and (K2) if $\planarClass$ is addable. In addition, each graph without complex components is in any \pl\ class. The proofs of \Cref{lem:k3,lem:k4} can be found in \Cref{sec:proofs_aux}.

\begin{lem}\label{lem:k3}
Let $\planarClass$ be a class of graphs satisfying (K1) and (K2). If, in addition, $\planarClass$ is addable, then $\planarClass$ satisfies (K3).
\end{lem}

\begin{lem}\label{lem:k4}
Let $\planarClass$ be a \pl\ class of graphs and $\fixGraph$ a graph without complex components. Then $\fixGraph \in \planarClass$.
\end{lem}

The next lemma indicates that a \pl\ class is quite universal and rich, because it includes the class of cactus graphs (a cactus graph is a graph in which every edge belongs to at most one cycle), the class of series-parallel graphs, the class of planar graphs, and the class of graphs on a surface.

\begin{lem}[\cite{planar, sparse_outerplanar, surface, master_thesis_mosshammer}]\label{lem:planar_like}
The following classes of graphs are \pl\ with critical exponent $\ce$:
\begin{enumerate}
\item the class of cactus graphs with $\ce=5/2$;
\item the class of series-parallel graphs with $\ce=5/2$;
\item \label{lem:planar_like_1} the class of planar graphs with $\ce=7/2$;
\item the class of graphs embeddable on an orientable surface of genus $g \in \N$ with $\ce=-5g/2+7/2$.
\end{enumerate}
\end{lem}

\begin{proof}[Proof of \Cref{lem:planar_like}]
As shown in \cite{planar, master_thesis_mosshammer, sparse_outerplanar}, the class of planar graphs \cite{planar}, the class of series-parallel graphs \cite{master_thesis_mosshammer}, and the class of cactus graphs \cite{sparse_outerplanar} satisfy (K2).
 
Obviously, these classes also fulfil (P1), (K1) and are addable. Thus, they are \pl\ classes due to \Cref{lem:k3}. Moreover, in \cite{surface} it was shown that the class of graphs that are embeddable on an orientable surface of genus $g\in \N\cup\{0\}$ satisfies (P2). Thus, they also form a \pl\ class of graphs, since they trivially fulfil (P1).
\end{proof}

Instead of proving \Cref{thm:main1} only for the class of planar graphs, we will show the following generalisation to \pl\ classes of graphs in \Cref{sec:proofs_main}. 
\begin{thm}\label{thm:general}
\Cref{thm:main1} is true for any \pl\ class of graphs.
\end{thm}

This immediately implies that \Cref{thm:main1} is also true for the graph classes in \Cref{lem:planar_like}.
\begin{coro}\label{coro:general}
	\Cref{thm:main1} is true for the class of cactus graphs, the class of series-parallel graphs, and the class of graphs embeddable on an orientable surface of genus $g \in \N\cup\{0\}$.
\end{coro}
In contrast to the classes of graphs in \Cref{coro:general}, the class of {\em outerplanar} graphs is {\em not} \pl, since subdividing an edge in an outerplanar graph can lead to a non-outerplanar graph. Hence, a non-outerplanar graph can have an outerplanar kernel, and thus (K1) is violated. Nevertheless, we can prove that \Cref{thm:main1} is also true for outerplanar graphs by using some results from \cite{sparse_outerplanar}: (i) for the cases $s^3 n^{-2} \to -\infty$ and $s=O\left(n^{2/3}\right)$ we use that the asymptotic probability that the uniform random graph $G(n,m)$ is outerplanar is bounded away from 0 (see \Cref{thm:critcalRegime}\ref{thm:critcalRegime_a}); (ii) if $s^3n^{-2}\to \infty$, we use the fact that a random outerplanar graph is \whp\ a cactus graph \cite[Theorem 4]{sparse_outerplanar}.
\begin{coro}\label{coro:outerplanar}
	\Cref{thm:main1} is true for the class of outerplanar graphs.
\end{coro}

In order to prove \Cref{thm:general} in \Cref{sec:proofs_main}, we will need the following two known facts. The first statement was shown in \cite{surface} by applying the core-kernel approach and provides useful information about a typical core and a typical kernel when $s^3n^{-2} \to \infty$. The later one deals with the cases $s^3n^{-2} \to -\infty$ and $s=O\left(n^{2/3}\right)$.

\begin{thm}[{\cite{surface}}]\label{thm:internal_structure}
Let $\planarClass$ be a \pl\ class of graphs, $P=\planarRandomGraph(n,m) \ur \mathcal \planarRandomGraph(n,m)$, $\LargestComponent=\largestComponent{\planarRandomGraph}$ the largest component of $P$, and $\rest=\Rest{\planarRandomGraph}=P\setminus L_1$. Assume $m=n/2+s$ for $s=s(n)=o(n)$ and $s^3n^{-2} \to \infty$. Then the following hold:
	\begin{enumerate}
		\item \label{thm:internal_structure_a}
		\whp\ $\numberVertices{\core{\LargestComponent}}=\Theta\left(sn^{-1/3}\right);$
		\item \label{thm:internal_structure_b}
		\whp\ $\numberVertices{\kernel{\LargestComponent}}=\Theta\left(sn^{-2/3}\right);$
		\item \label{thm:internal_structure_c}
		$\numberVertices{\core{{\rest}}}=O_p\left(n^{1/3}\right)$;
		\item \label{thm:internal_structure_d}
		$\numberVertices{\kernel{{\rest}}}=O_p\left(1\right)$;	
		\item \label{thm:internal_structure_e}
		\whp\ $\kernel{\planarRandomGraph}$ is cubic (i.e. 3-regular);
		\item \label{thm:internal_structure_f}
		$\numberVertices{\LargestComponent}=2s+O_p\left(n^{2/3}\right)$;
		\item \label{thm:internal_structure_g}
		$\numberEdges{\LargestComponent}=2s+O_p\left(n^{2/3}\right)$.	
	\end{enumerate}
\end{thm}

We note that the results in \Cref{thm:internal_structure} were not explicitly proven in \cite{surface}, but they immediately follow by combining Theorems 1.4(iii) and 5.4(iii), (v), (vi) and Corollaries 5.3 and 5.5 from \cite{surface}. Strictly speaking, the authors of \cite{surface} proved \Cref{thm:internal_structure} only for the class of graphs embeddable on an orientable surface of genus $g \in \N\cup\{0\}$, but they pointed out that \Cref{thm:internal_structure} generalises to the more general setting of \pl\ graph classes (see \cite[Remark 8.3]{surface}).

\begin{thm}[\cite{uni,luczak_pittel_wierman}]\label{thm:critcalRegime}
	Let $m=n/2+s$, where $s\leq \constant n^{2/3}$ for some constant $\constant \in \R$ and let $G=G(n,m)\ur \mathcal G(n,m)$ be the uniform random graph. Then the following hold. 
	\begin{enumerate}
		\item \label{thm:critcalRegime_a}
		$
		\liminf_{n \to \infty} \prob{G \text{ has no complex component}}>0;
		$
		\item \label{thm:critcalRegime_b}
		$\numberVertices{\largestComponent{G}}=O_p\left(n^{2/3}\right)$.
	\end{enumerate}
\end{thm}
We note that if $\planarClass$ is a \pl\ class of graphs, then each graph without a complex component lies in $\planarClass$ (see \Cref{lem:k4}). Thus, \Cref{thm:critcalRegime}\ref{thm:critcalRegime_a} implies $\liminf_{n \to \infty} \prob{G(n,m) \in \planarClass}>0$ in the case of $s\leq Mn^{2/3}$. That will be useful in the proof of \Cref{thm:general}.

\section{P\'olya urn model}\label{sec:polya}
In this section we present several useful results on the P\'olya urn model introduced in \Cref{sub:construction_kernel_core}. 

\subsection{Model with $N$ colours}
Given $N,k\in \mathbb N$, there are initially $\nc$ balls of $\nc$ distinct colours $\colour_1, \ldots, \colour_{\nc}$ in a urn. In each step we draw a ball uniformly at random from the urn. Then the drawn ball is returned to the urn along with an additional ball of the same colour. We repeat that procedure $\nd$ times. For each $\ind\in [\nc]$ we denote by $X_\ind$ the number of drawn balls of colour $\colour_\ind$ at the end of the procedure (i.e. after $\nd$ steps).

To derive bounds on the length of the shortest and longest cycle in the core (see \Cref{lem:random_core2}), we need the following bounds on the minimum and maximum values of the total numbers $X_1, \ldots, X_{f}$ of drawn balls of the first $f$ colours when $N$, $k$, and $f$ are functions in $n$. Although we believe such results should be known, we could not find them in literature and therefore we include their proofs in \Cref{sec:polya_proofs} for completeness.
 
\begin{thm}\label{main1}
For every $n \in \N$, we let $\nc=\nc(n), \nd=\nd(n) \in \N$, and $f=f(n) \in \N$ with $1\leq f \leq \nc$. We assume that $\nc=\omega(1)$ and that $f=O(1)$ or $f=\omega(1)$. Then the following hold.
	\begin{enumerate}
		\item\label{thm:main1_1}
		\begin{equation*}
	\hspace{-7ex}	X_*\, :=\, \min_{1\leq \ind \leq f}X_\ind\ =\ 
		\begin{cases}
		\Theta_p\left(\frac{\nd}{\nc f}\right) &\quad \text{if}~~ \nd=\omega\left(\nc f\right),
		\\
		O_p(1) &\quad \text{if}~~ \nd=O\left(\nc f\right).
		\end{cases}
		\end{equation*}
		\item \label{thm:main1_2}
		\begin{equation*}
		X^*\, :=\, \max_{1\leq \ind \leq f}X_\ind\ =\ 
		\begin{cases}
		\Theta_p\left(\frac{\nd}{\nc}\left(1+\log f\right)\right) &\quad \text{if}~~ \nd=\omega\left(\nc\right),
		\\
		O_p(1+\log f) &\quad \text{if}~~ \nd=O\left(\nc\right).
		\end{cases}
		\end{equation*}
	\end{enumerate}
\end{thm}

Another useful fact about the P\'olya urn model (which will be used in the proof of \Cref{lem:random_core}\ref{lem:random_core_a}) is the following result on the distribution of $X_i$, whose proof can be found in \Cref{sec:polya_proofs}.

\begin{prop}\label{prop:polya_main}
	Let $N,k\in \mathbb N$ be given. Then we have 
		\begin{align*}
		\sum_{\ind=1}^{\nc}\prob{X_\ind \leq 1}\ \leq\ \frac{2\nc^2}{\nd}.
		\end{align*}
\end{prop}

\smallskip
\subsection{Model with two colours}\label{sub:polya2}
Given $\bb, \wb, \nd \in \N$, there are initially $\bb$ black and $\wb$ white balls in a urn. Then we draw $\nd$ times uniformly at random a ball from the urn. In each step we return the drawn ball together with an additional ball of the same colour. Then we denote by $X$ the number of drawn black balls at the end of the procedure, i.e. after $\nd$ steps. 

\begin{thm}[\cite{eggenberger_polya}]\label{thm:polya2}
Let $\bb, \wb, \nd \in \N$ and $X$ be the number of drawn black balls after $k$ steps in the P\'olya urn model with initially $\bb$ black and $\wb$ white balls. Then we have
\begin{align*}
\expec{X}&=\frac{\bb\nd}{\bb+\wb} \quad \text{and}\quad 
\variance{X}=\frac{\bb\wb\nd\left(\bb+\wb+\nd\right)}{\left(\bb+\wb\right)^2\left(\bb+\wb+1\right)}.
\end{align*}
\end{thm}

\section{Random core}\label{sec:CoreLabelled}
In this section we study the process of obtaining a random core $C=C(K,k)$ from a fixed kernel $K$ by {\em randomly} subdividing the edges of $K$ with $\nd$ additional vertices for given $(K,k)$. Because it is hard to directly analyse $C$, we circumvent this difficulty by introducing an auxiliary random multigraph $\randomMultiCore$ which behaves asymptotically like $C$ and fits into the scheme of the P\'olya urn model.

\begin{definition}\label{def:random_core}
Given a pair $(K,\nd)$ of a multigraph $K$ on vertex set $[\numberVertices{K}]$ of minimum degree at least three and $\nd\in \N$, we denote by $\mathcal{C}(K,\nd)$ the set of all {\em simple} graphs on vertex set $[\numberVertices{K}+\nd]$ obtained from $K$ by subdividing the edges of $K$ by the vertices $\numberVertices{K}+1,\dots, \numberVertices{K}+k$. In other words, $\mathcal{C}(K,\nd)$ is the set of all cores on vertex set $[\numberVertices{K}+\nd]$ whose kernel is equal to $K$. Let $\randomCore=\randomCore(K,\nd)\ur \mathcal{C}(K,\nd)$. In addition, we define a random {\em multigraph} $\randomMultiCore=\randomMultiCore(K,\nd)$ by the following random experiment: we start with $G_0=K$. Given the multigraph $G_{i-1}$ we construct $G_{i}$ as follows (for $i=1, \ldots, \nd$). We choose uniformly at random an edge $e$ of $G_{i-1}$ and subdivide $e$ by one additional vertex, which obtains the label $\numberVertices{K}+i$. We note that $\edgeSet{G_{i-1}}$ is a multiset, i.e. if there are $r$ edges between vertices $v$ and $w$, we choose one of these edges with probability $r/\numberEdges{G_{i-1}}$. Then we let $\randomMultiCore=G_{\nd}$ be the resulting multigraph after $k$ steps.
\end{definition}

Later we will study a random \pl\ graph $\planarRandomGraph$ conditioned on the event that $\kernel{\planarRandomGraph}=K$ and $\subdivisionNumber{\planarRandomGraph}=\nd$ for some fixed kernel $K$ and fixed $k \in \N$. The next lemma says that the core of this \lq conditional\rq\ random graph is distributed like $\randomCore(K, \nd)$ from \Cref{def:random_core}. That fact will be quite useful when we apply \Cref{lem:splitting}.

\begin{lem}\label{rem:core_equal}
Let $\planarClass$ be a \pl\ class of graphs and $\mathcal{K}$ the class of all kernels of graphs in $\planarClass$. Given a pair $(K, \nd)$ of a multigraph $K\in \mathcal{K}$ on vertex set $[\numberVertices{K}]$ and $\nd \in \N$, we let $\widehat \planarClass_{(K, \nd)}(n,m)$ be the subclass of $\planarClass(n,m)$ consisting of all graphs $F$ whose kernel $K(F)$ is equal to $K$ and whose subdivision number $S(F)$ is equal to $k$, i.e. 
\[\widehat \planarClass_{(K, \nd)}(n,m):=\left\{F\in \planarClass(n,m) \ \mid\ \kernel{F}=K \ \ \text{and}\ \ \subdivisionNumber{F}=\nd\right\}.\]
Let $\condGraph{\planarRandomGraph}{\left(K, \nd\right)}\ \ur \ \widehat \planarClass_{(K, \nd)}(n,m)$ and $\randomCore(K, \nd)$ be the random core as defined in \Cref{def:random_core}. Then the core of $\condGraph{\planarRandomGraph}{\left(K, \nd\right)}$ is distributed like $\randomCore(K, \nd)$: for each fixed graph $\fixGraph$, we have
\[
\prob{\core{\condGraph{\planarRandomGraph}{\left(K, \nd\right)}}=\fixGraph}\ =\ \prob{\randomCore(K, \nd)=\fixGraph}.
\]
\end{lem}

In the next step, we provide a relation between the two random (multi-) graphs $\randomCore$ and $\randomMultiCore$ (see \Cref{lem:random_core}). In particular, we show that if $\nd=\omega\left(\nc^2\right)$, they are contiguous in the sense of \Cref{def:contiguous} (see \Cref{cor:contiguous}), where $\nc=\numberEdges{K}$ is the number of edges in $K$. In our applications we will have $k=\Theta\left(sn^{-1/3}\right)$ and $\nc=\Theta\left(sn^{-2/3}\right)$ (see \Cref{thm:internal_structure}) and we note that $sn^{-1/3}=\omega\left(\left(sn^{-2/3}\right)^2\right)$. In order to state this result, we need the following definition.

\begin{definition}[\2s]
Given a graph $\fixGraph$ we consider the construction of obtaining the core $\core{\fixGraph}$ from the kernel $\kernel{\fixGraph}$ (described in \Cref{sub:dec}). We call $\fixGraph$ {\em \2s} if each edge of $\kernel{\fixGraph}$ is subdivided by at least two vertices in that construction.
\end{definition}

\begin{lem}\label{lem:random_core}
Let $K$ be a multigraph on vertex set $[\numberVertices{K}]$ of minimum degree at least three with $\nc$ edges and let $\nd \in \N$. In addition, let the random multigraphs $\randomCore=\randomCore(K, \nd)$ and $\randomMultiCore=\randomMultiCore(K,\nd)$ be defined as in \Cref{def:random_core}. 
\begin{enumerate}
	\item \label{lem:random_core_a}
We have 
	\[ \prob{\randomMultiCore\text{ is simple}} \ \geq\ \prob{\randomMultiCore\text{ is 2-simple}}\ \geq\ 1-\frac{2\nc^2}{\nd}.\]
	\item \label{lem:random_core_b}
	Conditioning on the event that $\randomMultiCore$ is simple, the distributions of $\randomMultiCore$ and $\randomCore$ are the same: for each graph $H$, we have
	\[
	\prob{\randomMultiCore=H\;\middle|\;\randomMultiCore\text{ is simple}}\ =\ \prob{\randomCore=H}.
	\]
\end{enumerate}
\end{lem}
\begin{coro}\label{cor:contiguous}
For every $n \in \N$, let $K=K(n)$ be a multigraph on vertex set $[\numberVertices{K}]$ of minimum degree at least three with $\nc=\nc(n)$ edges and let $\nd=\nd(n) \in \N$. Let $\randomCore=\randomCore(n)=\randomCore(K(n),k(n))$ and $\randomMultiCore=\randomMultiCore(n)=\randomMultiCore(K(n),k(n))$ be as in \Cref{def:random_core}. If $\nd=\omega\left(\nc^2\right)$, then $\randomCore$ and $\randomMultiCore$ are contiguous.
\end{coro}

Amongst others, \Cref{lem:random_core}\ref{lem:random_core_a} states that \whp\ $\randomMultiCore$ is \2s\ if $\nd=\omega\left(\nc^2\right)$. Using that we can deduce the following result, which we will use later in the proof of \Cref{lem:kernel_cubic}.
\begin{coro}\label{cor:subdivision_twice}
Let $\planarClass$ be a \pl\ class of graphs and $P=\planarRandomGraph(n,m) \ur \mathcal \planarRandomGraph(n,m)$.
Assume $m=n/2+s$ for $s=s(n)=o(n)$ and $s^3n^{-2} \to \infty$. Then \whp\ $\planarRandomGraph$ is \2s.
\end{coro}

Our next results provide bounds on the lengths $\girth{\randomMultiCore}$ and $\longestCycle{\randomMultiCore}$ of the shortest and longest cycle in $\randomMultiCore$. We note that if $\nd=\omega\left(\nc^2\right)$, these results also hold for $\randomCore$ due to \Cref{cor:contiguous}.

\begin{lem}\label{lem:random_core2}
For every $n \in \N$, let $K=K(n)$ be a multigraph on vertex set $[\numberVertices{K}]$ of minimum degree at least three with $\nc=\nc(n)$ edges and $\NumberLoops=\NumberLoops(n)$ loops. In addition, let $\nd=\nd(n) \in \N$ and the random multigraph $\randomMultiCore=\randomMultiCore(n)=\randomMultiCore(K(n),k(n))$ be as in \Cref{def:random_core}. We assume that $\nc=\omega(1)$.
\begin{enumerate}
	\item\label{lem:random_core2_a}
If $\NumberLoops=O(1)$ or $\NumberLoops=\omega(1)$, then
		\begin{align*}
		\girth{\randomMultiCore}&=
		\begin{cases}
		O_p\left(\frac{\nd}{\nc \NumberLoops}\right) &\text{if}~~ \nd=\omega\left(\nc \NumberLoops\right) \text{ and } \NumberLoops\neq 0,
		\\
		O_p(1) &\text{if}~~ \nd=O\left(\nc \NumberLoops\right),
		\end{cases}
		\\
		\text{ and }~~~~~\girth{\randomMultiCore}&=
		\Omega_p\left(\frac{\nd}{\nc^2}\right) ~~~\text{if}~~ \nd=\omega\left(\nc ^2\right).~~~~~~~~~~~~~~~~~~
		\end{align*}
		In particular, if $\NumberLoops=\Theta\left(\nc\right)$, then
		\begin{equation*}
		\girth{\randomMultiCore}=
		\begin{cases}
		\Theta_p\left(\frac{\nd}{\nc^2}\right) & \text{if}~~ \nd=\omega\left(\nc^2\right),
		\\
		O_p(1) & \text{if}~~ \nd=O\left(\nc ^2\right).
		\end{cases}
		\end{equation*}
		\item \label{lem:random_core2_b}
		If $\NumberLoops> 0$ and $\nd=\omega(\nc)$,
		then we have $\longestCycle{\randomMultiCore}=\Omega_p\left(\frac{\nd}{\nc}\left(1+\log\NumberLoops\right)\right)$.
	\end{enumerate}
\end{lem}

In order to use \Cref{lem:random_core2} in the proof of \Cref{thm:general}, we will study the number of loops in a typical kernel of a random \pl\ graph in \Cref{sec:random_kernel} (see \Cref{cor:linear_loop}). 

We conclude this section with a result on the block structure of $\randomMultiCore$, which provides also an insight into the block structure of $C$ due to \Cref{cor:contiguous}. Roughly speaking, the next lemma says that the $i$-th largest block of a cubic kernel $K$ translates (during the process of constructing $\randomMultiCore$) to the $i$-th largest block of $\randomMultiCore$, provided the block is not too \lq small\rq. 

\begin{lem}\label{lem:block_core}
	For every $n \in \N$, let $\nd=\nd(n) \in \N$, $K=K(n)$ be a cubic multigraph on vertex set $[\numberVertices{K}]$, and the random multigraph $\randomMultiCore=\randomMultiCore(n)=\randomMultiCore(K(n),k(n))$ be defined as in \Cref{def:random_core}. We assume that $k=\omega\left(\numberVertices{K}\right)$ and let $i\in \N$ be fixed such that $\blockOrder{i}{K}=\omega\left(\left(\numberVertices{K}\right)^{1/2}\right)$. Then \whp
	\begin{align*}
	\blockOrder{i}{\randomMultiCore}=\Theta\left(\frac{\nd \blockOrder{i}{K}}{\numberVertices{K}}\right).
	\end{align*}
\end{lem}

\begin{remark}\label{rem:improve}
	The condition in \Cref{lem:block_core} that $\blockOrder{i}{K}$ is not too small, i.e. $\blockOrder{i}{K}=\omega\left(\left(\numberVertices{K}\right)^{1/2}\right)$, can be weakened by using stronger concentration results on the P\'olya urn model than the results in \Cref{thm:polya2}. We believe that using similar results as in \Cref{main1} one can show that the statement of \Cref{lem:block_core} is true even under the condition that $\blockOrder{i}{K}=\omega\left(\left(\numberVertices{K}\right)^\varepsilon\right)$ for some $\varepsilon>0$.
\end{remark}

\section{Random kernel}\label{sec:random_kernel}
Throughout this section, let $\planarClass$ be a \pl\ class of graphs and $\mathcal{K}$ the class of all kernels of graphs in $\planarClass$. 
We let $\planarRandomGraph=\planarRandomGraph(n,m)\ur \planarClass(n,m)$ and consider the weakly supercritical regime when $m=n/2+s$ for $s=s(n)=o(n)$ and $s^3n^{-2} \to \infty$. 

Due to \Cref{thm:internal_structure} we know that \whp\ $\kernel{\planarRandomGraph}$ is cubic and $\numberVertices{\kernel{\planarRandomGraph}}=\Theta\left(sn^{-2/3}\right)$. Hence, we might expect that $\kernel{\planarRandomGraph}$ \lq behaves\rq\ asymptotically like a graph chosen uniformly at random from all cubic multigraphs in $\mathcal{K}$ with $\Theta\left(sn^{-2/3}\right)$ many vertices. In the next lemma we show that this is indeed true. We note, however, that this result is not straightforward, since $K(\planarRandomGraph)$ is not equally distributed on the set of all possible cubic kernels in $\mathcal{K}$ with $\Theta\left(sn^{-2/3}\right)$ many vertices.
\begin{lem}\label{lem:kernel_cubic}
	Let $\planarClass$ be a \pl\ class of graphs and $\mathcal{K}$ the class of all kernels of graphs in $\planarClass$. Let $F:\planarClass \to \N$ be a graph theoretic function and $g_1, g_2:\N \to \N$ non-decreasing functions. We assume that for $K(2n,3n)\ur \mathcal{K}(2n,3n)$, we have \whp\ $g_1(n)\leq F(K(2n,3n))\leq g_2(n)$. In addition, assume $m=n/2+s$ for $s=s(n)=o(n)$ and $s^3n^{-2} \to \infty$. Then there exist constants $c_2\geq c_1>0$ such that for $\planarRandomGraph=\planarRandomGraph(n,m)\ur \mathcal \planarRandomGraph(n,m)$ we have \whp
	\[
	g_1\left(c_1sn^{-2/3}\right)\ \leq\ F\left(\kernel{\planarRandomGraph}\right)\ \leq\ g_2\left(c_2sn^{-2/3}\right).
	\]
\end{lem}

We recall that $\numberLoops{H}$ denotes the number of loops in a multigraph $H$. Next, we will show that for $K(2n,3n)\ur \mathcal{K}(2n,3n)$, we have \whp\ $\numberLoops{K(2n,3n)}=\Theta(n)$, which implies \whp\ $\numberLoops{\kernel{\planarRandomGraph}}=\Theta\left(sn^{-2/3}\right)$ by \Cref{lem:kernel_cubic}. Our proof of that result will be based on the following observation. Let $H$ be a cubic multigraph with a single loop at $w$. Then $w$ has precisely one neighbour $x\neq w$. Assuming that there is no loop at $x$, there are two (not necessarily distinct) additional neighbours $y$ and $z$ of $x$. Then we obtain again a cubic multigraph if we delete $w$ and $x$ in $H$ and add an edge $yz$. We note that by reversing this operation we can create a multigraph with a loop at $w$. This reverse operation leads to the following definition (see also \Cref{fig:loopOp}).
\begin{definition}[\LoopInsertion]\label{def:loop_insertion}
	Let $H_1, H_2$ be multigraphs, $yz\in \edgeSet{H_1}$, and $(w, x)\in \left(\vertexSet{H_2}\right)^2$ with $w\neq x$. We say that $H_2$ can be obtained from $H_1$ by a {\em\loopInsertion}\ at edge $yz$ with vertex pair $(w,x)$ if 
	\begin{align*}
	\vertexSet{H_2}&=\vertexSet{H_1}~\dot\cup~\{w,x\}\\
\text{ and }~~~~~	\edgeSet{H_2}&=\edgeSet{H_1}-yz+xy+xz+wx+ww.~~~~~
	\end{align*}
\end{definition}

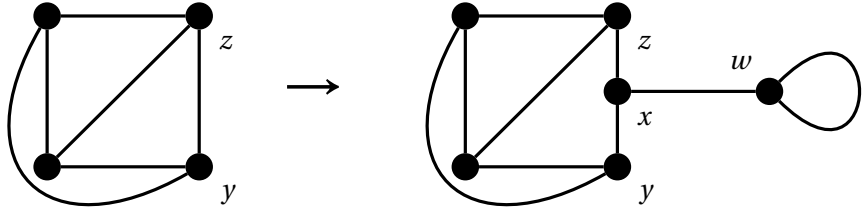
\begin{figure}[t]
	\centering
	\begin{tikzpicture}[scale=1, very thick, every node/.style={circle, fill=black}]
	\clip (-1cm,-0.52cm) rectangle (11.5cm,2.2cm);
	
	\node (1) at (0,0) [circle,draw]{};
	\node (2) at (0,2) [circle,draw]{};
	\node (3) at (2,0) [circle,draw, label=below right:$y$]{};
	\node (4) at (2,2) [circle,draw, label=below right:$z$]{};	
	
	\node (5) at (3.5,1) [circle, fill=none] {\Huge $\rightarrow$};
	
	\draw (1) to (2);
	\draw (1) to (3);
	\draw (1) to (4);
	\draw (2) -- (4);
	\draw (3) -- (4); 	
	\draw (2) to [out=240,in=210,looseness=2] (3);
	
	\node (6) at (5.5,0) [circle,draw]{};
	\node (7) at (5.5,2) [circle,draw]{};
	\node (8) at (7.5,0) [circle,draw, label=below right:$y$]{};
	\node (9) at (7.5,2) [circle,draw, label=below right:$z$]{};	
	\node (10) at (7.5,1) [circle,draw, label=below right:$x$]{};
	\node (11) at (9.5,1) [circle,draw, label=above left:$w$]{};	
	
	\draw (6) to (7);
	\draw (6) to (8);
	\draw (6) to (9);
	\draw (7) -- (9);
	\draw (8) to (10); 
	\draw (9) to (10);	
	\draw (7) to [out=240,in=210,looseness=2] (8);
	\draw (10) to (11);
	\draw (11) to [out=315,in=45,looseness=20] (11);
	
	\end{tikzpicture}
	\caption{A \loopInsertion\ performed at edge $yz$ with vertex pair $(w,x)$.}
	\label{fig:loopOp}
\end{figure}

Next, we state some simple observations about \loopInsertions.
\begin{remark}\label{lem:loop_insertion_stable}
	$\mathcal{K}$ is stable under \loopInsertions: If $H_2$ can be obtained by a \loopInsertion\ in $H_1$, then we have $H_1 \in \mathcal{K}$ if and only if $H_2 \in \mathcal{K}$.
\end{remark}

Note that subdividing an edge which occurs $r$ times in a multigraph $H$ increases the weight of $H$ by a factor of $r$. Therefore, we obtain the following.
\begin{remark}\label{prop:subdividing}
	Let $\fixGraph$ be a multigraph on vertex set $[n]$ and $\mathcal{\fixGraph}$ be the set of all multigraphs which can be obtained by subdividing precisely one edge of $\fixGraph$ by additional vertex $n+1$. Then we have $\left|\mathcal{\fixGraph}\right|=w(\fixGraph)\cdot \numberEdges{\fixGraph}$, where $w(\fixGraph)$ is the weight of $\fixGraph$ defined in (\ref{eq:weight}) and $\left|\mathcal{\fixGraph}\right|$ is the total weight of all multigraphs in $\mathcal{\fixGraph}$.
\end{remark}

As a consequence of \Cref{prop:subdividing} we obtain the number of ways to construct multigraphs by a loop insertion. 
\begin{remark}\label{prop:loop_insertion}
	Let $\fixGraph$ be a multigraph on vertex set $[n]$ and $\mathcal{\fixGraph}$ be the set of all multigraphs which can be obtained by performing a loop insertion (cf. \Cref{def:loop_insertion}) at some edge $e\in \edgeSet{\fixGraph}$ with vertex pair $(n+1,n+2)$. Then we have $\left|\mathcal{\fixGraph}\right|=w(\fixGraph)\cdot \numberEdges{\fixGraph}/2$, where $w(\fixGraph)$ is the weight of $\fixGraph$ defined in (\ref{eq:weight}).
\end{remark}
In order to see that \Cref{prop:loop_insertion} is true, we imagine that we perform a loop insertion in two steps. Firstly, we choose the edge $e\in \edgeSet{\fixGraph}$ and subdivide it by the vertex $n+1$. Secondly, we add the vertex $n+2$ together with an edge $\left\{n+1, n+2\right\}$ and a loop at $n+2$. Then \Cref{prop:loop_insertion} follows from \Cref{prop:subdividing} and the fact that inserting a loop halves the weight of a graph.

Next, we show that typically the number of loops in the kernel is linear (in the number of edges in the kernel). Due to \Cref{lem:kernel_cubic} it suffices to prove this result only for a random cubic kernel chosen from $\mathcal{K}(2n,3n)$. We use a \loopInsertion\ to construct all graphs in $\mathcal{K}(2n,3n)$ with a loop at a fixed vertex $v$. By doing that we can estimate the expected number of loops. Then we use the second moment method to show concentration around the mean. Recall that the number of loops in a graph $\fixGraph$ is denoted by $\numberLoops{\fixGraph}$.

\begin{lem}\label{thm:linear_loop}
	Let $\planarClass$ be a \pl\ class of graphs, $\mathcal{K}$ the class of all kernels of graphs in $\planarClass$, and $K(2n,3n)\ur \mathcal{K}(2n,3n)$. Then \whp
	\begin{align*}
	\numberLoops{K(2n,3n)}=(1+o(1)) \frac{3}{2\gamma}n,
	\end{align*}
	where $\gamma>0$ is as in \Cref{def:planar_like}.
\end{lem}

\begin{coro}\label{cor:linear_loop}
	Let $\planarClass$ be a \pl\ class of graphs and $P=\planarRandomGraph(n,m) \ur \mathcal \planarRandomGraph(n,m)$.
	Assume $m=n/2+s$ for $s=s(n)=o(n)$ and $s^3n^{-2} \to \infty$. Then \whp
	\begin{align*}
	\numberLoops{\kernel{\planarRandomGraph}}=\Theta\left(sn^{-2/3}\right).
	\end{align*}
	Furthermore, the largest component $\LargestComponent=\largestComponent{\planarRandomGraph}$ of $\planarRandomGraph$ satisfies \whp
	\begin{align*}
	\numberLoops{\kernel{\LargestComponent}}=\Theta\left(sn^{-2/3}\right).
	\end{align*}
\end{coro}

\section{Block structure}\label{sec:block_structure}
In this section we present several results which lead to the conclusion that \Cref{thm:block_structure} is also true for all \pl\ classes of graphs $\planarClass$ that are addable and have a critical exponent $3<\ce<4$ (see \Cref{thm:block_structure_general}). First we will analyse the block structure of a random cubic multigraph chosen from an appropriate class of multigraphs $\bridgeClass$ (the so-called \bridgeStable\ class), which will be later chosen as the class of connected kernels of graphs in $\planarClass$. Then we will deduce the block structure of the random \pl\ graph in \Cref{thm:block_structure_general} by using \Cref{lem:block_core,lem:kernel_cubic}. 

We recall that, as in the case of a simple graph, a block of a multigraph $H$ is a maximal 2-connected subgraph of $H$. Here we insist that a vertex with a loop forms a block. In order to understand the block structure of a random cubic multigraph chosen from the class $\bridgeClass$, we will first analyse \bridgeNumbers\ defined below (see also \Cref{fig:bridgeIns}). 

\begin{definition}[\BridgeNumber]\label{def:bridge_number}
	Given a connected multigraph $\fixMultigraph$, an edge $\bridge \in \edgeSet{\fixMultigraph}$ is called a {\em bridge} if $\fixMultigraph-\bridge$ is disconnected. For two distinct vertices $w, x \in \vertexSet{\fixMultigraph}$ we define the \bridgeNumber\ $\bn{wx}$ as follows. If $wx$ is a bridge in $\fixMultigraph$, then we set $\bn{wx}=\Bn{\fixMultigraph}{wx}:=\min\left\{\numberVertices{C_1}, \numberVertices{C_2}\right\}$, where $C_1$ and $C_2$ are the two components of $\fixMultigraph-wx$. Otherwise, we define $\bn{wx}:=0$.
\end{definition}

We observe that if $wx\notin \edgeSet{\fixMultigraph}$, then we have $\bn{wx}=0$. We will determine the distribution of the \bridgeNumber\ $\bn{wx}$ for fixed vertices $w$ and $x$. Intuitively, if $\bn{wx}$ is typically \lq small\rq, then we should have a unique largest block which is significantly larger than all other blocks. We will show that this is indeed the case (see \Cref{lem:bridge_number6,lem:bridge_number7}). To do so, we introduce the \bridgeInsertion, which is a operation that connects two components of a graph via a bridge (see also \Cref{fig:bridgeIns}).

\begin{definition}[\BridgeInsertion]\label{def:bridge_insertion}
	Let $H_1, H_2, H_3$ be connected multigraphs, $e_1=y_1z_1\in \edgeSet{H_1}, e_2=y_2z_2\in \edgeSet{H_2}$, and $w, x\in \vertexSet{H_3}$ be distinct. Then we say that $H_3$ can be obtained from $H_1$ and $H_2$ by a {\em \bridgeInsertion}\ at edges $e_1$ and $e_2$ with vertices $w$ and $x$ if 
	\begin{align*}
	\vertexSet{H_3}&=\vertexSet{H_1}~\dot\cup~\vertexSet{H_2}~\dot\cup~\{w,x\}\\
	\text{ and }~~~~~\edgeSet{H_3}&=\edgeSet{H_1}~\dot\cup~\edgeSet{H_2}-y_1z_1-y_2z_2+wy_1+wz_1+wx+xy_2+xz_2.~~~~~
	\end{align*}
\end{definition}

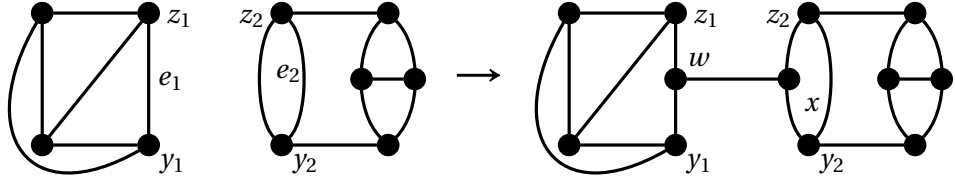
\begin{figure}[t]
	\centering
	\begin{tikzpicture}[scale=0.7, very thick, every node/.style={circle, fill=black}]
	\clip (-1cm,-0.62cm) rectangle (17.5cm,2.9cm);
	
	\node (1) at (0,0) [circle,draw, scale=0.8]{};
	\node (2) at (0,2.5) [circle,draw, scale=0.8]{};
	\node (3) at (2,0) [circle,draw, scale=0.8, label={[xshift=0.3cm, yshift=-0.8cm]$y_1$}]{};
	\node (4) at (2,2.5) [circle,draw, scale=0.8, label={[xshift=0.4cm, yshift=-0.6cm]$z_1$}]{};
	\node (5) at (4.5,0) [circle,draw, scale=0.8, label={[xshift=0.3cm, yshift=-0.8cm]$y_2$}]{};
	\node (6) at (4.5,2.5) [circle,draw, scale=0.8, label={[xshift=-0.4cm, yshift=-0.6cm]$z_2$}]{};
	\node (7) at (6.5,0) [circle,draw, scale=0.8]{};
	\node (8) at (6.5,2.5) [circle,draw, scale=0.8]{};	
	\node (10) at (9.9,0) [circle,draw, scale=0.8]{};
	\node (11) at (9.9,2.5) [circle,draw, scale=0.8]{};
	\node (12) at (11.9,0) [circle,draw, label={[xshift=0.3cm, yshift=-0.8cm]$y_1$}, scale=0.8]{};
	\node (13) at (11.9,2.5) [circle,draw, label={[xshift=0.4cm, yshift=-0.6cm]$z_1$}, scale=0.8]{};
	\node (14) at (14.4,0) [circle,draw, label={[xshift=0.3cm, yshift=-0.8cm]$y_2$}, scale=0.8]{};
	\node (15) at (14.4,2.5) [circle,draw, label={[xshift=-0.4cm, yshift=-0.6cm]$z_2$}, scale=0.8]{};
	\node (16) at (16.4,0) [circle,draw, scale=0.8]{};
	\node (17) at (16.4,2.5) [circle,draw, scale=0.8]{};		
	\node (18) at (11.9,1.25) [circle,draw, label={[xshift=0.3cm, yshift=-0.25cm]$w$}, scale=0.8]{};
	\node (19) at (14.03,1.25) [circle,draw, label={[xshift=0.3cm, yshift=-0.8cm]$x$}, scale=0.8]{};	
	
	\node (20) at (6,1.25) [circle,draw, scale=0.8]{};
	\node (21) at (7,1.25) [circle,draw, scale=0.8]{};
	\node (22) at (15.9,1.25) [circle,draw, scale=0.8]{};
	\node (23) at (16.9,1.25) [circle,draw, scale=0.8]{};		
	
	\node (9) at (8.2,1.25) [circle, fill=none] {\huge $\rightarrow$};
	
	\draw (1) to (2);
	\draw (1) to (3);
	\draw (1) to (4);
	\draw (2) -- (4);
	\draw (3) -- (4) node [fill=none, below=0.9cm, right=-0.1cm] {$e_1$}; 	
	\draw (2) to [out=240,in=210,looseness=2] (3);
	\draw (5) to (7);
	\draw (6) to (8);
	\draw [out=135,in=225,looseness=0.6] (5) to (6) node [fill=none, below=0.8cm, right=-0.3cm] {$e_2$};
	\draw (5) to [out=45,in=315,looseness=0.6] (6);
	\draw (7) to [out=135,in=225,looseness=0.8] (8);
	\draw (7) to [out=45,in=315,looseness=0.8] (8);
	\draw (10) to (11);
	\draw (10) to (12);
	\draw (10) to (13);
	\draw (11) -- (13);
	\draw (12) -- (13); 	
	\draw (11) to [out=240,in=210,looseness=2] (12);
	\draw (14) to (16);
	\draw (15) to (17);
	\draw (14) to [out=135,in=225,looseness=0.6] (15);
	\draw (14) to [out=45,in=315,looseness=0.6] (15);
	\draw (16) to [out=135,in=225,looseness=0.8] (17);
	\draw (16) to [out=45,in=315,looseness=0.8] (17);
	\draw (18) to (19);
	\draw (20) to (21);
	\draw (22) to (23);
	
	\end{tikzpicture}
	\caption{A \bridgeInsertion\ performed at edges $e_1$ and $e_2$ with vertices $w$ and $x$. The \bridgeNumber\ $\bn{wx}$ in the graph on the right-hand side is 5.}
	\label{fig:bridgeIns}
\end{figure}

In order to successfully apply \bridgeInsertions\ in some graph class $\bridgeClass$, we require two natural properties of $\bridgeClass$. 

\begin{definition}[\BridgeStable]\label{def:bridge_stable}
	A class $\bridgeClass$ of connected multigraphs is called {\em \bridgeStable} if the following two condition hold.
	\begin{itemize}
		\item[(B1)]
		$\bridgeClass$ is stable under \bridgeInsertions: if $H_3$ can be obtained by a \bridgeInsertion\ in $H_1$ and $H_2$, then we have $H_1, H_2 \in \bridgeClass$ if and only if $H_3 \in \bridgeClass$.
		\item[(B2)] Let $\bridgeClass(2n,3n)$ be the subclass of $\bridgeClass$ consisting of all connected multigraphs on vertex set $[2n]$ having $3n$ edges. Then there exist constants $\gamma>0$, $c>0$, and $\ce \in \R$ such that 
		\begin{align}\label{eq:formula_bridge}
		|\bridgeClass(2n, 3n)|\ =\ (1+o(1))cn^{-\ce}\gamma^{n}(2n)!.
		\end{align}
	\end{itemize}
	The constant $\ce$ in (B2) is called the {\em critical exponent}.
\end{definition}

We note that the class of all connected kernels of an addable \pl\ class is \bridgeStable. Next, let $\bridgeClass$ be some \bridgeStable\ class of multigraphs. We study the asymptotic distribution of the bridge number $\Bn{\bridgeGraph}{vw}$ for fixed vertices $v\neq w$ and $\bridgeGraph=\bridgeGraph(2n,3n) \ur \bridgeClass(2n,3n)$. More precisely, we estimate the probability $\prob{\Bn{\bridgeGraph}{wx}\geq 2f(n)+1}$ for an arbitrary function $f$. 

\begin{lem}\label{lem:bridge_number1}
	Let $\bridgeClass$ be a \bridgeStable\ class of multigraphs with critical exponent $\ce>2$ and $f$ a function with $f(n)=\omega(1)$, but $f(n)=o(n)$. In addition, let $\bridgeGraph=\bridgeGraph(2n,3n)\ur \bridgeClass(2n,3n)$ and let $c$ and $\gamma$ be as in (\ref{eq:formula_bridge}). Then for any pair of distinct vertices $w, x \in [2n]$ we have 
	\[
	\prob{\Bn{\bridgeGraph}{wx}\geq 2f(n)+1}\ =\ (1+o(1))\frac{9c}{2\gamma (\ce-2)}\cdot\frac{1}{f(n)^{\ce-2}n}.
	\]
\end{lem}

Using the second moment method, we will show that this number is concentrated around its expectation. 

\begin{lem}\label{lem:bridge_number2}
Let $\bridgeClass$ be a \bridgeStable\ class of multigraphs with critical exponent $\ce>3$ and $f$ be a function with $f(n)=\omega(1)$, $f(n)=o(n)$, and $f(n)=o(n^{1/(\ce-2)})$. Let $\bridgeGraph=\bridgeGraph(2n,3n)\ur \bridgeClass(2n,3n)$. In addition, let $w_1, w_2, w_3, w_4 \in [2n]$ be pairwise distinct and let $c$ and $\gamma$ be as in (\ref{eq:formula_bridge}). Then we have 
	\[
	\prob{\Bn{\bridgeGraph}{w_1w_2}, \Bn{\bridgeGraph}{w_3w_4}\geq 2f(n)+1}\ =\ (1+o(1))\left(\frac{9c}{2\gamma (\ce-2)}\cdot\frac{1}{f(n)^{\ce-2}n}\right)^2.
	\]
\end{lem}

Now we can combine \Cref{lem:bridge_number1,lem:bridge_number2} to determine the number of pairs $\{w,x\}$ with $\Bn{\bridgeGraph}{wx}\geq 2f(n)+1$.
\begin{lem}\label{lem:bridge_number3}
	Let $\bridgeClass$ be a \bridgeStable\ class of multigraphs with critical exponent $\ce>3$ and $f$ be a function with $f(n)=\omega(1)$, $f(n)=o(n)$, and $f(n)=o\left(n^{1/(\ce-2)}\right)$. In addition, let $\bridgeGraph=\bridgeGraph(2n,3n)\ur \bridgeClass(2n,3n)$ and let $c$ and $\gamma$ be as in (\ref{eq:formula_bridge}). Then \whp\ $\bridgeGraph$ has
	\begin{align*}
	(1+o(1))\frac{9c}{2\gamma (\ce-2)}\cdot \frac{1}{f(n)^{\ce-2}n}\cdot \binom{n}{2}
	\end{align*}
	many unordered pairs of vertices $w, x\in [2n]$ with $\Bn{\bridgeGraph}{wx}\geq 2f(n)+1$.
\end{lem}

Next, we show in two steps that \whp\ $\bridgeGraph=\bridgeGraph(2n,3n)\ur \bridgeClass(2n,3n)$ has a block with linearly many vertices. Firstly, we prove that \whp\ $\bridgeGraph$ has a \largeBlock\ block (see \Cref{def:large_block} and \Cref{lem:bridge_number5}). Secondly, we show that \whp\ this \largeBlock\ block has linearly many vertices (see \Cref{lem:bridge_number6}).
\begin{definition}\label{def:large_block}
	Let $\fixMultigraph$ be a connected multigraph. A block $\block$ of $\fixMultigraph$ is called \textit{\largeBlock}\ if every bridge $\bridge$ that shares a vertex with $\block$ satisfies $\numberVertices{C_\bridge}>\numberVertices{R_\bridge}$, where $C_\bridge$ is the component of $\fixMultigraph-\bridge$ containing $\block$ and $R_\bridge$ the other component of $\fixMultigraph-\bridge$.
\end{definition}

\Cref{lem:bridge_number1} says that \whp\ $\bridgeGraph=\bridgeGraph(2n,3n)\ur \bridgeClass(2n,3n)$ has no bridge $\bridge$ with $\Bn{\bridgeGraph}{\bridge}=\omega\left(n^{1/(\ce-2)}\right)$. Thus, assuming $\ce>3$ we obtain that \whp\ $\Bn{\bridgeGraph}{\bridge}=o(n)$ for every bridge $\bridge$. Intuitively, in such a case, we should have a \largeBlock\ block. The following lemma shows that this is indeed true.
\begin{lem}\label{lem:bridge_number5}
	Let $\fixMultigraph$ be a connected cubic multigraph on vertex set $[2n]$ such that $\Bn{\fixMultigraph}{wx}\leq(2n-2)/3$ for every pair of distinct vertices $w,x \in [2n]$. Then $\fixMultigraph$ contains a \largeBlock\ block. 
\end{lem}

Due to \Cref{lem:bridge_number1,lem:bridge_number5} we know that \whp\ $\bridgeGraph=\bridgeGraph(2n,3n)\ur \bridgeClass(2n,3n)$ has a \largeBlock\ block $B$ (in the case $\ce>3$). Let $\bridge_1, \ldots \bridge_r$ be the bridges that share a vertex with $B$. Then we get $2n=\numberVertices{\bridgeGraph}=\numberVertices{B}+\sum_{i=1}^{r}\Bn{\bridgeGraph}{e_i}$. Assuming that the bridge numbers $\Bn{\bridgeGraph}{e_i}$ are \lq small\rq, we get that $\numberVertices{B}$ or $r$ needs to be \lq large\rq. We note that in the latter case we again obtain that $\numberVertices{B}$ is \lq large\rq, since each vertex in $B$ can lie in at most one bridge. In \Cref{lem:bridge_number6} we will make this idea more precise. We already saw in \Cref{lem:bridge_number1} that the bridge numbers are typically \lq small\rq. However, we need the following stronger result in the proof of \Cref{lem:bridge_number6}.

\begin{lem}\label{lem:bridge_number4}
Let $\bridgeClass$ be a \bridgeStable\ class of multigraphs with critical exponent $\ce>2$, $0<\exponent<\ce-2$, and $\bridgeGraph=\bridgeGraph(2n,3n)\ur \bridgeClass(2n,3n)$. Then for any pair of distinct vertices $w, x \in [2n]$ we have
	\begin{align*}
	\expec{\left(\Bn{\bridgeGraph}{wx}\right)^\exponent}=\Theta(1/n).
	\end{align*}
\end{lem}

Recall that given a graph $\fixMultigraph$ and $i\in \N$, we denote by $\blockOrder{i}{\fixMultigraph}$ the number of vertices in the $i$-th largest block $\blockLargest{i}{\fixMultigraph}$ of $\fixMultigraph$.
\begin{lem}\label{lem:bridge_number6}
Let $\bridgeClass$ be a \bridgeStable\ class of multigraphs with critical exponent $\ce>3$. For $\bridgeGraph=\bridgeGraph(2n,3n)\ur \bridgeClass(2n,3n)$, we have $\blockOrder{1}{\bridgeGraph}=\Theta_p(n)$.
\end{lem}

Next, we consider the $i$-th largest block of $\bridgeGraph=\bridgeGraph(2n,3n)\ur \bridgeClass(2n,3n)$ for $i\geq 2$. We note that $\blockOrder{i}{\bridgeGraph}\leq \max \Bn{\bridgeGraph}{e}$ and thus $\blockOrder{i}{\bridgeGraph}=O_p\left(n^{1/(\ce-2)}\right)$ by \Cref{lem:bridge_number1}. On the other hand, we know by \Cref{lem:bridge_number3} that there is a bridge $\bridge$ in $\bridgeGraph$ with $\Bn{\bridgeGraph}{\bridge}=\Omega_p\left(n^{1/(\ce-2)}\right)$. Due to \Cref{lem:bridge_number6} we intuitively expect that both components of $\bridgeGraph-\bridge$ have again a block whose number of vertices is linear in the order of the component. By induction this would imply $\blockOrder{i}{\bridgeGraph}=\Omega_p\left(n^{1/(\ce-2)}\right)$. In the following lemma we show that this is indeed the case.

\begin{lem}\label{lem:bridge_number7}
	Let $\bridgeClass$ be a \bridgeStable\ class of multigraphs with critical exponent $\ce>3$ and $\bridgeGraph=\bridgeGraph(2n,3n)\ur \bridgeClass(2n,3n)$. Then for each $i\geq 2$, we have $\blockOrder{i}{\bridgeGraph}=\Theta_p\left(n^{1/(\ce-2)}\right)$. 
\end{lem}

\Cref{lem:bridge_number6,lem:bridge_number7} together with \Cref{lem:kernel_cubic} give us the block structure of a random kernel $\kernel{\planarRandomGraph}$. Now we combine this information with \Cref{lem:block_core} to obtain the block structure of a random \pl\ graph.

\begin{thm}\label{thm:block_structure_general}
	Let $\planarClass$ be a \pl\ class of graphs which is addable and has a critical exponent $3<\ce<4$. In addition, let $\planarRandomGraph=\planarRandomGraph(n,m)\ur \planarClass(n,m)$ and $\LargestComponent=\largestComponent{\planarRandomGraph}$ denote the largest component of $\planarRandomGraph$. Assume $m=n/2+s$ for $s=s(n)=o\left(n\right)$ and $s^3n^{-2} \to \infty$. Then the following hold.
	\begin{enumerate}
		\item\label{thm:block_structure_general_a}
		$\blockOrder{1}{\planarRandomGraph}=\Theta_p\left(sn^{-1/3}\right)$.
		\item\label{thm:block_structure_general_b}
		$\blockOrder{i}{\planarRandomGraph}=\Theta_p\left(s^{1/\left(\ce-2\right)}n^{\frac{\ce-4}{3(\ce-2)}}\right)$ for each $i\in \N$ with $i\geq 2$.
		\item\label{thm:block_structure_general_c}
		The number of blocks in $\LargestComponent$ is \whp\ $\Theta\left(sn^{-2/3}\right)$.
	\end{enumerate}
\end{thm}
The proof of \Cref{thm:block_structure_general} can be found in \Cref{proof:thm:block_structure_general}.
\begin{remark}\label{rem:block_structure}
	We believe that \Cref{thm:block_structure_general} is also true when $\ce\geq 4$. In that case we would need an improved version of \Cref{lem:block_core} where we can weaken the condition that $\blockOrder{i}{K}=\omega\left(\left(\numberVertices{K}\right)^{1/2}\right)$ to $\blockOrder{i}{K}=\omega\left(\left(\numberVertices{K}\right)^{\varepsilon}\right)$ for some $\varepsilon>0$. Using the ideas presented in \Cref{sec:polya} one may deduce such an improved statement (see also \Cref{rem:improve}). Nevertheless, we omit details, since we expect that the proofs become rather technical but do not provide any new insights. 
\end{remark}

\section{Proofs of main results}\label{sec:proofs_main}
Throughout this section, let $\planarClass$ be a \pl\ class of graphs and $\mathcal{K}$ the class of all kernels of graphs in $\planarClass$. Let $\planarClass(n,m)$ be the subclass of $\planarClass$ consisting of all graphs on vertex set $[n]$ with $m=m(n)$ edges and $P=\planarRandomGraph(n,m) \ur \mathcal \planarRandomGraph(n,m)$. Assume $m=n/2+s$ for $s=s(n)=o(n)$. Given a graph $\fixGraph$ we let $\largestComponent{\fixGraph}$ denote the largest component of $\fixGraph$ and $\rest(\fixGraph):=\fixGraph \setminus \largestComponent{\fixGraph}$, i.e. the graph obtained from $\fixGraph$ by deleting the largest component $\largestComponent{\fixGraph}$.

\proofof{thm:general}\label{proofof:thm:general}
We first consider the cases when $s^3n^{-2} \to -\infty$ or $s=O\left(n^{2/3}\right)$. Due to \Cref{lem:k4} each graph without complex components lies in $\planarClass$. Hence, we obtain by \Cref{thm:critcalRegime}\ref{thm:critcalRegime_a} that
$
\liminf_{n \to \infty} \prob{G(n,m) \in \planarClass(n,m)}>0.
$
Thus, each property that holds \whp\ in $G(n,m)$ is also true \whp\ in $\planarRandomGraph(n,m)$ and statements \ref{thm:main1_a} and \ref{thm:main1_b} follow from \Cref{thm:known}.

To prove \ref{thm:main1_c}, we assume $s^3n^{-2}\to \infty$. By \Cref{thm:internal_structure}\ref{thm:internal_structure_a}, i.e. \whp\ $\numberVertices{\core{\largestComponent{\planarRandomGraph}}}=\Theta\left(sn^{-1/3}\right)$, and the simple observation that $\longestCycle{\largestComponent{\planarRandomGraph}}\leq \numberVertices{\core{\largestComponent{\planarRandomGraph}}}$ we obtain that \whp\ \[\longestCycle{\largestComponent{\planarRandomGraph}}=O\left(sn^{-1/3}\right).\]

In order to prove the two other results on the girth and circumference of the largest component $\largestComponent{\planarRandomGraph}$, i.e. $\girth{\largestComponent{\planarRandomGraph}}=\Theta_p\left(n s^{-1}\right)$ 
and 
$\longestCycle{\largestComponent{\planarRandomGraph}}=\Omega_p\left(n^{1/3}\log \left(sn^{-2/3}\right)\right)$, we will use typical properties of the core and kernel of $\planarRandomGraph$. More precisely, let $\cl(n)$ be the subclass of $\planarClass(n,m)$ consisting of those graphs $\fixGraph$ with largest component $\largestComponent{\fixGraph}$ satisfying the following properties
\begin{align*}
	\numberVertices{\core{\largestComponent{\fixGraph}}} &=\Theta \left(sn^{-1/3}\right), \\
	\numberEdges{\kernel{\largestComponent{\fixGraph}}}&= \Theta \left(sn^{-2/3}\right),\\
	\text{ and }~~~~~ \numberLoops{\kernel{\largestComponent{\fixGraph}}}&= \Theta \left(sn^{-2/3}\right).~~~~~
\end{align*}
From \Cref{thm:internal_structure}\ref{thm:internal_structure_a}, \ref{thm:internal_structure_b}, and \ref{thm:internal_structure_e} we obtain that \whp\ $\numberVertices{\core{\largestComponent{\planarRandomGraph}}}=\Theta\left(sn^{-1/3}\right)$ and $\numberEdges{\kernel{\largestComponent{\planarRandomGraph}}}=3/2\cdot \numberVertices{\kernel{\largestComponent{\planarRandomGraph}}}= \Theta \left(sn^{-2/3}\right)$.
Furthermore, \Cref{cor:linear_loop} says that \whp\ $\numberLoops{\kernel{\largestComponent{\planarRandomGraph}}}=\Theta \left(sn^{-2/3}\right)$. Therefore, we have \whp\ $\planarRandomGraph \in \cl(n)$.

Next, we will apply \Cref{lem:splitting} to the class $\cl:=\bigcup_{n\in\N}\cl(n)$. So, we define the function $\func:\cl \to \mathcal{K} \times \N$ which maps a graph $\fixGraph \in \cl$ to the pair of kernel $\kernel{\largestComponent{\fixGraph}}$ and subdivision number $\subdivisionNumber{\largestComponent{\fixGraph}}$, i.e. \[\func(\fixGraph):=\left(\kernel{\largestComponent{\fixGraph}}, \subdivisionNumber{\largestComponent{\fixGraph}}\right).\] 

Let $\seq=\left(K_n, \nd_n\right)_{n \in \N}$ be a sequence that is feasible for $\left(\cl, \func\right)$ (cf. \Cref{def:feasible}) and let $\randomGraph=\randomGraph(n)\ur \cl(n)$. Due to the definition of $\func$ all possible realisations of $\largestComponent{\condGraph{\randomGraph}{\seq}}$ have the same kernel $K_n$ and the same subdivision number $k_n$. Hence, by \Cref{rem:core_equal} we have that $\core{\largestComponent{\condGraph{\randomGraph}{\seq}}}$ is distributed like $\randomCore(K_n, \nd_n)$, a graph chosen uniformly at random from the class of all cores with kernel $K_n$ and subdivision number $\nd_n$. From the definition of $\cl(n)$ we have $\nd_n=\Theta\left(sn^{-1/3}\right)$, $\numberEdges{K_n}=\Theta\left(sn^{-2/3}\right)$, and $\numberLoops{K_n}=\Theta\left(sn^{-2/3}\right)$. In particular, this yields $\nd_n=\omega\left(\numberEdges{K_n}^2\right)$ and $\numberLoops{K_n}=\Theta\left(\numberEdges{K_n}\right)$. Hence, by combining \Cref{lem:random_core2} and \Cref{cor:contiguous} we obtain 
\begin{align*}
\girth{\largestComponent{\condGraph{\randomGraph}{\seq}}}&=\girth{\randomCore(K_n, \nd_n)}=\Theta_p\left(\frac{k_n}{\numberEdges{K_n}^2}\right)=\Theta_p\left(\frac{sn^{-1/3}}{\left(sn^{-2/3}\right)^2}\right)=\Theta_p\left(n/s\right)\\
\text{ and }~~~~~
\longestCycle{\largestComponent{\condGraph{\randomGraph}{\seq}}}&=\longestCycle{\randomCore(K_n, \nd_n)}=\Omega_p\left(\frac{k_n}{\numberEdges{K_n}}\log \numberLoops{K_n} \right)=\Omega_p\left(n^{1/3}\log \left(sn^{-2/3}\right)\right).
\end{align*}
As the sequence $\seq$ was arbitrary, \Cref{lem:splitting} implies that the results above hold also for the (unconditional) random graph $\randomGraph$. Since \whp\ $\planarRandomGraph \in \cl(n)$, the same is true for $\planarRandomGraph$, i.e. \[\girth{\largestComponent{\planarRandomGraph}}=\Theta_p\left(ns^{-1}\right)\quad \text{and}\quad \longestCycle{\largestComponent{\planarRandomGraph}}=\Omega_p\left(n^{1/3}\log \left(sn^{-2/3}\right)\right),\]
as desired.

Next, we study the graph $\Rest{\planarRandomGraph}:=\planarRandomGraph \setminus \largestComponent{\planarRandomGraph}$ and prove that its circumference satisfies $\longestCycle{\Rest{\planarRandomGraph}}=\Theta_p\left(n^{1/3}\right)$. To this end, we will show that $\Rest{\planarRandomGraph}$ behaves similarly like the uniform random graph $G\left(n-\numberVertices{\LargestComponent}, m-\numberEdges{\LargestComponent}\right)$ and then apply \Cref{thm:known}\ref{thm:known2}. More precisely, given some constant $\constant>0$ we denote by $\tilde{\cl}(n)$ the subclass of $\planarClass(n,m)$ consisting of those graphs $H$ such that 
\begin{align}\label{eq:interval}
\numberVertices{\largestComponent{H}},\ \numberEdges{\largestComponent{H}} \in \left[2s-\constant n^{2/3}, 2s+\constant n^{2/3}\right].
\end{align}
We set $\tilde{\cl}:=\bigcup_{n\in\N} \tilde{\cl}(n)$ and let $\delta>0$. By \Cref{thm:internal_structure}\ref{thm:internal_structure_f} and \ref{thm:internal_structure_g} we can choose the constant $\constant$ such that 
\begin{align}\label{eq:rest}
\prob{\planarRandomGraph\in \tilde{\cl}(n)}\geq 1-\delta
\end{align}
for all $n$ large enough. Let $\tilde{\randomGraph}=\tilde{\randomGraph}(n)\ur \tilde{\cl}(n)$ and $\tilde{\func}$ be the function that assigns each graph $\fixGraph \in \tilde{\cl}$ its largest component, i.e. $\tilde{\func}(\fixGraph):=\largestComponent{\fixGraph}$. Moreover, let $\mathbf{\fixGraph}=(\fixGraph_n)_{n\in\N}$ be a sequence that is feasible for $\left(\tilde{\cl}, \tilde{\func}\right)$. We denote by $n_U=n_U(n):=n-\numberVertices{H_n}$ and $m_U=m_U(n):=m-\numberEdges{H_n}$ the number of vertices and edges in $\Rest{\condGraph{\tilde{\randomGraph}}{\mathbf{\fixGraph}}}$, respectively. Next, we will study relations between the distributions of $\tilde{R}_n:=\Rest{\condGraph{\tilde{\randomGraph}}{\mathbf{\fixGraph}}}$ and the uniform random graph $R_n:=G(n_U, m_U)\ur \mathcal G(n_U, m_U)$. To this end, let $\mathcal{R}(n)$ be the subclass of $\mathcal{G}(n_U,m_U)$ consisting of all possible realisations of $\tilde{R}_n$, i.e. the set of graphs $R$ such that 
\begin{align}\label{eq:rest1}
R\cup H_n \in \tilde{\cl} \quad \text{and} \quad \largestComponent{R\cup H_n}=H_n.
\end{align}
We claim that each graph $R\in \mathcal{G}(n_U,m_U)$ having no complex components and satisfying $\largestComponent{R}<\numberVertices{\fixGraph_n}$ is in $\mathcal{R}(n)$. Indeed such a graph $R$ satisfies $\largestComponent{R\cup H_n}=H_n$. Moreover, we have $\kernel{R\cup H_n}=\kernel{H_n}$. Thus, by the stability property (K1) of \pl\ classes (cf. \Cref{def:planar_like}) we get $R\cup H_n \in \planarClass$ and therefore $R\cup H_n \in \tilde{\cl}$. This implies $R\in \mathcal{R}(n)$ due to (\ref{eq:rest1}). Next, we will show that $|\mathcal{R}|$ is \lq large\rq\ in the sense that $\left|\mathcal{R}(n)\right|/\left|\mathcal{G}(n_U,m_U)\right|$ is bounded away from 0. To this end, we use (\ref{eq:interval}) to obtain
\begin{align}\label{eq:rest2}
m_U= m-\numberEdges{H_n}\leq m-2s+Mn^{2/3}=\frac{n}{2}-s+Mn^{2/3}\leq \frac{n}{2}-\frac{\numberVertices{H_n}}{2}+\frac{3Mn^{2/3}}{2}=\frac{n_U}{2}+\frac{3Mn^{2/3}}{2}.
\end{align}
Using (\ref{eq:interval}) and the fact that $s=o(n)$ we get $n_U=(1+o(1))n$. Combining that with (\ref{eq:rest2}) yields that for $n$ large enough
\begin{align}\label{eq:rest5}
m_U\leq \frac{n_U}{2}+2Mn_U^{2/3}.
\end{align}
Together with \Cref{thm:critcalRegime}\ref{thm:critcalRegime_a} this implies
\begin{align}\label{eq:rest3}
\liminf_{n \to \infty}~ \prob{R_n \text{ has no complex component}}>0.
\end{align}
As $\numberVertices{\fixGraph_n}=\Theta(s)$, we obtain by \Cref{thm:critcalRegime}\ref{thm:critcalRegime_b} that \whp
\begin{align}\label{eq:rest4}
	\numberVertices{\largestComponent{R_n}}<\numberVertices{\fixGraph_n}.
\end{align}
Combining (\ref{eq:rest3}) and (\ref{eq:rest4}) with the claim shown above yields
\begin{align}\label{eq:rest6}
\liminf_{n \to \infty} ~ \prob{R_n\in \mathcal{R}(n)}>0.
\end{align}
Similarly as in (\ref{eq:rest2}) we use (\ref{eq:interval}) to get
\begin{align*}
m_U= m-\numberEdges{H_n}\geq m-2s-Mn^{2/3}=\frac{n}{2}-s-Mn^{2/3}\geq \frac{n}{2}-\frac{\numberVertices{H_n}}{2}-\frac{3Mn^{2/3}}{2}=\frac{n_U}{2}-\frac{3Mn^{2/3}}{2},
\end{align*}
which yields $m_U\geq n_U/2-2Mn_U^{2/3}$ for large $n$. Combining that with (\ref{eq:rest5}) we obtain $m_U=n_U/2+O\left(n_U^{2/3}\right)$. Hence, we get by \Cref{thm:known}\ref{thm:known2} that $\longestCycle{R_n}=\Theta_p\left(n_U^{1/3}\right)=\Theta_p\left(n^{1/3}\right)$. By (\ref{eq:rest6}) each property that holds \whp\ in $R_n$ is also true \whp\ in $\tilde{R}_n \ur \mathcal{R}(n)$. Thus, we have $\longestCycle{\tilde{R}_n}=\Theta_p\left(n^{1/3}\right)$. From the definition of $\tilde{R}_n$ we get $\longestCycle{\Rest{\condGraph{\tilde{\randomGraph}}{\mathbf{\fixGraph}}}}=\Theta_p\left(n^{1/3}\right)$. Hence, by \Cref{lem:splitting} we have $\longestCycle{\Rest{\tilde{\randomGraph}}}=\Theta_p\left(n^{1/3}\right)$. Finally, using (\ref{eq:rest}) and observing that the choice of $\delta>0$ was arbitrary we have 
\[\longestCycle{\Rest{\planarRandomGraph}}=\Theta_p\left(n^{1/3}\right),\]
which completes the proof.
\qed

\proofof{thm:main1}
The statement follows by combining \Cref{lem:planar_like}\ref{lem:planar_like_1} and \Cref{thm:general}. \qed

\proofof{coro:general}
The result follows directly from \Cref{lem:planar_like} and \Cref{thm:general}. \qed

\proofof{coro:outerplanar}
The results in the regimes $s^3n^{-2} \to -\infty$ and $s=O\left(n^{2/3}\right)$ follow analogously as for \pl\ classes (see \Cref{proofof:thm:general}). In \cite[Theorem 4]{sparse_outerplanar} it is proven that in the case $s^3n^{-2} \to \infty$ \whp\ a random outerplanar graph is a cactus graph. Thus, the statements in that regime follow directly from \Cref{coro:general}. \qed

\proofofL{thm:block_structure_general}
We start by considering the blocks of the kernel $\kernel{\planarRandomGraph}$. By \Cref{lem:kernel_cubic} we have that $\kernel{\planarRandomGraph}$ \lq behaves\rq\ like a random cubic multigraph chosen from $\mathcal{K}$. Furthermore, by \Cref{thm:internal_structure}\ref{thm:internal_structure_b} and \ref{thm:internal_structure_d}, we know that \whp\ $\numberVertices{\kernel{\planarRandomGraph}}=\Theta\left(sn^{-2/3}\right)$. Combining that with \Cref{lem:bridge_number6} implies 
\begin{align}\label{eq:block1}
\blockOrder{1}{\kernel{\planarRandomGraph}}=\Theta_p\left(sn^{-2/3}\right).
\end{align}
Similarly, by \Cref{lem:bridge_number7} we have that for each $i\geq 2$
\begin{align}\label{eq:block2}
\blockOrder{i}{\kernel{\planarRandomGraph}}=\Theta_p\left(\numberVertices{\kernel{\planarRandomGraph}}^{1/\left(\ce-2\right)}\right)=\Theta_p\left(s^{1/\left(\ce-2\right)}n^{-2/\left(3(\ce-2)\right)}\right).
\end{align}
Next, we determine the orders of the blocks in the core $\core{\planarRandomGraph}$. To this end, we will use \Cref{lem:splitting,lem:block_core} and we fix $i\geq 2$. Let $\cl(n)$ be the subclass of $\planarClass(n,m)$ consisting of those graphs $\fixGraph$ satisfying the following properties
\begin{align*}
\numberVertices{\core{\fixGraph}}&=\Theta \left(sn^{-1/3}\right),\\
\numberEdges{\kernel{\fixGraph}}&=\Theta \left(sn^{-2/3}\right),\\
\blockOrder{1}{\kernel{\fixGraph}}&=\Theta\left(sn^{-2/3}\right),\\
\blockOrder{i}{\kernel{\fixGraph}}&=\Theta\left(s^{1/(\ce-2)}n^{-2/\left(3(\ce-2)\right)}\right),
\\
\text{and}~~~~~~~~~~~~~~~~~~~\kernel{\fixGraph}&\text{ is cubic}.
\end{align*} 
Due to \Cref{thm:internal_structure}\ref{thm:internal_structure_a}--\ref{thm:internal_structure_e}, (\ref{eq:block1}), and (\ref{eq:block2}) we can choose the implicit hidden constants in the above equations such that $\prob{\planarRandomGraph \in \cl(n)}>1-\delta$, for a fixed constant $\delta>0$. Let $\cl:=\bigcup_{n\in\N} \cl(n)$, $\tilde{\randomGraph}=\tilde{\randomGraph}(n)\ur \tilde{\cl}(n)$, and define the function $\func:\cl \to \mathcal{K}\times \N$ which assigns each graph $\fixGraph \in \cl$ to the pair of its kernel $\kernel{\fixGraph}$ and subdivision number $\subdivisionNumber{\fixGraph}$, i.e. 
\begin{align*}
\func\left(\fixGraph\right):=\left(\kernel{\fixGraph}, \subdivisionNumber{\fixGraph}\right).
\end{align*}
Let $\seq=\left(K_n, \nd_n\right)_{n\in \N}$ be a sequence feasible for $\left(\cl, \func\right)$ and $\randomCore_n:=\randomCore(K_n, \nd_n)$ as in \Cref{def:random_core}, i.e. a graph chosen uniformly at random from all cores with kernel $K_n$ and subdivision number $\nd_n$. By the definition of $\cl(n)$ we have $\nd_n=\omega\left(\numberEdges{K_n}^2\right)$. Thus, combining \Cref{cor:contiguous} and \Cref{lem:block_core} yields 
\begin{align*}
\blockOrder{1}{\randomCore_n}&=\Theta\left(\frac{k_n\blockOrder{1}{K_n}}{\numberVertices{K_n}}\right)=\Theta\left(sn^{-1/3}\right)\\
\text{ and }~~~~~\blockOrder{i}{\randomCore_n}&=\Theta\left(\frac{k_n\blockOrder{i}{K_n}}{\numberVertices{K_n}}\right)=\Theta\left(s^{1/\left(\ce-2\right)}n^{\frac{\ce-4}{3(\ce-2)}}\right).
\end{align*}
Due to \Cref{rem:core_equal} we have that $\core{\condGraph{\randomGraph}{\seq}}$ is distributed like $\randomCore_n$. Hence, by \Cref{lem:splitting} we obtain that \whp\
\begin{align*}
\blockOrder{1}{\core{\randomGraph}}=\Theta\left(sn^{-1/3}\right) 
\quad \text{and}\quad \blockOrder{i}{\core{\randomGraph}}=\Theta\left(s^{1/\left(\ce-2\right)}n^{\frac{\ce-4}{3(\ce-2)}}\right).
\end{align*}
As $\delta>0$ was arbitrary and $\prob{\planarRandomGraph \in \cl(n)}>1-\delta$, this implies
\begin{align}\label{eq:blockincore}
	\blockOrder{1}{\core{\planarRandomGraph}}=\Theta_p\left(sn^{-1/3}\right) 
	\quad \text{and}\quad \blockOrder{i}{\core{\planarRandomGraph}}=\Theta_p\left(s^{1/\left(\ce-2\right)}n^{\frac{\ce-4}{3(\ce-2)}}\right).
\end{align}
We note that each block outside the core $\core{\planarRandomGraph}$ is a cycle. Due to \Cref{thm:general}, the length of such a cycle is of order $O_p\left(n^{1/3}\right)$. This together with \eqref{eq:blockincore} and the observations that $n^{1/3}=o\left(s^{1/\left(\ce-2\right)}n^{\frac{\ce-4}{3(\ce-2)}}\right)$ implies statements \ref{thm:block_structure_general_a} and \ref{thm:block_structure_general_b}. 

For statement \ref{thm:block_structure_general_c}, we first observe that the number of blocks in $\core{\largestComponent{\planarRandomGraph}}$ is at most $\numberVertices{\kernel{\largestComponent{\planarRandomGraph}}}$. Thus, by \Cref{thm:internal_structure}\ref{thm:internal_structure_a} we have that \whp\ $\core{\largestComponent{\planarRandomGraph}}$, and therefore also $\largestComponent{\planarRandomGraph}$, has $O\left(sn^{-2/3}\right)$ many blocks. On the other hand, \whp\ $\kernel{\largestComponent{\planarRandomGraph}}$ has $\Theta\left(sn^{-2/3}\right)$ many loops due to \Cref{cor:linear_loop}. All these loops \lq translate\rq\ to different blocks in the core $\core{\largestComponent{\planarRandomGraph}}$. Thus, \whp\ $\largestComponent{\planarRandomGraph}$ has $\Omega\left(sn^{-2/3}\right)$ many blocks. Summing up, \whp\ $\largestComponent{\planarRandomGraph}$ contains $\Theta\left(sn^{-2/3}\right)$ many blocks.\qed

\proofof{thm:block_structure}
The statement follows by combining \Cref{lem:planar_like}\ref{lem:planar_like_1} and \Cref{thm:block_structure_general}. \qed

\section{Proofs of auxiliary results}\label{sec:proofs_aux}
\proofof{lem:splitting}
For each $n\in \N$ let $s_n^\ast \in \mathcal{S}$ be such that $\condprob{\randomGraph(n) \in \property}{\func\left(\randomGraph(n)\right)=s}$ is minimised for $s=s_n^\ast$ (among all $s$ for which there exists a graph $\fixGraph\in \cl(n)$ with $\func(H)=s$). Note that the sequence $\seq^\ast=\left(s_n^\ast\right)_{n\in \N}$ is feasible for $(\cl, \func)$ and therefore we obtain
\begin{align*}
\prob{\randomGraph(n) \in \property}
&\ =\ \sum_{s \in \mathcal{S}}\prob{\func(A(n))=s} \cdot \condprob{\randomGraph(n) \in \property}{\func\left(\randomGraph(n)\right)=s}
\\
&\ \geq\ \sum_{s \in \mathcal{S}}\prob{\func(A(n))=s}\cdot \condprob{\randomGraph(n) \in \property}{\func\left(\randomGraph(n)\right)=s_n^\ast}
\\
&\ =\ \prob{\left(\condGraph{\randomGraph}{\seq^\ast}\right)(n)\in \property}\\
&\ =\ 1-o(1),
\end{align*}
where the sums are taken over all $s\in \mathcal{S}$ for which there exists some $\fixGraph\in \cl(n)$ with $\func(H)=s$.
\qed

\proofof{lem:k3}
The assertion follows along the lines of the proof of Lemma 2 in \cite{planar}. \qed

\proofof{lem:k4}
We use the properties of $\planarClass$ from \Cref{def:planar_like}. By (K2) there exists some $\fixGraph_1 \in \planarClass$. Due to (K1) this yields $\fixGraph \cup \fixGraph_1 \in \planarClass$, as $\kernel{H\cup H_1}=\kernel{H}$. Finally, (P1) implies $\fixGraph \in \planarClass$. \qed

\proofof{rem:core_equal}
The assertion is equivalent to the statement that $\prob{\core{\condGraph{\planarRandomGraph}{\left(K, \nd\right)}}=\fixGraph}$ is independent of the choice of $\fixGraph\in \mathcal{C}(K,\nd)$. Hence, it suffices to prove that for each $\fixGraph\in \mathcal{C}(K,\nd)$ the set $\left\{F\in \planarClass(n,m) \ \mid\ \core{F}=\fixGraph,\ \kernel{F}=K,\ \subdivisionNumber{F}=\nd\right\}$ has the same number of elements. We observe that $\left\{F\in \planarClass(n,m) \ \mid\ \core{F}=\fixGraph,\ \kernel{F}=K,\ \subdivisionNumber{F}=\nd\right\}=\left\{F\in \planarClass(n,m) \ \mid\ \core{F}=\fixGraph\right\}$ and each graph $F\in \planarClass(n,m)$ having $\fixGraph$ as its core can be constructed as follows. First we replace each vertex in $\fixGraph$ by a rooted tree and then we attach a graph without complex components such that we obtain a graph on $n$ vertices and $m$ edges (see \cite[Section 3]{surface} for details). The number of different possibilities of that construction depends only on $n$, $m$, $\numberVertices{\fixGraph}$, and $\numberEdges{\fixGraph}$. We have $\numberVertices{\fixGraph}=\numberVertices{K}+\nd$ and $\numberEdges{\fixGraph}=\numberEdges{K}+\nd$, i.e. $\numberVertices{\fixGraph}$ and $\numberEdges{\fixGraph}$ are independent of the choice of $\fixGraph\in \mathcal{C}(K,\nd)$. Hence, the size of the set $\left\{F\in \planarClass(n,m) \ \mid\ \core{F}=\fixGraph\right\}$ is the same for all $\fixGraph\in \mathcal{C}(K,\nd)$, which implies the statement.
\qed

\proofof{lem:random_core}
Let $e_1, \ldots, e_\nc$ be the edges of $K$ and we denote by $X_i$ the number of vertices that are placed on edge $e_i$ when we subdivide $K$ to obtain $\randomMultiCore$. To prove \ref{lem:random_core_a}, we observe that 

\begin{align*}
\prob{\randomMultiCore\text{ is simple}}&\geq \prob{\randomMultiCore\text{ is \2s}}= \prob{\bigwedge_{\ind=1}^\nc \left(X_\ind \geq 2\right)}
 \geq 1- \sum_{\ind=1}^{\nc}\prob{X_\ind \leq 1} \geq 1-\frac{2\nc^2}{\nd}, 
\end{align*} 
where the last inequality follows from \Cref{prop:polya_main}.
 
In order to prove \ref{lem:random_core_b}, it suffices to show that $\prob{\randomMultiCore=H}$ is independent of the choice of $H \in \mathcal{C}(K,\nd)$. To that end, we count the number of ways our random process ends up with $\randomMultiCore=H$. We observe that there is a unique sequence $(G_0, \ldots, G_\nd)$ that leads to $G_\nd=H$. Thus, in each step $i$ there is a unique unordered pair of vertices $\left\{u_i, v_i\right\}$ such that subdividing an edge between $u_i$ and $v_i$ in $G_{i-1}$ leads to $G_i$. We denote by $q_i$ the number of edges in $G_{i-1}$ between $u_i$ and $v_i$. Then, there are $\prod_{i=1}^{\nd}q_i$ many ways of creating $H$. We note that the only way a multiple edge can be created during the process is by subdividing a loop and that all loops and multiple edges are destroyed in the end. Thus, we obtain $\prod_{i=1}^{\nd}q_i=\frac{1}{\weight{K}}$, where $\weight{K}$ is defined as in (\ref{eq:weight}). This shows \ref{lem:random_core_b}, since $\weight{K}$ is independent of the choice of $H \in \mathcal{C}(K,\nd)$.
\qed

\proofof{cor:contiguous}
We observe that by \Cref{lem:random_core} we have that \whp\ $\randomMultiCore$ is simple. For each graph property $\property$ we obtain by \Cref{lem:random_core}\ref{lem:random_core_b}
\begin{align*}
\prob{\randomMultiCore \in \property}&=\prob{\randomMultiCore \text{ is simple }}\condprob{\randomMultiCore \in \property}{\randomMultiCore \text{ is simple }}+\prob{\randomMultiCore \text{ is not simple}}\condprob{\randomMultiCore \in \property}{\randomMultiCore \text{ is not simple }}
\\
&=(1+o(1))\prob{\randomCore \in \property}+o(1).
\end{align*}
This implies that \whp\ $\randomCore \in \property$ if and only if \whp\ $\randomMultiCore\in\property$. \qed 

\proofof{cor:subdivision_twice}
We will use \Cref{lem:splitting}. Let $\property$ be the graph property of being \2s. In addition, let $\cl(n)$ be the subclass of $\planarClass(n,m)$ consisting of all graphs $\fixGraph$ with $\numberVertices{\core{H}}=\Theta\left(sn^{-1/3}\right)$ and $\numberEdges{\kernel{H}}=\Theta\left(sn^{-2/3}\right)$. Due to \Cref{thm:internal_structure}\ref{thm:internal_structure_a} and \ref{thm:internal_structure_c} we have that \whp\ $\numberVertices{\core{\planarRandomGraph}}=\Theta\left(sn^{-1/3}\right)$. In addition, by \Cref{thm:internal_structure}\ref{thm:internal_structure_b}, \ref{thm:internal_structure_d}, and \ref{thm:internal_structure_e}
we have that \whp\
\begin{align*}
\numberEdges{\kernel{\planarRandomGraph}}=3/2 \cdot \numberVertices{\kernel{\planarRandomGraph}}=\Theta\left(sn^{-2/3}\right).
\end{align*}
Hence, we obtain that \whp\ $\planarRandomGraph \in \cl(n)$. Let $\cl:=\bigcup_{n\in\N} \cl(n)$ and define the function $\func$ for a graph $\fixGraph \in \cl$ by $\func(\fixGraph):=\left(\kernel{\fixGraph}, \subdivisionNumber{\fixGraph}\right)$. Let $\seq=\left(K_n, \nd_n\right)_{n\in \N}$ be a sequence feasible for $(\cl, \func)$ and let $\randomCore(K_n, \nd_n)$ and $\randomMultiCore(K_n, \nd_n)$ be as in \Cref{def:random_core}. By definition of $\cl(n)$ we have $\nd_n=\omega\left(\numberEdges{K_n}^2\right)$. Thus, by \Cref{lem:random_core}\ref{lem:random_core2_a} we have that \whp\ $\randomMultiCore(K_n, \nd_n)$ is \2s. By \Cref{cor:contiguous} this is also true for $\randomCore(K_n, \nd_n)$. Let $\randomGraph=\randomGraph(n) \ur \cl(n)$. We note that $\core{\condGraph{\randomGraph}{\seq}}$ is distributed like $\randomCore(K_n, \nd_n)$ due to \Cref{rem:core_equal}. Thus, by \Cref{lem:splitting}, we have that \whp\ $\core{\randomGraph}$ is \2s. Since \whp\ $\planarRandomGraph \in \cl(n)$, it is also true that \whp\ $\core{\planarRandomGraph}$ is \2s. Finally, the statement follows, because $\planarRandomGraph$ is \2s\ if and only if $\core{\planarRandomGraph}$ is.
\qed

\proofof{lem:random_core2}
Let $e_1, \ldots, e_\nc$ be the edges of $K$ and $X_i$ the number of vertices that are placed on edge $e_i$ if we subdivide $K$ to obtain $\randomMultiCore$. Without loss of generality we may assume that $e_1, \ldots, e_\NumberLoops$ are the loops of $K$. Then the upper bounds on $\girth{\randomMultiCore}$ follow by \Cref{main1}\ref{thm:main1_1} and inequality (\ref{bound1}). For the lower bound on $\girth{\randomMultiCore}$ we use \Cref{main1}\ref{thm:main1_1} and (\ref{bound2}). The \lq in particular\rq\ statements follow immediately by combining the lower and upper bounds on $\girth{\randomMultiCore}$. Finally, we note that \ref{lem:random_core2_b} follows by \Cref{main1}\ref{thm:main1_2} and (\ref{bound3}).\qed

\proofof{lem:block_core}
We note that the blocks of $\randomMultiCore$ are the blocks of $K$ with additional vertices placed on the edges of $K$. For $j\in \N$ let $X_j$ be the total number of vertices that are placed on edges of the $j$-th largest block $\blockLargest{j}{K}$ of $K$. The minimum degree of each block is at least two and together with the fact that $K$ is cubic, this implies
\begin{align}\label{eq:1}
\blockOrder{j}{K}\leq \numberEdges{\blockLargest{j}{k}}\leq 3/2\blockOrder{j}{K}.
\end{align}
By \Cref{thm:polya2} we have
\begin{align}\label{eq:2}
\expec{X_j}=\frac{k\numberEdges{\blockLargest{j}{k}} }{\numberEdges{K}}\quad \text{ and }\quad \variance{X_j}=O\left(1\right)\frac{k^2\numberEdges{\blockLargest{j}{k}} }{\numberEdges{K}^2}.
\end{align}
Thus, by Chebyshev's inequality, (\ref{eq:1}), and (\ref{eq:2}) we obtain
\begin{align*}
\prob{X_j\leq \frac{\nd \blockOrder{j}{K}}{3\numberVertices{K}}}\leq
\prob{X_j\leq \frac{\expec{X_j}}{2}}\leq \frac{4\variance{X_j}}{\expec{X_j}^2}=O\left(1\right)\frac{k^2\numberEdges{\blockLargest{j}{k}} }{\numberEdges{K}^2}\frac{\numberEdges{K}^2}{k^2\numberEdges{\blockLargest{j}{k}}^2}=O\left(1\right)\blockOrder{j}{K}^{-1}.
\end{align*} 
Hence, \whp\ for all $j\leq i$ we have $X_j\geq \frac{\nd \blockOrder{j}{K}}{3\numberVertices{K}}\geq \frac{\nd \blockOrder{i}{K}}{3\numberVertices{K}}$, which shows \whp\ $\blockOrder{i}{\randomMultiCore}=\Omega\left(\frac{\nd \blockOrder{i}{K}}{\numberVertices{K}}\right)$. By again applying Chebyshev's inequality, (\ref{eq:1}), and (\ref{eq:2}), we have uniformly over all $j\geq i$ 
\begin{align*}
\prob{X_j\geq \frac{2\nd\blockOrder{i}{K}}{\numberVertices{K}}}&\leq
\prob{X_j\geq \expec{X_j}+\frac{\nd\blockOrder{i}{K}}{\numberVertices{K}}}
\leq
\frac{\numberVertices{K}^2\variance{X_j}}{\nd^2\blockOrder{i}{K}^2}
=O(1)\frac{\blockOrder{j}{K}}{\blockOrder{i}{K}^2}.
\end{align*}
Thus, by a standard union bound we obtain
\begin{align*}
\prob{\exists j\geq i: X_j\geq \frac{2\nd\blockOrder{i}{K}}{\numberVertices{K}}}=O(1)\blockOrder{i}{K}^{-2}\sum_{j\geq i}\blockOrder{j}{K}=O(1)\blockOrder{i}{K}^{-2}\numberVertices{K}=o(1).
\end{align*}
That yields \whp\ $\blockOrder{i}{\randomMultiCore}=O\left(\frac{\nd \blockOrder{i}{K}}{\numberVertices{K}}\right)$, which completes the proof.\qed

\proofof{lem:kernel_cubic}
We will use \Cref{lem:splitting}. To this end, let $c_1, c_2>0$ and $\cl(n)$ be the subclass of $\planarClass(n,m)$ consisting of all \2s graphs $H$ with a cubic kernel $\kernel{H}$ and satisfying $c_1sn^{-2/3}\leq \numberVertices{\kernel{H}}/2\leq c_2sn^{-2/3}$. Due to \Cref{cor:subdivision_twice} we know that \whp\ $\planarRandomGraph$ is \2s. Moreover, by \Cref{thm:internal_structure}\ref{thm:internal_structure_b}, \ref{thm:internal_structure_d}, and \ref{thm:internal_structure_e} we have that \whp\ $\kernel{\planarRandomGraph}$ is cubic and $\numberVertices{\kernel{\planarRandomGraph}}=\Theta\left(sn^{-2/3}\right)$. Thus, we can choose $c_1, c_2$ such that \whp\ $\planarRandomGraph \in \cl(n)$. Let $ \cl:=\bigcup_{n\in\N} \cl(n)$ and define the function $\func$ for a graph $\fixGraph \in \cl$ by 
\begin{align*}
\func(\fixGraph):=\numberVertices{\kernel{\fixGraph}}/2.
\end{align*}
Let $\seq=\left(\ell_n\right)_{n\in \N}$ be a sequence feasible for $\left(\cl, \func\right)$ and $A=A(n)\ur \cl(n)$. We note that for a fixed kernel $K\in \mathcal{K}(2\ell, 3\ell)$ and a fixed $\nd \in \N \cup \{0\}$ there are $\weight{K}\binom{\nd-3\ell-1}{3\ell-1}\nd!$ many ways to construct a \2s\ core with kernel $K$ and subdivision number $\nd$. Thus, $\kernel{\condGraph{\randomGraph}{\seq}}$ is distributed like $K\left(2\ell_n, 3\ell_n\right)$ by \Cref{rem:core_equal}. Hence, we obtain
\begin{align*}
&\prob{g_1\left(c_1sn^{-2/3}\right)\leq F\left(\kernel{\condGraph{\randomGraph}{\seq}}\right)\leq g_2\left(c_2sn^{-2/3}\right)}
\\
&\geq \prob{g_1(\ell_n)\leq F\left(K(2\ell_n, 3\ell_n)\right)\leq g_2(\ell_n)}\\
&=1-o(1),
\end{align*}
as $\ell_n\geq c_1sn^{-2/3}\to \infty$. Thus, the statement follows by \Cref{lem:splitting}.
\qed

\proofof{thm:linear_loop}
We recall that $\numberLoops{K}$ is the number of loops in $K=K(2n,3n)$ and observe that $\numberLoops{K}=\sum_{w\in [2n]}Z_w$, where $Z_w$ is the indicator random variable for the event that there is a loop at vertex $w$. In order to apply the second moment method, we estimate the probabilities $\prob{Z_w=1}$ and $\prob{Z_u=Z_w=1}$ for $u\neq w$. To this end, we will use \loopInsertions\ (cf. \Cref{def:loop_insertion}). We fix a vertex $w \in [2n]$ and consider all multigraphs in $\mathcal{K}(2n,3n)$ with a loop at $w$. We note that in all these multigraphs $w$ has precisely one neighbour $x\neq w$. We distinguish two cases depending on whether there is a loop at $x$ or not. Due to \Cref{lem:loop_insertion_stable} we can enumerate all these multigraphs with no loop at $x$ as follows:
\begin{itemize}
	\item
	choose a vertex $x \in [2n]\setminus\{w\}$;
	\item
	choose $H\in \mathcal{K}(2(n-1), 3(n-1))$ and relabel the vertices with the labels $[2n]\setminus \{w, x\}$;
	\item
	choose an edge $e \in E(H)$ and
	\item[] 
	 perform a \loopInsertion\ at edge $e$ with vertex pair $(w,x)$.
\end{itemize}
For simplicity we set $a_n:=\left|\mathcal{K}(2n,3n)\right|$. If $x$ and $H$ are fixed, the total weight of all multigraphs that can be built by choosing an edge $e\in E(H)$ and performing a \loopInsertion\ at edge $e$ with vertex pair $(w,x)$ is $\weight{H}3(n-1)/2$ by \Cref{prop:loop_insertion}. Hence, the total weight of all multigraphs which can be obtained by the above construction is
\begin{align}\label{eq:construction_1}
(2n-1)\cdot a_{n-1} \cdot 3(n-1)/2.
\end{align}
On the other hand, if there is a loop at $x$, the vertices $w$ and $x$ form a component with weight $1/4$. Thus, all such multigraphs can be enumerated as follows:
\begin{itemize}
	\item
	choose a vertex $x \in [2n]\setminus\{w\}$;
	\item
	choose $H\in \mathcal{K}(2(n-1), 3(n-1))$ and relabel the vertices with the labels $[2n]\setminus \{w, x\}$;
	\item
	add the component $C$ with $\vertexSet{C}=\{w,x\}$ and $\edgeSet{C}=\{ww, xx, wx\}$ to $H$.
\end{itemize}
The total weight of all multigraphs constructed in that way is
\begin{align}\label{eq:construction_2}
(2n-1)a_{n-1}/4.
\end{align}
We observe that each multigraph can be obtained at most once by one of the two above constructions. Thus, by combining (\ref{eq:construction_1}) and (\ref{eq:construction_2}) we obtain
\begin{align}\label{proof:linear_loop_0}
\prob{Z_v=1}=\frac{(2n-1)a_{n-1}3(n-1)/2+(2n-1)a_{n-1}/4}{a_n}=(1+o(1))3n^2\frac{a_{n-1}}{a_n}.
\end{align}
By using (K2) in \Cref{def:planar_like} we obtain 
\begin{align}\label{eq:construction_3}
\frac{a_{n-1}}{a_n}=\frac{\left(1+o(1)\right)}{4n^2\gamma}.
\end{align}
Plugging in (\ref{eq:construction_3}) in (\ref{proof:linear_loop_0}) yields 
\begin{align}\label{proof:linear_loop_1}
\prob{Z_v=1}=(1+o(1))\frac{3}{4\gamma}.
\end{align}
Similarly, we estimate the number of multigraphs with loops at $u$ and $w$. We observe that all such multigraphs in which $u$ and $w$ are not adjacent can be construct as follows:
\begin{itemize}
	\item choose a vertex $x \in [2n]\setminus\{u, w\}$;
	\item choose $H \in \mathcal{K}(2(n-1), 3(n-1))$ and relabel the vertices with the labels $[2n]\setminus \{w, x\}$ such that we obtain a multigraph with a loop at $u$;
	\item
	choose an edge $e \neq uu$ in $H$ and 
	\item[]
	perform a \loopInsertion\ at edge $e$ with vertex pair $(w,x)$ .
\end{itemize}
In the above construction we have $(2n-2)$ possible choices for $x$ and by (\ref{proof:linear_loop_1}) the weight of all multigraphs that can be chosen for $H$ is
\begin{align*}
\left(1+o(1)\right)\frac{3}{4\gamma}a_{n-1}.
\end{align*} 
For fixed $x$ and $H$, the total weight of all multigraphs obtained by choosing an edge $e$ and performing the loop insertion is $\weight{H}(3n-4)/2$ due to \Cref{prop:loop_insertion}. 
On the other hand, if $u$ and $w$ are adjacent and there are loops at $u$ and $w$, then $u$ and $w$ form an own component with weight $1/4$. Combining these two cases we get 
\begin{align*}
\prob{Z_u=Z_w=1}&=\left((2n-2)\left(1+o(1)\right)\frac{3}{4\gamma}a_{n-1}\frac{3n-4}{2}+\frac{1}{4}a_{n-1}\right)/a_n
=(1+o(1))\frac{9n^2}{4\gamma}\frac{a_{n-1}}{a_n}.
\end{align*}
Finally, from this together with (\ref{eq:construction_3}) we obtain 
\begin{align*}
\prob{Z_u=Z_w=1}
= (1+o(1))\left(\frac{3}{4\gamma}\right)^2
=(1+o(1))\prob{Z_u=1}\prob{Z_w=1}.
\end{align*}
Hence, the statement follows by the second moment method. \qed

\proofof{cor:linear_loop}
The first statement follows directly from \Cref{lem:kernel_cubic,thm:linear_loop}. For the second statement we note that $\numberLoops{\kernel{\rest}}\leq \numberEdges{\kernel{\rest}}=O_p(1)$ by \Cref{thm:internal_structure}\ref{thm:internal_structure_d}. Thus, by combining that with the first statement we obtain \whp \[\numberLoops{\kernel{\LargestComponent}}=\numberLoops{\kernel{\planarRandomGraph}}-\numberLoops{\kernel{\rest}}=\Theta\left(sn^{-2/3}\right).\]
\qed
 
In order to prove \Cref{lem:bridge_number1,lem:bridge_number2}, we need the following two results, whose proofs are elementary and can be found in \Cref{sec:appendix_sum}.

\begin{claim}\label{lem:sum1}	
	Let $\alpha>2$ and $f$ be a function such that $f(n)=\omega(1)$ and $f(n)=o\left(n\right)$. 
	We set $I(n):=\{(j,k)\in \N^2 \mid j+k=n-1, j,k \geq f(n) \}$ and $m_i:=\left|\bridgeClass(2i,3i)\right|$ for $i\in \N$. Then we have 
\begin{align*}
\frac{1}{m_n}\sum_{(j,k)\in I(n)}\binom{2n-2}{2j}m_jm_k j k=(1+o(1))\frac{c}{2\gamma(\ce-2)}\cdot\frac{1}{f(n)^{\ce-2}n}.
\end{align*}
\end{claim}

\begin{claim}\label{lem:sum2}
	Let $\alpha>3$ and $f$ be a function with $f(n)=\omega(1)$, $f(n)=o(n)$, and $f(n)=o(n^{1/(\ce-2)})$. We set $I(n):=\left\{\left(j,k,l\right) \in \N^3 \mid j+k+l=n-2, j,k\geq f(n)\right\}$ and $m_i:=\left|\bridgeClass(2i,3i)\right|$ for $i \in \N$. Then we have
		\begin{align*}
\frac{324}{m_n}\sum_{(j,k,l)\in I(n)}\binom{2n-4}{2j,2k,2l}m_jm_km_lj k l^2= (1+o(1))\left(\frac{9c}{2\gamma\left(\ce-2\right)}\cdot \frac{1}{f(n)^{\alpha-2}n}\right)^2.
	\end{align*}
\end{claim}

\proofof{lem:bridge_number1} 
Each multigraph $\fixMultigraph \in \bridgeClass(2n,3n)$ with $\Bn{\fixMultigraph}{wx}\geq2f(n)+1$ can be constructed in the following way:
\begin{itemize}
	\item
	choose $j,k \in \N$ such that $j+k=n-1$ and $j,k\geq f(n)$;
	\item
	choose $2j$ labels from $[2n]\setminus \{w,x\}$ and denote them by $L$;
	\item
	choose $\fixMultigraph_1\in \bridgeClass(2j,3j)$ and relabel the vertices with $L$;
	\item
	choose $\fixMultigraph_2\in \bridgeClass(2k,3k)$ and relabel the vertices with $[2n]\setminus \left(L\cup\{w,x\}\right)$;
	\item
	choose edges $e_1 \in \edgeSet{\fixMultigraph_1}, e_2 \in \edgeSet{\fixMultigraph_2}$ and 
	\item[]
	 perform a \bridgeInsertion\ (cf. \Cref{def:bridge_insertion}) at edges $e_1$ and $e_2$ with vertices $w$ and $x$.
\end{itemize}
By \Cref{prop:subdividing} this construction gives multigraphs with a total weight of
\begin{align}\label{eq:construction_bridge}
\sum_{(j,k)\in I(n)}\binom{2n-2}{2j}m_jm_k\cdot (3j) \cdot (3k),
\end{align}
where $I(n):=\{(j,k)\in \N^2 \mid j+k=n-1, j,k \geq f(n) \}$ and $m_i:=\left|\bridgeClass(2i,3i)\right|$ for $i\in \N$. Hence, using \Cref{lem:sum1} yields
\begin{align*}
\prob{\Bn{\bridgeGraph}{wx}\geq 2f(n)+1}
=\frac{1}{m_n}\sum_{(j,k)\in I(n)}\binom{2n-2}{2j}m_jm_k \cdot (3j) \cdot (3k)=(1+o(1))\frac{9c}{2\gamma(\ce-2)}\cdot\frac{1}{f(n)^{\ce-2}n},
\end{align*}
as desired. \qed
 
\proofof{lem:bridge_number2}
We denote by $E_i$ the event that $\bn{w_{2i-1}w_{2i}}\geq 2f(n)+1$, where $i\in \{1,2\}$. In addition, let $E_3$ be the event that there is an edge in $\bridgeGraph$ with one endpoint in $\{w_1, w_2\}$ and the other in $\{w_3, w_4\}$. We start by estimating the probability $\prob{E_1\wedge E_2 \wedge E_3}$. We observe that if $E_3$ is true, then at least one of the four events $w_1w_3 \in \edgeSet{\bridgeGraph}, w_1w_4 \in \edgeSet{\bridgeGraph}, w_2w_3 \in \edgeSet{\bridgeGraph}$, $w_2w_4 \in \edgeSet{\bridgeGraph}$ is true. Thus, by symmetry reasons we obtain
\begin{align}\label{eq:bridge1}
\prob{E_1 \wedge E_2 \wedge E_3}& \leq
4\prob{E_1 \wedge E_2 \wedge \left(w_2w_3\in \edgeSet{\bridgeGraph}\right)}. 
\end{align}
Next, we note that the event $E_2$ implies $w_3w_4\in \edgeSet{\bridgeGraph}$. Using that in (\ref{eq:bridge1}) yields
\begin{align*}
\prob{E_1 \wedge E_2 \wedge E_3}& \leq4\prob{E_1 \wedge \left(w_2w_3, w_3w_4 \in \edgeSet{\bridgeGraph}\right)}
\\
&=4\prob{E_1}\cdot \condprob{w_2w_3, w_3w_4 \in \edgeSet{\bridgeGraph}}{E_1}.
\end{align*}
Using \Cref{lem:bridge_number1} for an estimate of $\prob{E_1}$ and the fact that $\condprob{w_2w_3, w_3w_4 \in \edgeSet{\bridgeGraph}}{E_1}=\Theta\left(n^{-2}\right)$ we get 
\begin{align*}
\prob{E_1 \wedge E_2 \wedge E_3}&=O(1)\frac{1}{f(n)^{\ce-2}n}\frac{1}{n^2}=o(1)\frac{1}{f(n)^{2\ce-4}n^2}.
\end{align*}
Next, we estimate the probability $\prob{E_1 \wedge E_2 \wedge \neg E_3}$, where $\neg E_3$ is the event that $E_3$ is not true. We observe that we can enumerate all multigraphs $\fixMultigraph \in \bridgeClass(2n,3n)$ satisfying $\Bn{H}{w_1w_2}, \Bn{H}{w_3w_4}\geq 2f(n)+1$, and $w_1w_3, w_1w_4, w_2w_3, w_2w_4 \notin \edgeSet{H}$ by the following construction:
\begin{itemize}
	\item 
	choose $j, k, l \in \N$ with $j+k+l=n-2$ and $j,k\geq f(n)$;
	\item
	choose a partition $J~\dot\cup~K~\dot\cup~L=[2n]\setminus\{w_1, \ldots, w_4\}$ with $|J|=2j, |K|=2k, |L|=2l$;
	\item
	choose $\bridgeGraph_1 \in \bridgeClass(2j,3j)$, $\bridgeGraph_2 \in \bridgeClass(2k,3k)$, $\bridgeGraph_3 \in \bridgeClass(2l,3l)$ and relabel these multigraphs with the labels $J, K$ and $L$, respectively;
	\item
	choose $i\in \{1,2\}$ and $i'\in\{3,4\}$;
	\item
	choose edges $e_1 \in \edgeSet{\bridgeGraph_1}$, $e_2 \in \edgeSet{\bridgeGraph_2}$ and $e_3\neq e_4 \in \edgeSet{\bridgeGraph_3}$ and 
	\item[]
	perform bridge insertions (cf. \Cref{def:bridge_insertion}) with vertices $w_i, w_{3-i}$ and $w_{i'}, w_{7-i'}$ at the edges $e_1, e_3$ and $e_2, e_4$, respectively.
\end{itemize}
Hence, letting $I(n):=\left\{\left(j,k,l\right) \in \N^3 \mid j+k+l=n-2, j,k\geq f(n)\right\}$, $m_i:=\left|\bridgeClass(2i,3i)\right|$ for $i \in \N$, we obtain 
\begin{align*}
\prob{E_1 \wedge E_2 \wedge \neg E_3}
&= \frac{1}{m_n}\sum_{(j,k,l)\in I(n)}\binom{2n-4}{2j,2k,2l}m_jm_km_l\cdot 4 \cdot 3j \cdot 3k \cdot \left(3l\right)^2\\
&= (1+o(1))\left(\frac{9c}{2\gamma\left(\ce-2\right)}\cdot \frac{1}{f(n)^{\alpha-2}n}\right)^2,
\end{align*}
where the last equality follows from \Cref{lem:sum2}.
\qed

\proofof{lem:bridge_number3}
We will use the second moment method. To that end, let $Z_p$ be the indicator random variable that an unordered pair $p=\left\{w,x\right\}$ of vertices satisfies $\bn{wx}\geq 2f(n)+1$ and $Z=\sum_{p}Z_p$. Now let $w, x, y \in [2n]$ be distinct vertices and $p=\left\{w,x\right\}$ and $q=\left\{w,y\right\}$ be unordered pairs. Then we have
\begin{align*}
\expec{Z_pZ_q}&=\prob{Z_p=1}\condprob{Z_q=1}{Z_p=1}\leq \prob{Z_p=1}\condprob{wy\in \edgeSet{\bridgeGraph}}{Z_p=1}
\\
&=O(1)\prob{Z_p=1}/n.
\end{align*}
Hence, by using \Cref{lem:bridge_number1,lem:bridge_number2} we obtain
\begin{align*}
\variance{Z}&\leq \sum_{p}\expec{Z_p}+\sum_{\left|p \cap q\right|=1}\expec{Z_pZ_q}+\sum_{p \cap q=\emptyset}\covariance{Z_p}{Z_q}
\\
&\leq \expec{Z}+O(1)\expec{Z}+o(1)\expec{Z}^2,
\end{align*}
which implies the statement due to the second moment method.
\qed

\proofof{lem:bridge_number5}
We construct an auxiliary directed graph $\widetilde{H}$ as follows: First we orient each bridge $wx$ in $H$ from $w$ to $x$ such that $\numberVertices{H_w}>\numberVertices{H_x}$, where $H_w$ and $H_x$ are the components of $H-wx$ containing $w$ and $x$, respectively. Then we contract each block to a single vertex and denote by $\widetilde{H}$ the obtained directed graph. We note that the edges of $\widetilde{H}$ are exactly the bridges of $H$ and that each vertex of $\widetilde{H}$ corresponds to either a vertex or a block in $H$. Furthermore, $\widetilde{H}$ contains no cycle and therefore, has a vertex $v$ with no incoming edge. If $v$ is a vertex in $H$, then $v$ has three neighbours $w_1, w_2, w_3$ in $H$ and the edges $vw_1, vw_2, vw_3$ are all bridges and oriented away from $v$. Hence, we obtain 
\begin{align*}
2n=\numberVertices{H}=1+\Bn{H}{vw_1}+\Bn{H}{vw_2}+\Bn{H}{vw_3}\leq 1+3\cdot (2n-2)/3=2n-1,
\end{align*}
where we used in the last inequality the assumption that each \bridgeNumber\ is at most $(2n-2)/3$. This implies that $v$ corresponds to a block in $H$. As all bridges are oriented away from $v$, this block is \largeBlock. \qed

\proofof{lem:bridge_number4}
We note that $\prob{\bn{wx}=1}\leq \prob{wx\in \edgeSet{\bridgeGraph}}=\Theta(1/n)$. In addition, by (\ref{eq:formula_bridge}) and (\ref{eq:construction_bridge}) we have that uniformly over all $1\leq j \leq (n-1)/2$ 
\begin{align*}
\prob{\bn{wx}=2j+1}=\Theta(1)j^{-\ce+1}/n.
\end{align*}
Thus, we obtain
\begin{align*}
\expec{\Bn{\bridgeGraph}{wx}^\exponent}&=\prob{\bn{wx}=1}+\sum_{j=1}^{\lfloor(n-1)/2\rfloor}(2j+1)^\exponent \ \prob{\bn{wx}=2j+1}
\\
&=O\left(1/n\right)+\Theta(1/n)\sum_{j=1}^{\lfloor(n-1)/2\rfloor}j^{-\ce+\exponent+1}=\Theta(1/n),
\end{align*}
since $-\ce+\exponent+1<-1$. This completes the proof. \qed

\proofof{lem:bridge_number6}
Let $h(n)=\omega(1)$ be a function and $1< \exponent \leq \ce-2$ a constant. Using \Cref{lem:bridge_number4} we have that \whp\ $\bridgeGraph$ satisfies $\sum_{w\neq x}\Bn{\bridgeGraph}{wx}^\exponent\leq nh(n)^{\exponent-1}$. Moreover, due to \Cref{lem:bridge_number1} \whp\ all pairs of vertices $w\neq x \in [2n]$ satisfy $\Bn{\bridgeGraph}{wx}\leq (2n-2)/3$. Hence, $\bridgeGraph$ contains \whp\ a \largeBlock\ block by \Cref{lem:bridge_number5}. Now it suffices to show that each $\fixMultigraph \in \bridgeClass(2n,3n)$ with a \largeBlock\ block and satisfying $\sum_{w\neq x}\Bn{H}{wx}^\exponent\leq nh(n)^{\exponent-1}$ has a block with at least $n/h(n)$ many vertices. To that end, let $\block$ be a \largeBlock\ block in $\fixMultigraph$ and $\bridge_1, \ldots, \bridge_r$ the bridges that share a vertex with $\block$. For $i \in [r]$ we denote by $C_i$ the component of $\fixMultigraph-\bridge_i$ not containing $\block$. We note that $\numberVertices{\block}+\sum_{i=1}^{r}\numberVertices{C_i}=2n$ and $\numberVertices{C_i}=\Bn{H}{e_i}$. Hence, we get $\sum_{i=1}^{r}\numberVertices{C_i}^\exponent\leq nh(n)^{\exponent-1}$. Now we assume that $\numberVertices{\block}<n/h(n)$. Then by using Jensen's inequality for the convex function $x\mapsto x^\exponent$ and the simple fact $r\leq \numberVertices{\block}$, we obtain that for $n$ large enough 
\begin{align*}
nh(n)^{\exponent-1}&\geq \sum_{i=1}^{r}\numberVertices{C_i}^\exponent\geq \left(\sum_{i=1}^{r}\numberVertices{C_i}\right)^\exponent/r^{\exponent-1}
\\
&\geq\left(2n-n/h(n)\right)^\exponent/\left(n/h(n)\right)^{\exponent-1}>nh(n)^{\exponent-1},
\end{align*}
a contradiction. Hence, we obtain $\numberVertices{B}\geq n/h(n)$, which completes the proof. \qed

\begin{remark}
	In the proof of \Cref{lem:bridge_number6} we actually showed the following stronger statement. If there are $\Theta_p(n)$ many vertices outside the largest block of $\bridgeGraph$, then $\bridgeGraph$ contains a block which shares a vertex with $\Theta_p(n)$ many bridges. We note that in most of our applications this assumption is satisfied, for example, when \whp\ $\bridgeGraph$ contains linearly many loops (see \Cref{thm:linear_loop}).
\end{remark}

\proofof{lem:bridge_number7}
We note that for every pair of two different blocks $B\neq B'$ there is a bridge $e$ such that $B$ and $B'$ lie in different components of $\bridgeGraph-\bridge$. Hence, we obtain by \Cref{lem:bridge_number1} that $\blockOrder{i}{\bridgeGraph}\leq\max_{w\neq x}\Bn{\bridgeGraph}{wx}=O_p\left(n^{1/\left(\ce-2\right)}\right)$. Next, we show $\blockOrder{i}{\bridgeGraph}=\Omega_p\left(n^{1/\left(\ce-2\right)}\right)$ by induction on $i$. To that end, let $L>0$ be a constant and $L'=L'(n)=Ln^{1/(\ce-2)}$. In addition, let $\bridgeMarkedClass(n)$ be the class of pairs $(\fixMultigraph, \bridge)$, where $\fixMultigraph \in \bridgeClass(2n,3n)$ and $\Bn{\fixMultigraph}{\bridge}\geq L'(n)$. Moreover, let $\left(\bridgeGraph', \bridge'\right)$ be a pair chosen uniformly at random from $\bridgeMarkedClass(n)$. 
Next, we show that the distributions of $\bridgeGraph$ and $\bridgeGraph'$ are \lq similar\rq. More precisely, let $\property$ be some graph property. We claim that 
\begin{align}\label{eq:claim}
\left(\forall L>0: \bridgeGraph' \in \property~ \text{whp}\right) \quad \Longrightarrow \quad\left(\bridgeGraph \in \property~ \text{whp}\right).
\end{align}
To prove it, we define $\Nb{\fixMultigraph}{j}:=\left|\left\{\bridge \in \edgeSet{\fixMultigraph}\mid \Bn{\fixMultigraph}{\bridge}\geq j\right\}\right|$ for a multigraph $\fixMultigraph$ and $j\in \N$. With this notation we obtain by \Cref{lem:bridge_number1}
\begin{align}\label{eq:second_largest1}
\left|\bridgeMarkedClass(n)\right|=\left|\bridgeClass(2n,3n)\right| \cdot \expec{\Nb{\bridgeGraph}{L'}}=\Theta(1)\left|\bridgeClass(2n,3n)\right|.	
\end{align} Next, we observe that for each $\fixMultigraph \in \bridgeClass(2n,3n)$ we have 
\begin{align}\label{eq:second_largest2}
\prob{\bridgeGraph'=\fixMultigraph}=\prob{\bridgeGraph=\fixMultigraph}\Nb{\fixMultigraph}{L'}\frac{\left|\bridgeClass(2n,3n)\right|}{\left|\bridgeMarkedClass(n)\right|}.
\end{align}
Combining (\ref{eq:second_largest1}) and (\ref{eq:second_largest2}) yields
\begin{align*}
\prob{\bridgeGraph \notin \property}&\leq \prob{\Nb{\bridgeGraph}{L'}=0}+\prob{\bridgeGraph \notin \property \wedge \Nb{\bridgeGraph}{L'}\geq 1}
\\
&= \prob{\Nb{\bridgeGraph}{L'}=0}+O(1)\prob{\bridgeGraph'\notin \property}\\
&= \prob{\Nb{\bridgeGraph}{L'}=0}+o(1),
\end{align*}
where we assumed that \whp\ $\bridgeGraph'\in \property$. Finally, we observe that for each $\delta>0$, we can choose $L>0$ such that $\prob{\Nb{\bridgeGraph}{L'}=0}\leq \delta$ by \Cref{lem:bridge_number1}. This shows (\ref{eq:claim}).

Next, we prove $\blockOrder{i}{\bridgeGraph'}=\Omega_p\left(n^{1/\left(\ce-2\right)}\right)$, which implies $\blockOrder{i}{\bridgeGraph}=\Omega_p\left(n^{1/\left(\ce-2\right)}\right)$ by (\ref{eq:claim}). We note that we can enumerate all pairs $(\fixMultigraph, \bridge)\in \bridgeMarkedClass(n)$ as follows:
\begin{itemize}
	\item choose an unordered pair $\left\{w,x\right\}$ with $w\neq x \in [2n]$ and set $e=wx$;
	\item choose $j,k \in \N$ such that $j+k=n-1$ and $j,k\geq L'/2$;
	\item choose a partition $J~\dot\cup~K=[2n]\setminus\{w, x\}$ with $|J|=2j$ and $|K|=2k$;
	\item choose $H_1\in \bridgeClass(2j,3j)$ and $H_2\in \bridgeClass(2k,3k)$ and relabel the vertices with $J$ and $K$, respectively;
	\item choose edges $e_1\in \edgeSet{H_1}$ and $e_2\in \edgeSet{H_2}$ and 
	\item[]
	perform a \bridgeInsertion\ (cf. \Cref{def:bridge_insertion}) at edges $e_1$ and $e_2$ with vertices $w$ and $x$.
\end{itemize}
Next, we will use \Cref{lem:splitting}, which we only formulated for graph classes, but it is straightforward that it is also true for classes of graphs where one edge is marked. Let $\cl(n)=\bridgeMarkedClass(n)$ and define the function $\func$ for a pair $(\fixMultigraph, \bridge) \in \cl:=\bigcup_{n\in\N} \cl(n)$ by $\func(\fixMultigraph, \bridge):=\left(e, \Bn{\fixMultigraph}{\bridge}\right)$. Let $\left(\bridge_n, j_n\right)_{n\in \N}$ be a sequence feasible for $\left(\cl, \func\right)$ and $(\condRandomGraph_n, \bridge_n)$ be a pair chosen uniformly at random from all elements in $\cl(n)$ that evaluate to $\left(\bridge_n, j_n\right)$ under $\func$. Now let $k_n=n-1-j_n$ and let $C_1$ and $C_2$ be the two graphs obtained by reversing the bridge insertion of $e_n$ in $\condRandomGraph_n$. By the above construction, $C_1$ and $C_2$ are distributed like $\bridgeGraph(2j_n, 3j_n)$ and $\bridgeGraph(2k_n, 3k_n)$, respectively. Without loss of generality we may assume $j_n\geq (n-1)/2$. Now by \Cref{lem:bridge_number6} and induction hypothesis we obtain $\blockOrder{i-1}{C_1}=\Omega_p\left(n^{1/(\ce-2)}\right)$ and $\blockOrder{1}{C_2}=\Omega_p(k_n)=\Omega_p\left(n^{1/(\ce-2)}\right)$. Hence, we have $\blockOrder{i}{\condRandomGraph_n}=\Omega_p\left(n^{1/(\ce-2)}\right)$. Now we get $\blockOrder{i}{\bridgeGraph'}=\Omega_p\left(n^{1/(\ce-2)}\right)$ by \Cref{lem:splitting}. This together with (\ref{eq:claim}) implies that $\blockOrder{i}{\bridgeGraph}=\Omega_p\left(n^{1/(\ce-2)}\right)$.
\qed

\section{Discussion}\label{sec:discussion}
\Cref{thm:main1}\ref{thm:main1_c} raises the question about the precise asymptotic order of the circumference $\longestCycle{\LargestComponent}$ of the largest component $L_1$ of a uniform random planar graph in the weakly supercritical regime. The reason why we provided only a lower and an upper bound for $\longestCycle{\LargestComponent}$ is partly because we could not determine the precise order of the circumference $\longestCycle{K}$ of a random {\em cubic planar} multigraph $K$. If there were a function $f=f(n)$ such that $\longestCycle{K}=\Theta_p(f)$, then our proof would imply that $\longestCycle{\LargestComponent}=\Omega_p\left(n^{1/3}f\left(sn^{-2/3}\right)\right)$ and $\longestCycle{\LargestComponent}=O_p\left(n^{1/3}f\left(sn^{-2/3}\right)\log n\right)$. That closes the gap up to a factor of $\log n$. Moreover, if $f(n)=\omega\left(n^{1/2}\right)$, then our methods lead even to $\longestCycle{\LargestComponent}=\Theta_p\left(n^{1/3}f\left(sn^{-2/3}\right)\right)$. 

We note that Robinson and Wormald \cite{hamiltonian} showed that \whp\ a random cubic (general, not necessarily planar) graph has a Hamiltonian cycle. We know that this is not the case for random cubic planar multigraphs, since the longest cycle misses all vertices which have a loop attached. By \Cref{thm:linear_loop} there are linearly many such vertices. Nevertheless, we believe that \whp\ there is a cycle of linear length.
\begin{conjecture}
Let $K=K(2n,3n)$ be a graph chosen uniformly at random from the class of all cubic planar multigraphs on vertex set $[2n]$. Then, we have $\longestCycle{K}=\Theta_p(n)$.
\end{conjecture}
If the above conjecture were true, we would immediately obtain the following result.
\begin{conjecture}
Let $\mathcal \planarRandomGraph$ be the class of planar graphs, $\planarRandomGraph=\planarRandomGraph(n,m) \ur \mathcal \planarRandomGraph(n,m)$, and $\LargestComponent=\largestComponent{\planarRandomGraph}$ the largest component of $\planarRandomGraph$. Assume $m=n/2+s$ for $s=s(n)=o\left(n\right)$ and $s^3n^{-2} \to \infty$. Then $\longestCycle{\LargestComponent}=\Theta_p\left(sn^{-1/3}\right)$.
\end{conjecture}

In \Cref{thm:block_structure_general} we determined the block structure for \pl\ classes of graphs which are addable and have a critical exponent $3<\ce<4$. We already pointed out that it should be straightforward to generalise these results to the case $\ce\geq 4$ (see \Cref{rem:block_structure}). However, we believe that this is not the case any more if $\ce<3$. Panagiotou and Steger \cite{ps} showed that in random $n$-vertex graphs, there is a drastic change in the block structure when the critical exponent $\ce$ is around 3. For example, they showed that a random planar graph on $n$ vertices (whose critical exponent is known to be $\ce=7/2$) has \whp\ a block of order linear in $n$, while the largest block of a random outerplanar graph or in a random series-parallel graph on $n$ vertices (whose critical exponent is known to be $\ce=5/2$) is of order $O\left(\log n\right)$. This leads to the following conjecture (see \Cref{lem:bridge_number6,lem:bridge_number7} for comparable results for the case $\ce>3$).

\begin{conjecture}
	Let $\bridgeClass$ be a \bridgeStable\ class of multigraphs with critical exponent $\ce<3$ and $\bridgeGraph=\bridgeGraph(2n,3n)\ur \bridgeClass(2n,3n)$. Then for each $i \in \N$, we have $\blockOrder{i}{\bridgeGraph}=O\left(\log n\right)$. 
\end{conjecture}

If this were true, we would obtain the following result (see \Cref{thm:block_structure_general} for a comparable statement for $\ce>3$).

\begin{conjecture}
	Let $\planarClass$ be a \pl\ class of graphs which is addable and has a critical exponent $\ce<3$ and $\planarRandomGraph=\planarRandomGraph(n,m)\ur \planarClass(n,m)$. Assume $m=n/2+s$ for $s=s(n)=o\left(n\right)$ and $s^3n^{-2} \to \infty$. Then for each $i\in \N$, we have
		$\blockOrder{i}{\planarRandomGraph}=\Omega_p\left(n^{1/3}\right)$ and $\blockOrder{i}{\planarRandomGraph}=O_p\left(n^{1/3}\left(\log n\right)^2\right)$.
\end{conjecture}

\section*{Acknowledgement}
The authors thank the anonymous referees for many helpful remarks to
improve the presentation of this paper.

\bibliographystyle{plain}
\bibliography{kang-missethan-cycles}

\appendix
\newpage
\section{Proofs of \Cref{main1} and \Cref{prop:polya_main}} \label{sec:polya_proofs}

Before providing the proof of \Cref{main1} we briefly illustrate the proof strategy. As for an upper bound for $\minimum$, we will find a \lq small\rq\ function $g_1=g_1(n)$ such that $\prob{\minimum\geq g_1}=o(1)$. For $1\leq \ind \leq f$ we denote by $A_\ind$ the event that $X_\ind\geq g_1$ and observe that $\minimum\geq g_1$ if and only if $A_\ind$ is true for {\em all} $1\leq \ind \leq f$. Moreover, we intuitively expect that 
\begin{align}\label{eq:negative_dependent}
\prob{\bigwedge_{\ind=1}^{f}A_\ind}\leq \prod_{\ind=1}^{f}\prob{A_\ind},
\end{align}
because given that $X_1, \ldots, X_{\ind-1}$ are \lq large\rq\ (for some $1\leq \ind \leq f)$, the probability that $X_\ind$ is also \lq large\rq\ might decrease. 
If (\ref{eq:negative_dependent}) holds (see \Cref{prop:polya_3}), then we obtain
\begin{align*}
\prob{\minimum\geq g_1}=\prob{\bigwedge_{\ind=1}^{f}A_\indSecond}\leq \prod_{\ind=1}^{f}\prob{A_\ind} = \prod_{\ind=1}^{f} \prob{X_\ind \geq g_1}.
\end{align*}

In order to derive a lower bound for $\minimum$, we will determine a \lq large\rq\ function $g_2=g_2(n)$ such that $\prob{\minimum\leq g_2}=o(1)$. To that end, we observe that if $\minimum \leq g_2$, then there is {\em at least one} $1\leq \ind \leq f$ such that $X_\ind \leq g_2$. Thus, we obtain
\begin{align*}
\prob{\minimum\leq g_2}\leq \sum_{\ind=1}^{f}\prob{X_\ind \leq g_2}.
\end{align*}
Therefore, in both cases, it is enough to find good bounds for $\prob{X_\ind \geq g_1}$ and $\prob{X_\ind \leq g_2}$. Such bounds are obtained in \Cref{prop:polya_4}. 

In order to make the aforementioned idea more precise, we need two known facts about the P\'olya urn model. The first one is about the marginal distribution of $X_i$ and will be our starting point for deducing bounds on $\prob{X_\ind \geq g_1}$ and $\prob{X_\ind \leq g_2}$. 

\begin{prop}[{\cite[Theorem 3.1]{mahmoud}}]\label{prop:polya}
	Let $N,k\in \mathbb N$ be given. For $\ind \in \left[\nc\right]$ and $\val \in \{0, \ldots, \nd\}$, we have
	\begin{equation*}
	\prob{X_\ind=\val} \, =\, \frac{\binom{\nd+\nc-\val-2}{\nc-2}}{\binom{\nd+\nc-1}{\nc-1}},
	\end{equation*}
	and in particular, $\prob{X_\ind=\val} \le \prob{X_\ind=0} \le \frac{\nc}{\nd+\nc}$.
\end{prop}

\begin{prop}[{\cite{block1982, KARLIN}}]\label{prop:polya_3}
	Let $N\in \mathbb N$ be given. For $\ind \in \left[\nc\right]$ and $\val_1, \ldots, \val_\ind \in \N_0$, we have
	\begin{align}
	\prob{\bigwedge_{\indSecond=1}^\ind \left(X_\indSecond \geq \val_\indSecond\right)}&\leq \prod_{\indSecond=1}^{\ind}\prob{X_\indSecond \geq \val_\indSecond} \label{eq:polya1}\\
	\text{ and }~~~~~\prob{\bigwedge_{\indSecond=1}^\ind \left(X_\indSecond \leq \val_\indSecond\right)}&\leq \prod_{\indSecond=1}^{\ind}\prob{X_\indSecond \leq \val_\indSecond}.~~~~~~~~~~~~~~~~~\label{eq:polya2}
	\end{align}
\end{prop}

We note that a more general version of \Cref{prop:polya_3} was proven in \cite[Example 5.5]{block1982} by using a fact from \cite[(1.8)]{KARLIN}. A random vector $(X_1, \ldots, X_\nc)$ satisfying (\ref{eq:polya1}) and (\ref{eq:polya2}) is also called {\em negatively dependent} (see e.g. \cite{negative_dependence} for details). Next, we derive some bounds for $\prob{X_\ind\leq \val}$ and $\prob{X_\ind\geq \val}$ by using \Cref{prop:polya}.

\begin{prop}\label{prop:polya_4}
	Let $N,k\in \mathbb N$ be given.
	\begin{enumerate}	
		\item \label{prop:polya_4b}	
		For $\ind \in \left[\nc\right]$ and $\val \in \{0, \ldots, \nd\}$, we have
		\begin{align*}
		\prob{X_\ind\leq \val}&\leq (\val+1) \frac{\nc}{\nd+\nc}\\
		\text{ and }~~~~~\prob{X_\ind\geq \val}&\leq 2\exp\left(-\frac{(\nc-2)}{\nd+\nc}\val\right).~~~~~~~~~~~~~~
		\end{align*}
		\item \label{prop:polya_4d}
		If in addition $\val\leq \frac{\nd}{2}$, then we have
		\begin{equation*}
		\prob{X_\ind\geq \val}\leq 1-\frac{(\nc-1)}{\nd+\nc} \val \exp\left(-\frac{2\nc}{\nd} \val\right).
		\end{equation*} 
		\item \label{prop:polya_4e}
		If in addition $\nd\geq 8\nc$ and $\val\leq \frac{\nd}{2}$, then we have
		\begin{equation*}
		\prob{X_\ind\leq \val}\leq \exp \left(-\frac{1}{64}\exp\left(-\frac{2\nc }{\nd}\val\right)\right).
		\end{equation*}
	\end{enumerate}
\end{prop}

\proofof{prop:polya_4}
Throughout the proof, we use \Cref{prop:polya} without stating explicitly. Then the first inequality in \ref{prop:polya_4b} follows by 
\begin{align*}
\prob{X_\ind\leq \val}=\sum_{\valSecond=0}^{\val}\prob{X_\ind=\valSecond}\leq\sum_{\valSecond=0}^{\val}\prob{X_\ind=0}\leq (\val+1)\frac{\nc}{\nd+\nc}.
\end{align*}
For the second inequality in \ref{prop:polya_4b} we may assume $\nc\geq 3$, since otherwise the statement is trivially fulfilled. We get by using $1+z\leq \exp(z)$ for $z\in \R$
\begin{align*}
\prob{X_\ind=\val}=\frac{\nc-1}{\nd+\nc-1}\prod_{a=2}^{\nc-1}\frac{\nd+\nc-\val-a}{\nd+\nc-a}
\leq\frac{\nc-1}{\nd+\nc-1}\left(\frac{\nd+\nc-\val-2}{\nd+\nc-2}\right)^{\nc-2}
\leq
\frac{\nc-1}{\nd+\nc-1}\exp\left(-\frac{(\nc-2)\val}{\nd+\nc}\right).
\end{align*}
Next, we observe that for $\valSecond\in \{0, \ldots ,\nd-1\}$
\[\frac{\prob{X_\ind=\valSecond+1}}{\prob{X_\ind=\valSecond}}=\frac{\nd-\valSecond}{\nd+\nc-\valSecond-2}\leq 1-\frac{\nc-2}{\nd+\nc-2}.\]
Hence, we obtain
\begin{align*}
\prob{X_\ind\geq \val}&=\sum_{\valSecond=\val}^{\nd}\prob{X_\ind=\valSecond}
\\
&\leq \prob{X_\ind=\val}\sum_{\valSecond=\val}^{\nd}\left(1-\frac{\nc-2}{\nd+\nc-2}\right)^{\valSecond-\val}
\\
&\leq \frac{\nc-1}{\nd+\nc-1}\exp\left(-\frac{(\nc-2)\val}{\nd+\nc}\right) \frac{\nd+\nc-2}{\nc-2}
\\
&\leq 2\exp\left(-\frac{(\nc-2)\val}{\nd+\nc}\right),
\end{align*}
which proves \ref{prop:polya_4b}. Next, we assume $\val \leq \frac{\nd}{2}$ and show \ref{prop:polya_4d}. To that end, we use $1-z\geq \exp(-2z)$ for $z\in \left[0,\frac{1}{2}\right]$ to obtain
\begin{align}\label{proof:prop_polya_4d}
\prob{X_\ind=\val}=\frac{\nc-1}{\nd+\nc-1}\prod_{a=2}^{\nc-1}\frac{\nd+\nc-\val-a}{\nd+\nc-a} 
\geq\frac{\nc-1}{\nd+\nc}\left(\frac{\nd-\val}{\nd}\right)^{\nc}
\geq \frac{\nc-1}{\nd+\nc}\exp\left(-\frac{2\nc \val}{\nd}\right).
\end{align}
Using that yields
\begin{align*}
\prob{X_i< \val}=\sum_{\valSecond=0}^{\val-1} \prob{X_\ind=\valSecond }\geq \val \prob{X_\ind=\val}
\geq \val\frac{\nc-1}{\nd+\nc}\exp\left(-\frac{2\nc \val}{\nd}\right).
\end{align*}
This shows \ref{prop:polya_4d}. Finally, we assume $\nd\geq 8\nc$ and $\val \leq \frac{\nd}{2}$. Then, we have for $\valSecond \leq \frac{3\nd}{4}$
\[\frac{\prob{X_\ind=\valSecond+1}}{\prob{X_\ind=\valSecond}}=\frac{\nd-\valSecond}{\nd+\nc-\valSecond-2}\geq 1-\frac{\nc}{\nd-\valSecond}\geq 1-\frac{4\nc}{\nd}.\]
Thus, for $\val\leq \frac{\nd}{2}$ we obtain by using (\ref{proof:prop_polya_4d})
\begin{align*}
\prob{X_\ind>\val}&\geq\sum_{\valSecond=\val+1}^{\frac{3\nd}{4}}\prob{X_\ind=\valSecond}\geq \prob{X_\ind=\val}\sum_{\valSecond=\val+1}^{\frac{3\nd}{4}}\left(1-\frac{4\nc}{\nd}\right)^{\valSecond-\val}
\\
&\geq \prob{X_\ind=\val} \frac{\nd}{8\nc}\left(1-\left(1-\frac{4\nc}{\nd}\right)^{\frac{\nd}{4}}\right)
\\
&\geq \prob{X_\ind=\val} \frac{\nd}{8\nc} \left(1-\exp\left(-\nc\right)\right)
\\
&\geq \frac{\nc-1}{\nd+\nc}\exp\left(-\frac{2\nc \val}{\nd}\right) \frac{\nd}{16\nc}
\\
&\geq
\frac{1}{64}\exp\left(-\frac{2\nc \val}{\nd}\right).
\end{align*}
Hence, we conclude the proof with
\begin{align*}
\prob{X_\ind\leq \val}\leq 1-\frac{1}{64}\exp\left(-\frac{2\nc \val}{\nd}\right)\leq \exp\left(-\frac{1}{64}\exp\left(-\frac{2\nc \val}{\nd}\right)\right).
\end{align*}\qed

\proofof{main1}
Throughout the proof, we let $n$ be large and $h=h(n)=\omega(1)$. To obtain the claimed bounds on $\minimum=\min_{1\leq \ind \leq f}X_\ind$, it suffices to show that 
\begin{enumerate}[label={(\alph*)}]
	\item\label{proof:main1_1}
	if $\nd=\omega\left(\nc f\right)$, then $\minimum=\Omega_p\left(\frac{\nd}{\nc f}\right)$;
	\item\label{proof:main1_2} 
	if $\nd=\Omega\left(\nc\right)$ and $f=\omega(1)$, then $\minimum=O_p\left(\frac{\nd}{\nc f}\right)$;
	\item\label{proof:main1_3}
	$\minimum=O_p\left(\frac{\nd}{\nc}\right)$.	
\end{enumerate}

To prove \ref{proof:main1_1}, it is enough to show $\prob{\minimum\leq \frac{\nd}{h\nc f}} =o(1)$ for any $h=o\left(\frac{\nd}{\nc f}\right)$. To this end, let $\val=\frac{\nd}{h\nc f}$. If $\minimum \leq x$, then $X_\ind \leq x$ for some $1\leq \ind \leq f$. Thus, by \Cref{prop:polya_4}\ref{prop:polya_4b} we obtain
\begin{align*}
\prob{\minimum\leq x}\leq \sum_{\ind=1}^{f}\prob{X_\ind\leq x}
\leq \sum_{\ind=1}^{f} (\val+1) \frac{\nc}{\nd+\nc} 
\leq 2f \val \frac{\nc}{\nd+\nc} 
=2f\frac{\nd}{h\nc f}\frac{\nc}{\nd+\nc}
= \Theta(1)\frac{1}{ h} =o(1),
\end{align*}
as desired.

To prove \ref{proof:main1_2}, it suffices to show $\prob{\minimum\geq \frac{h\nd}{\nc f}} =o(1)$ for any $h=o\left(f\right)$. Now let 
$x= \frac{h\nd}{\nc f}$ and for each $1\leq \ind \leq f$ we denote by $A_\ind$ the event that $X_\ind\geq x$. If $\minimum \geq x$, then $X_\ind \geq x$ for all $1\leq \ind \leq f$. Thus, by \Cref{prop:polya_3} we have
\begin{align}\label{proof:condition}
\prob{\minimum\geq x}&= \prob{\bigwedge_{i=1}^{f} A_i}\leq \prod_{\ind=1}^{f}\prob{A_\ind}. 
\end{align}
By \Cref{prop:polya_4}\ref{prop:polya_4d}, 
uniformly over all $1\leq \ind \leq f$, we have
\begin{align*}
\prob{A_\ind}
&\leq 1-\frac{(\nc-1) \val}{\nd+\nc}\exp\left(-\frac{2\nc \val}{\nd}\right)\\
&=1-\frac{\nc-1}{\nd+\nc}\frac{h\nd}{\nc f}\exp\left(-\frac{2\nc }{\nd}\frac{h\nd}{\nc f}\right)
\\
&= 1-\Theta(1)\frac{h}{f} 
\\
&\leq \exp\left(-\Theta(1)\frac{h}{f}\right).
\end{align*}
This together with (\ref{proof:condition}) yields the desired result
\begin{align*}
\prob{\minimum\geq x}\leq \exp\left(-\Theta(1)\frac{h}{f}\right)^f = \exp\left(-\Theta(1) h \right) = \exp\left(-\omega(1)\right) =o(1).
\end{align*}

Finally, \ref{proof:main1_3} follows by Markov's inequality and the fact that $\expec{\minimum}\leq\expec{X_1}=\frac{\nd}{\nc}$.

\bigskip
In order to derive the claimed bounds on $\maximum=\max_{1\leq \ind \leq f}X_\ind$, we prove the following assertions.
\begin{enumerate}[label={(\alph*)}]\addtocounter{enumi}{3}
	\item \label{proof:main1_5}
	If $\nd=\omega(\nc)$ and $f=\omega(1)$, then $\maximum=\Omega_p\left(\frac{\nd}{\nc}\left(1+\log f\right)\right)$;
	\item \label{proof:main1_6}
	if $\nd=\omega(\nc)$, then $\maximum=\Omega_p\left(\frac{\nd}{\nc}\right)$;
	\item \label{proof:main1_4}
	if $\nd=\omega(\nc)$, then $\maximum=O_p\left(\frac{\nd}{\nc}\left(1+\log f\right)\right)$;
	\item \label{proof:main1_7}
	if $\nd=O(\nc)$, then $\maximum=O_p(1+\log f)$.
\end{enumerate}

To show \ref{proof:main1_5}, we assume $\nd=\omega(\nc)$ and $f=\omega(1)$ and let $x=\frac{\nd}{h\nc}\left(1+\log f\right)$. If $\maximum \leq x$, then $X_\ind \leq x$ for all $1\leq \ind \leq f$. For each $1\leq \ind \leq f$ we denote by $B_\ind$ the event that $X_\ind\leq x$. Using \Cref{prop:polya_3} (for the first inequality) and \Cref{prop:polya_4}\ref{prop:polya_4e} (for the second inequality) yields for large $n$

\begin{align*}
\prob{\maximum \leq x}
&= \prob{\bigwedge_{i=1}^{f} B_i} 
\leq \prod_{\ind=1}^{f}\prob{B_\ind}
\\
&\leq \prod_{\ind=1}^{f}\exp \left(-\frac{1}{64}\exp\left(-\frac{2\nc }{\nd}\val\right)\right)
\\
&=\exp \left(-\frac{f}{64}\exp\left(-\frac{2\nc}{\nd}\frac{\nd}{h\nc}\left(1+\log f\right)\right)\right)
\\
&\leq \exp\left(-\frac{1}{64}\exp\left(\log f-\frac{4}{h}\log f\right)\right)
\\
&=\exp\left(-\frac{1}{64}\exp\left(\left(1-\frac{4}{h}\right)\log f\right)\right)
=o(1).
\end{align*}

In order to prove \ref{proof:main1_6}, we let $x=\frac{\nd}{h\nc}$ and use \Cref{prop:polya_4}\ref{prop:polya_4b} to get
\begin{align*}
\prob{\maximum \leq x}\leq \prob{X_1 \leq x}\leq (x+1) \frac{\nc}{\nd+\nc}=\left(\frac{\nd}{h\nc}+1\right)\frac{\nc}{\nd+\nc}\leq\frac{1}{h}+\frac{\nc}{\nd}=o\left(1\right),
\end{align*}
where we used in the last equality that $h=\omega(1)$ and $\nd=\omega(\nc)$.

To show \ref{proof:main1_4}, we assume $\nd=\omega(\nc)$ and let $x=\frac{h\nd}{\nc}\left(1+\log f\right)$. If $\maximum \leq x$, then $X_\ind \leq x$ for some $1\leq \ind \leq f$. Therefore, by \Cref{prop:polya_4}\ref{prop:polya_4b} we obtain
\begin{align*}
\prob{\maximum \geq x}&\leq \sum_{\ind=1}^{f}\prob{X_\ind \geq x}
\\
&\leq \sum_{\ind=1}^{f} 2\exp\left(-\frac{(\nc-2)}{\nd+\nc}\val\right)
\\
&=2f\exp\left(-\frac{(\nc-2)}{\nd+\nc}\frac{h\nd}{\nc}\left(1+\log f\right)\right)
\\
&=2\exp\left(\log f-\omega(1)\left(1+\log f\right)\right)=o(1).
\end{align*}

To prove \ref{proof:main1_7}, we assume $\nd=O(\nc)$ and let $\val=h\left(1+\log f\right)$. Using \Cref{prop:polya_4}\ref{prop:polya_4b} we get
\begin{align*}
\prob{\maximum \geq \val}&\leq \sum_{\ind=1}^{f}\prob{X_\ind\geq \val}
\\
&\leq 2f \exp\left(-\frac{\left(\nc-2\right)}{\nd+\nc} \val\right)
\\
&=2f \exp\left(-\frac{\left(\nc-2\right)}{\nd+\nc} h\left(1+\log f\right)\right)
\\
&=2\exp\left(\log f-\omega(1)\left(1+\log f\right)\right)=o(1).
\end{align*}
This concludes the proof. \qed

\proofof{prop:polya_main}
It follows directly from the first inequality of \Cref{prop:polya_4}\ref{prop:polya_4b}.
\qed
 
\section{Proofs of \Cref{lem:sum1,lem:sum2}}\label{sec:appendix_sum}
\proofof{lem:sum1}
 Using the formula for $m_i$ from (\ref{eq:formula_bridge}) yields
 \begin{align*}
 \frac{1}{m_n}\sum_{(j,k)\in I(n)}\binom{2n-2}{2j}m_jm_k j k=\left(1+o(1)\right)\frac{cn^{\ce-2}}{4\gamma}\sum_{(j,k)\in I(n)}j^{-\ce+1}k^{-\ce+1}.
 \end{align*}
 Thus, it suffices to show 
 \begin{align}\label{eq:appendix}
 \sum_{(j,k)\in I(n)}j^{-\ce+1}k^{-\ce+1}=\left(1+o(1)\right)\frac{2}{\ce-2}
 \cdot\frac{1}{f(n)^{\ce-2}n^{\ce-1}}.
 \end{align}
To that end, let $h=\omega(1)$ be such that $h(n)=\omega\left(f(n)\right)$ and $h(n)=o(n)$. We obtain
\begin{align*}
\sum_{(j,k)\in I(n)}j^{-\ce+1}k^{-\ce+1}&\leq 2\sum_{j=f(n)}^{h(n)}j^{-\ce+1}\left(n-1-j\right)^{-\ce+1}+2\sum_{j=h(n)+1}^{\lfloor n/2 \rfloor}j^{-\ce+1}\left(n-1-j\right)^{-\ce+1}
\\
&\leq(2+o(1))n^{-\ce+1}\int_{f(n)-1}^{h(n)}x^{-\ce+1}\,dx+\Theta(1)n^{-\ce+1}\int_{h(n)-1}^{\infty}x^{-\ce+1}\,dx
\\
&=(2+o(1))n^{-\ce+1}\frac{f(n)^{-\ce+2}}{\ce-2}.
\end{align*}
Similarly, we have that for $n \to \infty$
\begin{align*}
\sum_{(j,k)\in I(n)}j^{-\ce+1}k^{-\ce+1}&\geq 2\sum_{j=f(n)}^{h(n)}j^{-\ce+1}\left(n-1-j\right)^{-\ce+1}
\\
&\geq 2n^{-\ce+1}\int_{f(n)}^{h(n)}x^{-\ce+1}\,dx
\\
&=(2+o(1))n^{-\ce+1}\frac{f(n)^{-\ce+2}}{\ce-2}.
\end{align*}
This shows (\ref{eq:appendix}), which finishes the proof. \qed 

\proofof{lem:sum2}
We note that the asymptotic formula for $m_l$ in (\ref{eq:formula_bridge}) holds only for \lq large\rq\ $l$. Thus, we split the given sum into two parts, one for the terms where $l$ is \lq small\rq\ and the other for terms where $l$ is \lq big\rq. To make that more precise, we set $I_1(n):=\left\{\left(j,k,l\right) \in I(n) \mid l\leq n/2\right\}$. Due to (\ref{eq:formula_bridge}) there exists a constant $A>0$ such that $m_l\leq Al^{-\ce}\gamma^l (2l)!$ for all $l\in \N$. Combining that with the asymptotic formulas for $m_j, m_k$, and $m_n$ from (\ref{eq:formula_bridge}) yields
\begin{align}\label{eq:appendix2}
\frac{324}{m_n}\sum_{(j,k,l)\in I_1(n)}\binom{2n-4}{2j,2k,2l}m_jm_km_lj k l^2=O(1)n^{\ce-4}S_1,
\end{align}
where
\begin{align*}
S_1=\sum_{(j,k,l)\in I_1(n)}j^{-\ce+1}k^{-\ce+1}l^{-\ce+2}.
\end{align*}
We define
$J(l):=\left\{(j,k) \in \N^2 \mid j+k=n-2-l, j,k\geq f(n) \right\}$. Analogous to the proof of (\ref{eq:appendix}) we obtain
\begin{align*}
S_1&=\sum_{l=1}^{\lfloor n/2 \rfloor}l^{-\ce+2}\sum_{(j,k)\in J(l)}j^{-\ce+1}k^{-\ce+1}
\\
&=\sum_{l=1}^{\lfloor n/2 \rfloor}l^{-\ce+2} \sum_{(j,k)\in J(l)}\Theta(1)\left(n-l\right)^{-\ce+1}f(n)^{-\ce+2}
\\
&=\Theta(1)f(n)^{-\ce+2}n^{-\ce+1}\sum_{l=1}^{\lfloor n/2 \rfloor}l^{-\ce+2}
\\
&=\Theta(1)f(n)^{-\ce+2}n^{-\ce+1}
\\
&=o(1)f(n)^{-2\ce+4}n^{-\ce+2}.
\end{align*}
This in (\ref{eq:appendix2}) implies 
\begin{align}\label{eq:appendix4}
\frac{324}{m_n}\sum_{(j,k,l)\in I_1(n)}\binom{2n-4}{2j,2k,2l}m_jm_km_lj k l^2=o(1)f(n)^{-2\ce+4}n^{-2}.
\end{align}
Next, we consider those terms where $l$ is \lq big\rq, i.e. terms with indices in $\bar{I}_1(n):=I(n)\setminus I_1(n)$. Using (\ref{eq:formula_bridge}) we obtain 
\begin{align}\label{eq:appendix3}
\frac{324}{m_n}\sum_{(j,k,l)\in \bar{I}_1(n)}\binom{2n-4}{2j,2k,2l}m_jm_km_lj k l^2=\left(1+o(1)\right)\frac{81c^2}{4\gamma^2}n^{\ce-4}\bar{S}_1,
\end{align}
where
\begin{align*}
\bar{S}_1=\sum_{(j,k,l)\in \bar{I}_1(n)}j^{-\ce+1}k^{-\ce+1}l^{-\ce+2}.
\end{align*}
Now we partition $\bar{I}_1(n)$ into smaller parts. More precisely, 
let $h$ be a function such that $h(n)=\omega\left(f(n)\right)$, but $h(n)=o\left(n\right)$. We define 
\begin{align*}
I_2(n)&:=\left\{\left(j,k,l\right) \in \bar{I}_1(n) \mid j\geq h(n)\right\};
\\
I_3(n)&:=\left\{\left(j,k,l\right) \in \bar{I}_1(n) \mid j< h(n), k \geq h(n)\right\};
\\
I_4(n)&:=\left\{\left(j,k,l\right) \in \bar{I}_1(n) \mid j,k < h(n)\right\}.
\end{align*}
In addition, for $i\in \{2,3,4\}$ we set 
\begin{align*}
S_i:=\sum_{(j,k,l)\in I_i(n)}j^{-\ce+1}k^{-\ce+1}l^{-\ce+2}.
\end{align*}
Similarly as in the proof of (\ref{eq:appendix}), we have
\begin{align*}
S_2&\leq \left(n/2\right)^{-\ce+2}\sum_{j=h(n)}^{\infty}j^{-\ce+1}\sum_{k=f(n)}^{\infty}k^{-\ce+1}
\\
&=\Theta(1)n^{-\ce+2}h(n)^{-\ce+2}f(n)^{-\ce+2}
\\
&=o(1)f(n)^{-2\ce+4}n^{-\ce+2}.
\end{align*}
Next, we observe $S_3\leq S_2$ and
\begin{align*}
S_4&=(1+o(1))n^{-\ce+2}\sum_{j=f(n)}^{h(n)-1}j^{-\ce+1}\sum_{k=f(n)}^{h(n)-1}k^{-\ce+1}
\\
&=(1+o(1))n^{-\ce+2}\frac{f(n)^{-2\ce+4}}{\left(-\ce+2\right)^2}.
\end{align*}
Thus, we have 
\begin{align*}
\bar{S}_1=S_2+S_3+S_4=\frac{\left(1+o(1)\right)}{\left(\ce-2\right)^2}f(n)^{-2\ce+4}n^{-\ce+2}.
\end{align*}
Plugging this in (\ref{eq:appendix3}) yields
\begin{align*}
\frac{324}{m_n}\sum_{(j,k,l)\in \bar{I}_1(n)}\binom{2n-4}{2j,2k,2l}m_jm_km_lj k l^2=\left(1+o(1)\right)\frac{81c^2}{4\gamma^2\left(\ce-2\right)^2}f(n)^{-2\ce+4}n^{-2}.
\end{align*}
Combining that with (\ref{eq:appendix4}) yields the statement.
\qed

\end{document}